\documentclass[a4paper, 10pt, parskip=half]{scrartcl}

\usepackage[utf8]{inputenc}
\usepackage[T1]{fontenc}
\usepackage{lmodern}
\usepackage{csquotes}
\usepackage{amsmath}
\usepackage{amssymb}
\usepackage{amsthm}
\usepackage{amsfonts}
\usepackage{dsfont}
\usepackage{mathtools}
\usepackage{algorithm}
\usepackage{algorithmic}
\usepackage{marginnote}
\usepackage[dvipsnames]{xcolor}
\usepackage{hyperref}

\hypersetup{%
  colorlinks=true,
  linkcolor=black,
  citecolor=black,
  urlcolor=blue,
  pdftitle={A cross-diffusion system modelling rivaling gangs: global existence of bounded solutions and numerical evidence for nontrivial stabilization},
  pdfauthor={Mario Fuest, Shahin Heydari},
  pdfkeywords={},
  bookmarksopen=true,
}

\usepackage{authblk}
\usepackage{float}
\usepackage{subcaption}

\usepackage[numbers, sort&compress]{natbib}

\newcommand{\R}{\mathbb{R}}

\newcommand{\N}{\mathbb{N}}

\newcommand{\ur}[1]{\mathrm{#1}}
\newcommand{\ure}{\ur{e}}

\ifdefined\labelenumi%
  \renewcommand{\labelenumi}{(\roman{enumi})}
  
\fi

\newcommand{\eps}{\varepsilon}

\DeclareMathOperator{\id}{id}

\newcommand{\defs}{\coloneqq}
\newcommand{\sfed}{\eqqcolon}

\newcommand{\ra}{\rightarrow}

\newcommand{\nea}{\nearrow}

\newcommand{\ol}{\overline}

\newcommand{\ds}{\,\mathrm{d}s}

\newcommand{\ddt}{\frac{\mathrm{d}}{\mathrm{d}t}}

\newcommand{\hp}{\hphantom}
\newcommand{\pe}{\mathrel{\hp{=}}}

\newcommand{\tmax}{T_{\max}}

\newcommand{\intom}{\int_\Omega}

\newcommand{\intnstom}{\int_0^t \int_\Omega}

\newcommand{\inttnstom}{\int_{t_0}^t \intom}

\newcommand{\inttntom}{\int_{t_0}^T \intom}

\newcommand{\Ombar}{\ol \Omega}

\newcommand{\Ombarinf}{\Ombar \times [0, \infty)}

\newcommand{\leb}[2][\Omega]{\ensuremath{L^{#2}(#1)}}

\newcommand{\sob}[3][\Omega]{\ensuremath{W^{#2, #3}(#1)}}

\newcommand{\con}[2][\Ombar]{\ensuremath{C^{#2}(#1)}}

\newcommand{\tops}{\texorpdfstring}

\makeatletter
\renewenvironment{proof}[1][\proofname]{\par
  \pushQED{\qed}%
  \normalfont \topsep0\p@\relax
  \trivlist
  \item[\hskip\labelsep\scshape
  #1\@addpunct{.}]\ignorespaces
}{%
  \popQED\endtrivlist\@endpefalse
}
\makeatother

\newtheorem{base}{Base}[section]
\numberwithin{equation}{section}

\newtheorem{theorem}[base]{Theorem} \newtheorem*{theorem*}{Theorem}
\newtheorem{lemma}[base]{Lemma} \newtheorem*{lemma*}{Lemma}
\newtheorem{prop}[base]{Proposition} \newtheorem*{prop*}{Proposition}
\newtheorem{cor}[base]{Corollary} \newtheorem*{cor*}{Corollary}

\theoremstyle{definition}
\newtheorem{remark}[base]{Remark} \newtheorem*{remark*}{Remark}
\newtheorem{definition}[base]{Definition} \newtheorem*{definition*}{Definition}
 \newtheorem*{example*}{Example}
 \newtheorem*{cond*}{Condition}
 \newtheorem*{algo*}{Algorithm}

\numberwithin{algorithm}{section}

\newif\ifhyperconst             

\makeatletter
\newcounter{globalconst}
\newcommand{\newgc}[2][]{%
\refstepcounter{globalconst}%
\ltx@label{gc:\thesection:#2}%
\ifhyperconst
  \hyperref[gc:\thesection:#2]{C}_{\ref{gc:\thesection:#2}#1}
\else
  C_{\begin{NoHyper}\ref{gc:\thesection:#2}\end{NoHyper}}
\fi}
\ifhyperconst
  \newcommand{\gc}[2][]{\hyperref[gc:\thesection:#2]{C}_{\ref{gc:\thesection:#2}#1}}
\else
  \newcommand{\gc}[2][]{C_{\begin{NoHyper}\ref{gc:\thesection:#2}\end{NoHyper}}}
\fi

\newcounter{localconst}[base]
\newcommand{\newlc}[2][]{%
\refstepcounter{localconst}%
\ltx@label{lc:\thesection:\arabic{base}:#2}%
\ifhyperconst
  \hyperref[lc:\thesection:\arabic{base}:#2]{c}_{\ref{lc:\thesection:\arabic{base}:#2}#1}
\else
  c_{\begin{NoHyper}\ref{lc:\thesection:\arabic{base}:#2}\end{NoHyper}}
\fi}
\ifhyperconst
  \newcommand{\lc}[2][]{\hyperref[lc:\thesection:\arabic{base}:#2]{c}_{\ref{lc:\thesection:\arabic{base}:#2}#1}}
\else
  \newcommand{\lc}[2][]{c_{\begin{NoHyper}\ref{lc:\thesection:\arabic{base}:#2}\end{NoHyper}}}
\fi
\makeatother

\setkomafont{title}{\normalfont\Large}
\title{A cross-diffusion system modelling rivaling gangs: global existence of bounded solutions and FCT stabilization for numerical simulation}
\usepackage{authblk}

\author{Mario Fuest\footnote{e-mail: fuest@ifam.uni-hannover.de, corresponding author}}
\affil{Leibniz Universität Hannover, \\ Institut für Angewandte Mathematik, \\ Welfengarten 1, 30167 Hannover, Germany}
\author{Shahin Heydari\footnote{e-mail: heydari@karlin.mff.cuni.cz}}
\affil{Charles University, \\ Faculty of Mathematics and Physics, \\ Sokolovska 83, 18675 Praha 8, Czech Republic}
\date{}

\begin{document}

\maketitle

\KOMAoptions{abstract=true}
\begin{abstract}
\noindent
  For the gang territoriality model
  \begin{align*}
    \begin{cases}
      u_t = D_u \Delta u + \chi_u \nabla \cdot (u \nabla w), \\
      v_t = D_v \Delta v + \chi_v \nabla \cdot (v \nabla z), \\
      w_t = -w + \frac{v}{1+v}, \\
      z_t = -z + \frac{u}{1+u},
    \end{cases}
  \end{align*}
  where $u$ and $v$ denote the densities of two rivaling gangs which spray graffiti (with densities $z$ and $w$, respectively) and partially move away from the other gang's graffiti,
  we construct global, bounded classical solutions.
  By making use of quantitative global estimates, we prove that these solutions converge to homogeneous steady states if $\|u_0\|_{L^\infty(\Omega)}$ and $\|v_0\|_{L^\infty(\Omega)}$ are sufficiently small.\\[2pt]
  Moreover, we perform numerical experiments which show that for different choices of parameters, the system may become diffusion- or convection-dominated, where in the former case the solutions converge toward constant steady states while in the later case nontrivial asymptotic behavior such as segregation is observed. In order to perform these experiments, we apply a nonlinear finite element flux-corrected transport method (FEM-FCT) which is positivity-preserving. Then, we treat the nonlinearities in both the system and the proposed nonlinear scheme simultaneously using fixed-point iteration.
  \\[5pt]
 \textbf{MSC Classification (2020):} {35K55 (primary), 35A01, 35B40, 35Q91, 65M22, 65M60, 91D10}
\end{abstract}

\section{Introduction}
Graffiti, the artful wall writing usually on public property and in open view,
is sprayed by urban gang members not only to express themselves, convey their attitudes and communicate with fellow gang members
but also to mark their area of control, i.e., the gang's territory \cite{LeyCybriwskyUrbanGraffitiTerritorial1974, BrownGraffitiIdentityDelinquent1978}.
In order to describe the interaction of two rivaling gangs attempting to establish or defend territories by spraying intimidating graffiti,
 \cite{AlsenafiBarbaroConvectionDiffusionModel2018} introduces the model
\begin{align}\label{prob}
  \begin{cases}
    u_t = D_u \Delta u + \chi_u \nabla \cdot (u \nabla w)                     & \text{in $\Omega \times (0, \infty)$}, \\
    v_t = D_v \Delta v + \chi_v \nabla \cdot (v \nabla z)                     & \text{in $\Omega \times (0, \infty)$}, \\
    w_t = -w + f(v)                                                           & \text{in $\Omega \times (0, \infty)$}, \\
    z_t = -z + g(u)                                                           & \text{in $\Omega \times (0, \infty)$}, \\
    D_u \partial_\nu u + \chi_u u \partial_\nu w = D_v \partial_\nu v + \chi_v v \partial_\nu z = 0 & \text{on $\partial \Omega \times (0, \infty)$}, \\
    (u, v, w, z)(\cdot, 0) = (u_0, v_0, w_0, z_0)                             & \text{in $\Omega$},
  \end{cases}
\end{align}
where positive parameters $D_u, D_v, \chi_u, \chi_v$, suitable smooth initial data $u_0, v_0, w_0, z_0$ and spray rates $f, g$ (\cite{AlsenafiBarbaroConvectionDiffusionModel2018} proposes $f = g = \id$) are given.
Here, $u$ and $v$ denote the densities of two rivaling gangs which mark their territory by spraying graffiti with densities $z$ and $w$, respectively.
Apart from random motion, each gang is repelled by the other gangs' graffiti and thus partially moves away from high hostile graffiti concentrations;
the latter key effect is modelled by the taxis terms $+ \chi_u \nabla \cdot (u \nabla w)$, $+ \chi_v \nabla \cdot (v \nabla z)$ in the first two equations.
On the other hand, both gang's graffiti are immobile and decay over time.

\textbf{Related systems.}
In the present paper, we study \eqref{prob} both analytically and numerically.
Before presenting our results, we compare \eqref{prob} to related problems, namely double cross-diffusion and haptotaxis systems.
Under the assumption that the graffiti densities equilibrate instantly (and that $f=g=\id$), \cite{BarbaroEtAlAnalysisCrossdiffusionModel2021} reduces \eqref{prob} to the two-component system
\begin{align}\label{prob:2eq}
  \begin{cases}
    u_t = D_u \Delta u + \chi_u \nabla \cdot (u \nabla v), \\
    v_t = D_v \Delta v + \chi_v \nabla \cdot (v \nabla u)
  \end{cases}
\end{align}
(with positive parameters),
which can also be interpreted as a model for gangs \emph{directly} repelling each other,
and proves a weak-stability result as well as convergence of weak solutions (if they exist) to constant steady states under a smallness condition.
A key difficulty in establishing even a local solution theory for \eqref{prob:2eq} for large data
consists of the nonpositive definiteness of the diffusion matrix $\begin{pmatrix} D_u & \chi_u u \\ \chi_v v & D_v \end{pmatrix}$ whenever $uv > \frac{D_u D_v}{\chi_u \chi_v}$.
Nonetheless, global existence results have recently been obtained which, however, either need to require a certain regularization \cite{DucasseEtAlCrossdiffusionSystemObtained2023}
or can only guarantee solution properties within the parabolic regime \cite{XuExistenceTheoremPartially2023}.

The problem that a diffusion matrix is not positive definite can be overcome in multiple ways,
for instance by replacing linear diffusion with porous medium type diffusion, e.g., $\Delta u$ and $\Delta v$ by $\nabla \cdot (u \nabla u)$ and $\nabla \cdot (v \nabla v)$, respectively.
Indeed, for such a system, global, locally bounded weak solutions exist as long as $D_u D_v > \chi_u \chi_v$,
see \cite{LaurencotMatiocBoundedWeakSolutions2022a} and also its precedent \cite{LaurencotMatiocBoundedWeakSolutions2022}.
%
Moreover, one can also consider \eqref{prob:2eq} with $\chi_u \chi_v < 0$, which then models pursuit--evasion dynamics \cite{TyutyunovEtAlMinimalModelPursuitevasion2007}.
Again the diffusion matrix is positive definite as long as both components are nonnegative
but in contrast to the repulsion-repulsion problem with degenerate diffusion above,
apparently there only exists a single quasi energy functional 
and the a~priori estimates thereby gained only suffice to construct global weak solutions in the one-dimensional setting \cite{TaoWinklerFullyCrossdiffusiveTwocomponent2021, TaoWinklerExistenceTheoryQualitative2020};
in the higher-dimensional case, global weak solutions are only known to exist if the diffusion is sufficiently enhanced or the taxis is saturated \cite{FuestGlobalWeakSolutions2022}.
Furthermore, homogeneous steady states of \eqref{prob:2eq} with $\chi_u \chi_v < 0$ are asymptotically stable
in the sense that global classical solutions emanating from nearby initial data (exist and) converge to these equilibria \cite{FuestGlobalSolutionsHomogeneous2020}.

Another way of regularizing \eqref{prob:2eq} consists in replacing the cross-diffusive contributions with smoother functions; that is, in considering
\begin{align}\label{prob:kernels}
  \begin{cases}
    u_t = D_u \Delta u + \chi_u \nabla \cdot (u \nabla (K \ast v)), \\
    v_t = D_v \Delta v + \chi_v \nabla \cdot (v \nabla (K \ast u))
  \end{cases}
\end{align}
where $K$ denotes a spatial averaging kernel, for instance. (If $K$ is the Dirac delta distribution, one again obtains \eqref{prob:2eq}.)
The system \eqref{prob:kernels} can inter alia be used to describe territorial formations of various animals which remember direct encounters \cite{PottsLewisSpatialMemoryTaxisDriven2019};
for existence results we refer to \cite{GiuntaEtAlLocalGlobalExistence2022, JungelEtAlNonlocalCrossdiffusionSystems2022} and references therein.
Moreover, for the case when $K$ is Green's function for $-\Delta + 1$ with Neumann boundary conditions,
i.e., when $K \ast \varphi$ is the solution $\psi$ of the elliptic equation $-\Delta \psi + \psi = \varphi$ in $\Omega$ with $\partial_\nu \psi = 0$ on $\partial \Omega$,
and when $\chi_u \chi_v < 0$, global classical solutions of \eqref{prob:kernels} are constructed in \cite{LiEtAlLargeTimeBehavior2020}.
For related systems where the signal equations also regularize in time, see for instance \cite{PottsLewisHowMemoryDirect2016} or \cite{WuGlobalBoundednessDiffusive2021}.

In contrast, the third and fourth equations in \eqref{prob} regularize in time only.
Thus, mathematically, \eqref{prob} is related to haptotaxis problem such as
\begin{align*}
  \begin{cases}
    u_t = D_u \Delta u - \nabla \cdot(u \chi_u(v) \nabla v), \\
    v_t = -uv
  \end{cases}
\end{align*}
studied for instance in \cite{CorriasEtAlChemotaxisModelMotivated2003, CorriasEtAlGlobalSolutionsChemotaxis2004}.
These systems share the challenge of controlling cross-diffusion terms involving spatial derivatives of the signal(s) without relying on spatial regularity gained due to diffusion terms in the signal equation(s).

\paragraph{Main analytical results.}
For our global existence result regarding \eqref{prob}, we need to require that
\begin{align}\label{eq:intro:f_g}
  f, g \in C^1([0, \infty)) \cap L^\infty((0, \infty))
  \quad \text{are nonnegative}.
\end{align}
(However, Corollary~\ref{cor:ex_small_data} and Theorem~\ref{th:conv} below also apply to unbounded $f, g$ such as $f = g = \id$.)
A prototypical choice is given by $f(s) = g(s) = \frac{s}{1+s}$ for $s \ge 0$,
which not only satisfies \eqref{eq:intro:f_g} but also guarantees that no graffiti comes into existence out of nowhere by fulfilling $f(0) = g(0) = 0$.
For this example, the amount of sprayed graffiti increases roughly proportionally to the corresponding gang density at that point as long as the latter is rather small but is then limited by some positive constant.
Such a saturation effect appears to be reasonable: For large gang densities, the amount of additional wall writings in an area may be limited by available space rather than by the amount of gang members willing to spray graffiti.

For any such choice of graffiti production terms and any reasonably smooth initial data, we can construct globally bounded classical solutions of \eqref{prob}.
\begin{theorem}\label{th:main}
  Let
  \begin{align}\label{eq:main:omega}
    \Omega \subset \R^n, n \in \{1, 2, 3\}, \text{ be a smooth, bounded domain},
  \end{align}
  $D_u, D_v, \chi_u, \chi_v > 0$,
  $f, g$ be as in \eqref{eq:intro:f_g},
  $\alpha \in (0, 1)$ and $w_0, z_0 \in C^{2+\alpha}(\Ombar; [0, \infty))$.
  Then there exists $C > 0$ such that for all $M > 0$ and all nonnegative $u_0, v_0 \in \con{2+\alpha}$ with
  \begin{align}\label{eq:main:u0_v0}
    D_u \partial_\nu u_0 + \chi_u u_0 \partial_\nu w_0 = D_v \partial_\nu v_0 + \chi_v v_0 \partial_\nu z_0 = 0 \text{ on } \partial \Omega
    \quad \text{and} \quad
    \|u_0\|_{\leb\infty} + \|v_0\|_{\leb\infty} \le M,
  \end{align}
  there exists a unique, global, nonnegative classical solution $(u, v, w, z)$ of \eqref{prob},
  which satisfies the estimates
  \begin{align}\label{eq:main:bound}
    \|u(\cdot, t)\|_{\leb\infty} + \|v(\cdot, t)\|_{\leb\infty} \le CM
    \quad \text{and} \quad
    \|w(\cdot, t)\|_{\leb\infty} + \|z(\cdot, t)\|_{\leb\infty} \le C
    \qquad \text{for all $t \ge 0$}.
  \end{align}
\end{theorem}

A direct consequence of this theorem, especially of the bound \eqref{eq:main:bound}, is the existence of global small data solutions of the original system proposed in \cite{AlsenafiBarbaroConvectionDiffusionModel2018}.
\begin{cor}\label{cor:ex_small_data}
  Assume \eqref{eq:main:omega}, let $D_u, D_v, \chi_u, \chi_v > 0$, let $f(s) = g(s) = s$ for $s \ge 0$, $\alpha \in (0, 1)$ and let $w_0, z_0 \in C^{2+\alpha}(\Ombar; [0, \infty))$.
  Then there exists $M > 0$ such that for all nonnegative $u_0, v_0 \in \con{2+\alpha}$ fulfilling \eqref{eq:main:u0_v0}, 
  there exists a unique, global, bounded, nonnegative classical solution $(u, v, w, z)$ of \eqref{prob}.
\end{cor}

Having established global existence of solutions to \eqref{prob}, a natural next question is whether 
the first two components of these solutions separate.
We first give a negative answer for small initial data and show that the solutions converge towards homogeneous equilibria.
A corresponding result for the two-component system \eqref{prob:2eq} (with $\chi_1, \chi_2 > 0$)
has already been observed for global weak solutions (whose existence, however, is not known yet) in \cite[Theorem~5.1]{BarbaroEtAlAnalysisCrossdiffusionModel2021}
and also numerically in \cite[Section~6]{BarbaroEtAlAnalysisCrossdiffusionModel2021}.

\begin{theorem}\label{th:conv}
  Assume \eqref{eq:main:omega},
  let $D_u, D_v, \chi_u, \chi_v > 0$,
  suppose that $f, g \in C^1([0, \infty))$ are nonnegative
  and let $w_0, z_0 \in C^{2+\alpha}(\Ombar)$ for some $\alpha \in (0, 1)$ also be nonnegative.
  Then there exists $M > 0$ such that for all nonnegative $u_0, v_0 \in \con{2+\alpha}$ fulfilling \eqref{eq:main:u0_v0},
  there is a global classical solution $(u, v, w, z)$ of \eqref{prob} fulfilling
  \begin{alignat}{4}
    u(\cdot, t) &\ra \ol u_0    &\quad &\text{and} \quad & v(\cdot, t) &\ra \ol v_0     &&\qquad \text{in $\leb p$ for all $p \in [1, \infty)$ and} \label{eq:conv:uv}\\
    w(\cdot, t) &\ra f(\ol v_0) &\quad &\text{and} \quad & z(\cdot, t) &\ra g(\ol u_0)  &&\qquad \text{in $\sob12$ and in $\leb p$ for all $p \in [1, \infty)$} \label{eq:conv:wz}
  \end{alignat}
  as $t \to \infty$. (Here, we have set $\ol \varphi \defs \frac{1}{|\Omega|} \intom \varphi$ for $\varphi \in \leb1$.)
\end{theorem}

While Theorem~\ref{th:conv} settles the asymptotic behavior of small data solutions, it does not address the situation for large data;
in particular, nontrivial large-time stabilization leading to some kind of segregation may still be possible.
Usually, such questions are linked to the stability of heterogeneous steady states:
For instance, related systems modelling other types of segregation
such as the Shigesada--Kawasaki--Teramoto model~\cite{ShigesadaEtAlSpatialSegregationInteracting1979}
or the Cahn--Hilliard equation~\cite{CahnHilliardFreeEnergyNonuniform1958}
feature a rich structure of heterogeneous steady states which may attract various solutions,
see for instance \cite{MimuraKawasakiSpatialSegregationCompetitive1980, KutoYamadaMultipleCoexistenceStates2004} for the former
and \cite{WeiWinterStationaryCahnHilliardEquation1998, RybkaHoffmannConvergenceSolutionsCahnHilliard1999} for the later model as well as references therein.
For \eqref{prob}, however, the situation is entirely different: Convergence towards nonconstant smooth steady states is impossible simply because there are no such equilibria;
that is,
all smooth solutions of
\begin{align}\label{prob:ss}
  \begin{cases}
    0 = D_u \Delta u + \chi_u \nabla \cdot (u \nabla w) & \text{in $\Omega$}, \\
    0 = D_v \Delta v + \chi_v \nabla \cdot (v \nabla z) & \text{in $\Omega$}, \\
    0 = -w + f(v)                                       & \text{in $\Omega$}, \\
    0 = -z + g(u)                                       & \text{in $\Omega$}, \\
    D_u \partial_\nu u + \chi_u u \partial_\nu w = D_v \partial_\nu v + \chi_v v \partial_\nu z = 0 & \text{on $\partial \Omega$}
  \end{cases}
\end{align}
are constant.
In \cite[Proposition~4.2]{BarbaroEtAlAnalysisCrossdiffusionModel2021}, this has already been shown in the two-dimensional setting for $f(s) = g(s) = s$, $s \ge 0$.
(In fact, there even the existence of heterogeneous weak solutions to \eqref{prob:ss} fulfilling an entropy estimate is ruled out.)
We show that for many natural choices of $f$ and $g$, no nontrivial smooth steady states exist.
\begin{prop}\label{prop:no_hetero_ss}
  Let $\Omega \subset \R^n$, $n \in \N$, be a smooth, bounded domain,
  let $f, g$ be nonnegative, real analytic functions on $(-\eps, \infty)$ for some $\eps > 0$
  and suppose that $(u, v, w, z) \in (\con2)^4$ is a nonnegative classical solution of \eqref{prob:ss}.
  Then $u, v, w$ and $z$ are constant.
\end{prop}

This leaves open the question whether separation may occur at all.
However, Proposition~\ref{prop:no_hetero_ss} does not rule out the most drastic way of separation,
namely convergence towards multiples of characteristic functions of disjoint sets.
Likewise, other forms of large-time behavior involving infinite time gradient blow up are not excluded either.

\paragraph{Numerical simulations.}
As answering the question whether the gangs may separate analytically appears to be very difficult,
we instead perform numerical experiments which not only address the asymptotic behavior of large-time solutions but also generally illustrate the evolution of gang densities throughout time for various parameter regimes.

Cross-diffusion systems such as (\ref{prob}) can be considered as representatives of diffusion-convection-reaction equations (DCR) in computational fluid dynamics and often standard discretization methods for approximating the numerical solutions of DCR equations produce spuriously oscillating solutions whenever the convection is much larger than diffusion or reaction. Consequently, a vast variety of stabilization techniques has been introduced over the years to overcome these problems. The most popular among them is the Streamline upwind/Petrov--Galerkin (SUPG) method introduced in \cite{BrooksHughes1982}, for which another stabilization acting in crosswind direction has been added by nonlinear so-called spurious oscillations at layers diminishing (SOLD) methods \cite{JK07}. Local-projection stabilization (LPS) schemes \cite{Ganesan2010}, unusual stabilized finite element methods \cite{FrancaFarhat1998}, Mizukami--Hughes method \cite{MiukamiHughes1985}, Galerkin-Least-Square methods \cite{HughesFrancaHulbert1989} are among the other methods which have also been developed for stabilizing DCR problems in the convection-dominated regime. While all the aforementioned approaches attempt to stabilize the finite element method by adding additional terms to the Galerkin finite element discretization, flux-corrected transport (FCT) schemes have been developed  \cite{Kuz09, Kuz12a, KT02, book1975flux} as a different technique which is a nonlinear scheme working on the algebraic level by modifying the algebraic equation obtained from the Galerkin finite element method. Stabilization methods were further investigated and developed for time-dependent DCR models, see for instance \cite{Burman2010, john2021solvability, JhonNovo2011, FrancaHaukeMasud2006}, to just mention a few. However, most of these finite element stabilization techniques deal with the cases where the convection terms are linear and their implementation to nonlinear convection terms still calls for further investigation. To this end, applications of FEM-FCT have been studied in \cite{strehl2010flux, strehl2013positivity, sulman2019positivity} for Keller--Segel models; \cite{huang2021fully} considered a different model containing chemotaxis-Stokes equations and analyzed the error in \cite{FengHuangWang2021}, and the analysis of the solvability and positivity-preservation of the FEM-FCT for a model of cancer invasion has been studied recently in \cite{Hey2023}.

Accordingly, it is not surprising that also for the double cross-diffusion system (\ref{prob}) studied in this paper, standard discretization schemes such as Galerkin finite element method give rise to nonphysical oscillations leading to negative values in the approximate solutions when the sensitivity magnitudes $\chi_u$ and $\chi_v$ of the strongly nonlinear convection terms are large.
Thus as a remedy, we employ a high-resolution nonlinear finite element flux-corrected transport method to reduce the oscillations and preserve the positivity of the solution. Moreover, we deal with strong nonlinear coupling in the system and nonlinearity of the proposed scheme simultaneously using a fixed-point iteration.

To answer the question whether the gangs' populations separate from each other or not, we perform a series of numerical experiments using our newly designed algorithm which is implemented in finite element library deal.II \cite{deal2020, dealII94}. We show that for small values of $\chi$, gang populations stay completely mixed and that the approximate solutions converge toward constant steady-states, see Subsection~\ref{sec:numeric_conv_sss}. However, both partial and complete separation is observed for large values of $\chi$, see Subsections~\ref{sec:numeric_conv_nonsss} and \ref{sec:complete_seg}. Moreover, we compare our outcome with the reported results for the two-component version of \eqref{prob} \cite{BarbaroEtAlAnalysisCrossdiffusionModel2021} and
a related agent-based model \cite{AlsenafiBarbaroConvectionDiffusionModel2018, Alsenafi2021}.

\paragraph{Plan of the paper.}
This article is organized as follows:
Following a brief local existence result in Section~\ref{sec:local_ex}, Section~\ref{sec:global_ex} establishes various a~priori bounds which eventually allow us to prove global existence of solutions, i.e., Theorem~\ref{th:main} and Corollary~\ref{cor:ex_small_data}.
Next, Theorem~\ref{th:conv} and Proposition~\ref{prop:no_hetero_ss}, that is, statements on the asymptotical stability and existence of steady states, are derived in Section~\ref{sec:steady_states}. In Section~\ref{Numerics}, we first discretize the system using the $\theta$-method in time and a Galerkin finite element scheme in space and then enforce the positivity using the FEM-FCT scheme whenever the Galerkin method fails. We demonstrate the numerical results for different choices of parameters and study the convergence of our proposed  method with respect to time step and mesh in Section~\ref{simulation}.

\section{Local existence of classical solutions to a transformed system}\label{sec:local_ex}
As the third and fourth equations in \eqref{prob} do not regularize in space, the graffito-taxis terms in the first two equations are particularly challenging to deal with.
To overcome this issue, we introduce the transformations
\begin{align*}
  a \defs u \ure^{\xi_u w}
  \quad \text{and} \quad
  b \defs v \ure^{\xi_v z},
  \qquad
  \text{where} \quad
  \xi_u \defs \frac{\chi_u}{D_u}
  \quad \text{and} \quad
  \xi_v \defs \frac{\chi_v}{D_v},
\end{align*}
variations of which have been already used for the analysis of several haptotaxis systems (cf.\ \cite{FontelosEtAlMathematicalAnalysisModel2002} and \cite{FriedmanTelloStabilitySolutionsChemotaxis2002} for early examples).
Indeed, as
\begin{align*}
      a_t
  &=  u_t \ure^{\xi_u w} + \xi_u w_t u \ure^{\xi_u w} \\
  &=  \ure^{\xi_u w} \nabla \cdot (D_u \nabla (a \ure^{-\xi_u w}) + \chi_u a \ure^{-\xi_u w} \nabla w)
      + \xi_u w_t a \\
  &=  D_u \ure^{\xi_u w} \nabla \cdot (\ure^{-\xi_u w} \nabla a)
      - \xi_u aw + \xi_u a f(v)
\end{align*}
and likewise
\begin{align*}
      b_t
   =  D_v\ure^{\xi_v z} \nabla \cdot (\ure^{-\xi_v z} \nabla b)
      - \xi_v bz + \xi_v b g(u)
\end{align*}
in $\Omega \times (0, \infty)$ whenever $(u, v, w, z)$ is a global (classical) solution of \eqref{prob},
the system \eqref{prob} is equivalent to
\begin{align}\label{prob:transformed}
  \begin{cases}
    a_t = D_u \ure^{\xi_u w} \nabla \cdot (\ure^{-\xi_u w} \nabla a) - \xi_u aw + \xi_u a f(b \ure^{-\xi_v z}) & \text{in $\Omega \times (0, \infty)$}, \\
    b_t = D_v \ure^{\xi_v z} \nabla \cdot (\ure^{-\xi_v z} \nabla b) - \xi_u bz + \xi_v b g(a \ure^{-\xi_u w}) & \text{in $\Omega \times (0, \infty)$}, \\
    w_t = -w + f(b \ure^{-\xi_v z})                                                                        & \text{in $\Omega \times (0, \infty)$}, \\
    z_t = -z + g(a \ure^{-\xi_u w})                                                                        & \text{in $\Omega \times (0, \infty)$}, \\
    \partial_\nu a = \partial_\nu b = 0                                                                    & \text{on $\partial \Omega \times (0, \infty)$}, \\
    (a, b, w, z)(\cdot, 0) = (a_0, b_0, w_0, z_0)                                                          & \text{in $\Omega$},
  \end{cases}
\end{align}
where $a_0 = u_0 \ure^{\xi_u w_0}$ and $b_0 \defs v_0 \ure^{\xi_v z_0}$.
For this transformed system we have the following local existence result.

\begin{lemma}\label{lm:local_ex}
  Let $\Omega \subset \R^n$, $n \in \{1, 2, 3\}$, be a smooth, bounded domain, let $f, g$ be as in \eqref{eq:intro:f_g}, $\alpha \in (0, 1)$ and
  \begin{align}\label{eq:local_ex_init_reg}
    a_0, b_0, w_0, z_0 \in \con{2+\alpha} \quad \text{with} \quad \partial_\nu a_0 = \partial_\nu b_0 = 0 \text{ on } \partial \Omega.
  \end{align}
  Then there exists $\tmax \in (0, \infty]$ and a unique quadruple of nonnegative functions
  \begin{align*}
    (a, b, w, z) \in (C^{2+\alpha, 1+\frac{\alpha}{2}}(\Ombar \times (0, \tmax)) \cap C^1(\Ombar \times [0, \tmax)))^4
  \end{align*}
  with
  \begin{align*}
    (\nabla a, \nabla b, \nabla w, \nabla z) \in (C^1(\Ombar \times [0, \tmax)))^4
  \end{align*}
  solving \eqref{prob:transformed} classically with the property that if $\tmax < \infty$, then
  \begin{align}\label{eq:local_ex:ex_crit}
    \limsup_{t \nea \tmax} \left(
      \|a(\cdot, t)\|_{\con{1+\tilde\alpha}}
      + \|b(\cdot, t)\|_{\con{1+\tilde\alpha}}
    \right) = \infty
  \end{align}
  for all $\tilde \alpha \in (0, 1)$.
\end{lemma}
\begin{proof}
  This can be shown by means of a fixed point argument and parabolic regularity theory,
  see for instance \cite[Lemma~2.5]{FuestEtAlGlobalExistenceClassical2022} or \cite[Lemma~2.1 and Lemma~2.2]{TaoWinklerEnergytypeEstimatesGlobal2014} for details.
\end{proof}

\section{Global existence and boundedness}\label{sec:global_ex}
In this section, we always assume that
\begin{align}\label{eq:main_ass}
  \Omega \text{ is a smooth bounded domain in $\R^n$, $n \in \{1, 2, 3\}$}
  \text{ and }
  f, g \text{ are as in \eqref{eq:intro:f_g}}.
\end{align}

In order to prove that the solution constructed in Lemma~\ref{lm:local_ex} is global in time,
we need to show that \eqref{eq:local_ex:ex_crit} does not hold for some $\tilde \alpha \in (0, 1)$;
that is, that all solution components remain bounded in $\con{2+\tilde \alpha}$.
This is achieved by a series of a priori estimates, which in part rely on previously established bounds.
In particular, in Lemma~\ref{lm:nabla_w_z_l4_total} we apply the Lemma~\ref{lm:nabla_w_z_lp_cmp} (and hence indirectly also Lemmata~\ref{lm:w_z_linfty}--\ref{lm:a_b_linfty}) to the solution of \eqref{prob:transformed} with initial data $(a, b, z, w)(\cdot, t_0)$ for some $t_0 \in (0, \tmax)$.
Therefore, we need to carefully track the dependency of the constants in the estimates below on the initial data and thus introduce the condition
\begin{align}\label{eq:bdd_M}
  a_0, b_0, w_0, z_0 
  \quad \text{fulfill \eqref{eq:local_ex_init_reg} for some $\alpha \in (0, 1)$ and }
  \begin{cases}
    \|a_0\|_{\leb\infty} \le M, \\
    \|b_0\|_{\leb\infty} \le M, \\
    \|w_0\|_{\leb\infty} \le L, \\
    \|z_0\|_{\leb\infty} \le L
  \end{cases}
\end{align}
for $M, L > 0$.
That is, for fixed $M, L > 0$,
the constants given by Lemma~\ref{lm:w_z_linfty}, Lemma~\ref{lm:a_b_l1}, Lemma~\ref{lm:a_b_linfty} and Lemma~\ref{lm:nabla_w_z_lp_cmp} below
do not depend on the precise form of the initial data, provided those fulfill \eqref{eq:bdd_M}.
Also, the constants in the former three lemmata may not depend on $M$.

\subsection{\tops{$L^\infty$}{L infty} estimates}
We start with two rather basic estimates, namely $L^\infty$ bounds for the last two and $L^1$ bounds for the first two equations in \eqref{prob}.
Already at this point we make use of the fact that \eqref{eq:intro:f_g} contains boundedness of $f$ and $g$.
\begin{lemma}\label{lm:w_z_linfty}
  Assume \eqref{eq:main_ass} and let $L > 0$.
  Then there exists $\newgc{w_z_linfty} > 0$ such that for all $M > 0$ and all initial data satisfying \eqref{eq:bdd_M},
  the corresponding solution $(a, b, w, z)$ of \eqref{prob:transformed} given by Lemma~\ref{lm:local_ex} fulfills
  \begin{align}\label{eq:w_z_linfty:statement}
    \|w(\cdot, t)\|_{\leb\infty} \le \gc{w_z_linfty}
    \quad \text{and} \quad
    \|z(\cdot, t)\|_{\leb\infty} \le \gc{w_z_linfty}
    \qquad \text{for all $t \in (0, \tmax)$}.
  \end{align}
\end{lemma}
\begin{proof}
  The functions $\ol w \defs \max\{\|w_0\|_{\leb\infty}, \|f\|_{L^\infty((0, \infty))}\}$ and $\ol z \defs \max\{\|z_0\|_{\leb\infty}, \|g\|_{L^\infty((0, \infty))}\}$ are bounded supersolutions
  of the third and fourth subproblems in \eqref{prob:transformed}, respectively,
  and both $w$ and $z$ are nonnegative by Lemma~\ref{lm:local_ex}.
\end{proof}

\begin{lemma}\label{lm:a_b_l1}
  Assume \eqref{eq:main_ass} and let $L > 0$.
  Then there exists $\newgc{a_b_l1} > 0$ such that for all $M > 0$ and all initial data satisfying \eqref{eq:bdd_M},
  the corresponding solution $(a, b, w, z)$ of \eqref{prob:transformed} given by Lemma~\ref{lm:local_ex} fulfills
  \begin{align*}
    \|a(\cdot, t)\|_{\leb1} \le \gc{a_b_l1} M
    \quad \text{and} \quad
    \|b(\cdot, t)\|_{\leb1} \le \gc{a_b_l1} M
    \qquad \text{for all $t \in (0, \tmax)$}.
  \end{align*}
\end{lemma}
\begin{proof}
  Integrating the first equation in \eqref{prob} shows that $\intom u(\cdot, t) = \intom u_0$ for $t \in (0, \tmax)$. 
  With $\gc{w_z_linfty}$ as given by Lemma~\ref{lm:w_z_linfty}, the definition of $a$ thus implies
  \begin{align*}
        \intom a(\cdot, t)
    =   \intom (u \ure^{\xi_u w})(\cdot, t)
    \le \ure^{\xi_u \gc{w_z_linfty}} \intom u_0
    \le M |\Omega| \ure^{\xi_u \gc{w_z_linfty}}
    \qquad \text{for all $t \in (0, \tmax)$}.
  \end{align*}
  The bound for $b$ can be derived analogously.
\end{proof}

As to $L^\infty$ bounds of $a$ and $b$, we note that
\begin{align*}
  \ol a(x, t) \defs \|a_0\|_{\leb\infty} \ure^{\xi_u t\|f\|_{L^\infty((0, \infty))}}
  \quad \text{and} \quad
  \ol b(x, t) \defs \|b_0\|_{\leb\infty} \ure^{\xi_v t\|g\|_{L^\infty((0, \infty))}},
  \qquad (x, t) \in \Ombar \times [0, \tmax), 
\end{align*}
are supersolutions of the first two subproblems in \eqref{prob:transformed} and that accordingly $a$ and $b$ are bounded locally in time.
However, by employing testing procedures and a Moser-type iteration (following \cite{AlikakosBoundsSolutionsReactiondiffusion1979} and \cite{TaoWinklerBoundednessQuasilinearParabolic2012}),
we are also able to obtain $L^\infty$ bounds which are not only time-independent but which additionally depend favorably on $M$ as well.
\begin{lemma}\label{lm:a_b_linfty}
  Assume \eqref{eq:main_ass} and let $L > 0$.
  Then there exists $\newgc{a_b_linfty} > 0$ such that for all $M > 0$ and all initial data satisfying \eqref{eq:bdd_M},
  the corresponding solution $(a, b, w, z)$ of \eqref{prob:transformed} given by Lemma~\ref{lm:local_ex} fulfills
  \begin{align}\label{eq:a_b_linfty:statement}
    \|a(\cdot, t)\|_{\leb\infty} \le \gc{a_b_linfty} M
    \quad \text{and} \quad
    \|b(\cdot, t)\|_{\leb\infty} \le \gc{a_b_linfty} M
    \qquad \text{for all $t \in (0, \tmax)$}.
  \end{align}
\end{lemma}
\begin{proof}
  We fix initial data satisfying \eqref{eq:bdd_M} and the corresponding solution $(a, b, w, z)$ of \eqref{prob:transformed} given by Lemma~\ref{lm:local_ex}.
  
  Moreover, by a quantitative version of Ehrling's lemma proved in \cite[Lemma~2.5]{FuestBlowupProfilesQuasilinear2020}, there exist $\newlc{ehrling} > 0$ and $\mu > 0$ such that
  \begin{align*}
          \lc{f_linfty} p \intom \varphi^2
    &\le  \frac{2D_u\ure^{-\xi_u\gc{w_z_linfty}}}{p} \intom |\nabla \varphi|^2
          + \lc{ehrling} p^\mu \left( \intom |\varphi| \right)^2
    \qquad \text{for all $\varphi \in \sob12$ and all $p \ge 2$}.
  \end{align*}
  where $\newlc{f_linfty} \defs \xi_u \|f\|_{L^\infty((0, \infty))}+1$ and where $\gc{w_z_linfty}$ is given by Lemma~\ref{lm:w_z_linfty}.
  For $p \ge 2$ and $T \in (0, \tmax)$, we can thus calculate
  \begin{align*}
          \ddt \intom \ure^{-\xi_u w} a^p
    &=    p \intom \ure^{-\xi_u w} a^{p-1} a_t 
          - \xi_u \intom \ure^{-\xi_u w} a^p w_t  \\
    &=    - D_u p \intom \ure^{-\xi_u w} \nabla a \cdot \nabla a^{p-1}
          + \xi_u(p-1) \intom \ure^{-\xi_u w} a^p (-w + f(b \ure^{-\xi_v z})) \\
    &\le  - \frac{4D_u(p-1)\ure^{-\xi_u\gc{w_z_linfty}}}{p^2} \intom |\nabla a^\frac p2|^2
          + \lc{f_linfty} (p-1) \intom \ure^{-\xi_u w} a^p \\
    &\le  - \lc{f_linfty} \intom \ure^{-\xi_u w} a^p
          - \frac{2D_u\ure^{-\xi_u\gc{w_z_linfty}}}{p} \intom |\nabla a^\frac p2|^2
          + \lc{f_linfty} p \intom (a^\frac p2)^2 \\
    &\le  - \lc{f_linfty} \intom \ure^{-\xi_u w} a^p
          + \lc{ehrling} p^\mu \left( \intom a^{\frac p2} \right)^2 \\
    &\le  - \lc{f_linfty} \intom \ure^{-\xi_u w} a^p
          + \lc{ehrling} p^\mu \sup_{s \in (0, T)} \left( \intom a^{\frac p2}(\cdot, s) \right)^2
    \qquad \text{in $(0, T)$},
  \end{align*}
  which in combination with an ODE comparison argument implies
  \begin{align*}
        \intom (\ure^{-\xi_u w} a^p)(\cdot, t)
    \le \max\left\{ \intom \ure^{-\xi_u w_0} a_0^p,\; \frac{\lc{ehrling} p^\mu}{\lc{f_linfty}} \sup_{s \in (0, T)} \left( \intom a^{\frac p2}(\cdot, s) \right)^2 \right\}
  \end{align*}
  for all $T \in (0, \tmax)$, all $t \in (0, T)$ and all $p \ge 2$.
  Thus, setting $\newlc{a_lp} \defs \ure^{\xi_u \gc{w_z_linfty}} \max\{1, \frac{\lc{ehrling}}{\lc{f_linfty}}\}$, we obtain
  \begin{align}\label{eq:a_b_linfty:a_lp_est}
        \intom a^p(\cdot, t)
    \le \lc{a_lp} \max\left\{ \intom a_0^p,\; p^\mu \sup_{s \in (0, T)} \left( \intom a^{\frac p2}(\cdot, s) \right)^2 \right\}
  \end{align}
  for all $T \in (0, \tmax)$, all $t \in (0, T)$ and all $p \ge 2$.
  We next set
  \begin{align*}
    p_j \defs 2^j
    \quad \text{and} \quad
    A_j(T) \defs \sup_{s \in (0, T)} \|a(\cdot, s)\|_{\leb{p_j}}
    \qquad \text{for $j \in \N_0$ and $T \in (0, \tmax)$},
  \end{align*}
  so that with $\newlc{Aj_left_bdd} \defs \lc{a_lp} \max\{1, |\Omega|\}$ and
  as the condition $\|a_0\|_{\leb\infty} \le M$ in \eqref{eq:bdd_M} implies $\|a_0\|_{\leb{p_j}} \le |\Omega|^{\frac{1}{p_j}} M \le \max\{1, |\Omega|\} M$,
  we further infer from \eqref{eq:a_b_linfty:a_lp_est} that
  \begin{align*}
    A_j(T) \le \max\left\{\lc{Aj_left_bdd} M, (\lc{a_lp} p_j^\mu)^\frac{1}{p_j} A_{j-1}(T) \right\}
    \qquad \text{for all $j \in \N$ and $T \in (0, \tmax)$}.
  \end{align*}
  We now fix $T \in (0, \tmax)$.
  If there are infinitely many $j \in \N_0$ with $A_j(T) \le \lc{Aj_left_bdd} M$,
  then
  \begin{align}\label{eq:a_b_linfty:bdd_1}
        \|a\|_{L^\infty(\Omega \times (0, T))}
    =   \liminf_{j \to \infty} A_j(T)
    \le \lc{Aj_left_bdd} M.
  \end{align}
  Else there exists $j_0 \in \N_0$  such that
  \begin{align}\label{eq:a_b_linfty:aj_right}
    A_j(T) \le (\lc{a_lp} p_j^\mu )^\frac1{p_j} A_{j-1}(T)
    \qquad \text{for all $j > j_0$}.
  \end{align}
  (We note that while $j_0$ may depend on the initial data, the constant $\gc{a_b_linfty}$ defined below will not.)
  Choosing $j_0$ minimal, we can moreover assume that
  \begin{align}\label{eq:a_b_linfty:est_bj0}
    A_{j_0}(T) \le \max\{\gc{a_b_l1}, \lc{Aj_left_bdd}\} M,
  \end{align}
  where $\gc{a_b_l1}$ is as in Lemma~\ref{lm:a_b_l1}. (That lemma guarantees $A_0(T) \le \gc{a_b_l1} M$.)
  By induction and as $p_j = 2^j$ for $j \in \N$, \eqref{eq:a_b_linfty:aj_right} implies
  \begin{align*}
        A_{j}(T)
    \le \left( \prod_{k=j_0+1}^j (\lc{a_lp} p_k^\mu )^\frac1{p_k} \right) A_{j_0}
    =   {\lc{a_lp}\vphantom{2}}^{\sum_{k=j_0+1}^j 2^{-k}} \cdot 2^{\mu \sum_{k=j_0+1}^j k 2^{-k}} \cdot A_{j_0}
    \qquad \text{for all $j > j_0$}.
  \end{align*}
  In combination with \eqref{eq:a_b_linfty:est_bj0}, we arrive at
  \begin{align*}
          \|a\|_{L^\infty(\Omega \times (0, T))}
    &=    \lim_{j \ra \infty} A_j(T)
    \le   {\lc{a_lp}\vphantom{2}}^{\sum_{k=j_0+1}^\infty 2^{-k}} \cdot 2^{\mu \sum_{k=j_0+1}^\infty k 2^{-k}} \cdot A_{j_0} \\
    &\le  {\lc{a_lp}\vphantom{2}}^{\sum_{k=0}^\infty 2^{-k}} \cdot 2^{\mu \sum_{k=0}^\infty k 2^{-k}} \cdot \max\{\gc{a_b_l1}, \lc{Aj_left_bdd}\} M
    \sfed \newlc{Aj_right_bdd} M.
  \end{align*}
  Together with \eqref{eq:a_b_linfty:bdd_1} this implies $\|a\|_{L^\infty(\Omega \times (0, T))} \le \gc{a_b_linfty} M$ for all $T \in (0, \tmax)$,
  where $\gc{a_b_linfty} \defs \max\{\lc{Aj_left_bdd}, \lc{Aj_right_bdd}\}$.
  Letting $T \nea \tmax$, we obtain the first statement in \eqref{eq:a_b_linfty:statement},
  while the second one follows upon an analogous computation for the second solution component of \eqref{prob:transformed}, possibly after enlarging $\gc{a_b_linfty}$.
\end{proof}

\subsection{Gradient estimates}
Our next step towards showing that \eqref{eq:local_ex:ex_crit} cannot hold consists in deriving uniform-in-time $L^4$ estimates for $\nabla w$ and $\nabla z$.
To that end, we follow a technique introduced in \cite[Lemmata~3.13--3.15]{TaoWinklerEnergytypeEstimatesGlobal2014} for a haptotaxis system.
As a preparation, we first note that the time regularization in the third and fourth equation in \eqref{prob:transformed} implies
that space-time gradient estimates for $a$ and $b$ imply uniform-in-time gradient estimates for $w$ and $z$.
\begin{lemma}\label{lm:nabla_w_z_lp_cmp}
  Assume \eqref{eq:main_ass} and let $L, M > 0$. 
  Moreover, let $T_0 > 0$ and $p \in (1, \infty)$.
  Then there exists $\newgc{nabla_w_z_lp_cmp} > 0$ such that for all initial data satisfying \eqref{eq:bdd_M},
  the corresponding solution $(a, b, w, z)$ of \eqref{prob:transformed} given by Lemma~\ref{lm:local_ex} fulfills
  \begin{align}\label{eq:nabla_w_z_lp_cmp:statement}
        \intom |\nabla w(\cdot, t)|^p + \intom |\nabla z(\cdot, t)|^p
    \le \gc{nabla_w_z_lp_cmp} \left( \intom w_0^p + \intom z_0^p + \intnstom |\nabla a|^p + \intnstom |\nabla b|^p \right)
  \end{align}
  for all $t \in (0, T)$, where $T \defs \min\{T_0, \tmax\}$.
\end{lemma}
\begin{proof}
  We again fix initial data satisfying \eqref{eq:bdd_M} and the solution $(a, b, w, z)$ of \eqref{prob:transformed} given by Lemma~\ref{lm:local_ex}.
  By Lemma~\ref{lm:a_b_linfty}, we can find $\gc{a_b_linfty} > 0$ such that $a, b \le \gc{a_b_linfty} M$ in $\Omega \times (0, T)$.
  Since according to the variation-of-constants formula we may write $w(x, t) = \ure^{-t} w_0(x) + \int_0^t \ure^{-(t-s)} f((b \ure^{-\xi_v z})(x, s)) \ds$ for $(x, t) \in \Omega \times (0, \tmax)$,
  we have
  \begin{align*}
          \intom |\nabla w(\cdot, t)|^p
    &\le  2^{p-1} \intom |\ure^{-t} \nabla w_0|^p
          + 2^{p-1} \intom \int_0^t \ure^{-(t-s)} |\nabla (f((b \ure^{-\xi_v z})(\cdot, s)))|^p \ds \\
    &\le  2^{p-1} \intom |\nabla w_0|^p
          + 2^{p-1} \|f'\|_{C^0([0, \gc{a_b_linfty} M])} \intnstom |\nabla (b \ure^{-\xi_v z})|^p
    \qquad \text{for $t \in (0, T)$}.
  \end{align*}
  As
  \begin{align*}
          |\nabla (b \ure^{-\xi_v z})|^p
    &\le  2^{p-1} \ure^{-\xi_v pz} |\nabla b|^p + 2^{p-1} \xi_v^p b^p \ure^{-\xi_v pz} |\nabla z|^p \\
    &\le  2^{p-1} |\nabla b|^p + 2^{p-1} \xi_v^p \gc{a_b_linfty}^p M^p |\nabla z|^p
    \qquad \text{in $\Omega \times (0, T)$},
  \end{align*}
  and together with an analogous argumentation for $\intom |\nabla z(\cdot, t)|^p$,
  we can conclude
  \begin{align*}
          \intom |\nabla w(\cdot, t)|^p + \intom |\nabla z(\cdot, t)|^p
    &\le  \newlc{gronwall_1} \left( \intom |\nabla w_0|^p + \intom |\nabla z_0|^p + \intnstom |\nabla a|^p + \intnstom |\nabla b|^p \right) \\
    &\pe  + \newlc{gronwall_2} \int_0^t \left( \intom |\nabla w|^p + \intom |\nabla z|^p \right)
  \end{align*}
  for $t \in (0, T)$,
  where
  \begin{align*}
    \lc{gronwall_1} &\defs  \max\left\{
      2^{p-1},
      2^{2p-2} \|f'\|_{C^0([0, \gc{a_b_linfty} M])},
      2^{2p-2} \|g'\|_{C^0([0, \gc{a_b_linfty} M])}
    \right\}
  \intertext{and}
    \lc{gronwall_2} &\defs
      2^{2p-2} \gc{a_b_linfty}^p M^p
      \max\left\{ \xi_v^p \|f'\|_{C^0([0, \gc{a_b_linfty} M])}, \xi_u^p \|g'\|_{C^0([0, \gc{a_b_linfty} M])} \right\}.
  \end{align*}
  Thus, Gronwall's inequality asserts the statement for $\gc{nabla_w_z_lp_cmp} \defs \lc{gronwall_1} \ure^{\lc{gronwall_2} T_0}$.
\end{proof}

In order to show that the right-hand side in \eqref{eq:nabla_w_z_lp_cmp:statement} is bounded for $p=4$,
we consider the evolution of the function $\frac12 (\intom |\nabla a|^2 + \intom |\nabla b|^2)$.
As it turns out, however, we can only control its time derivative on small timescales
and thus aim to derive estimates in $(t_0, \tmax)$ for $t_0$ close to $\tmax$ (which may be assumed to be finite for the sake of contradiction) only.
To that end, it is crucial that Lemma~\ref{lm:a_b_linfty} provides $L^\infty$ estimates for $a$ and $b$ in terms of the $\leb\infty$ norm of $a_0$ and $b_0$,
as we then may apply Lemma~\ref{lm:nabla_w_z_lp_cmp} to the solution with initial data $(a, b, w, z)(\cdot, t_0)$
without worrying about the dependency of the constant $\gc{nabla_w_z_lp_cmp}$ thereby obtained on $t_0$.
\begin{lemma}\label{lm:nabla_w_z_l4_total}
  Assume \eqref{eq:main_ass} as well as \eqref{eq:bdd_M} for some $L, M > 0$ and denote the solution of \eqref{prob:transformed} given by Lemma~\ref{lm:local_ex} by $(a, b, w, z)$.
  Moreover, let $T \in (0, \tmax] \cap (0, \infty)$.
  Then there exists $\newgc{nabla_w_z_l4_total} > 0$ such that
  \begin{align}\label{eq:nabla_w_z_l4_total:statement}
        \intom |\nabla w(\cdot, t)|^4 + \intom |\nabla z(\cdot, t)|^4
    \le \gc{nabla_w_z_l4_total}
    \qquad \text{for $t \in (0, T)$}.
  \end{align}
\end{lemma}
\begin{proof}
  It follows from the Gagliardo--Nirenberg inequality (that is, from \cite{NirenbergEllipticPartialDifferential1959}; or, more directly, from \cite[Lemma~A.3]{FuestGlobalSolutionsHomogeneous2020}))
  that there is $\newlc{gni} > 0$ such that
  \begin{align}\label{eq:nabla_w_z_l4_total:gni}
    \frac{\max\{D_u, D_v\}}{\min\{D_u, D_v\}} \intom |\nabla \varphi|^4 \le \lc{gni}  \left( \intom |\Delta \varphi|^2 \right) \|\varphi\|_{\leb\infty}^2 + \lc{gni} \|\varphi\|_{\leb\infty}^4
  \end{align}
  for all $\varphi \in \con2$ with $\partial_\nu \varphi = 0$ on $\partial \Omega$.
  Moreover, we let $\gc{w_z_linfty}$ and $\gc{a_b_linfty}$ be as given by Lemma~\ref{lm:w_z_linfty} and Lemma~\ref{lm:a_b_linfty}, respectively.
  Then \eqref{eq:nabla_w_z_l4_total:gni} and Lemma~\ref{lm:a_b_linfty} imply
  \begin{align}\label{eq:nabla_w_z_l4_total:gni2}
        \intom |\nabla a|^4 + \intom |\nabla b|^4
    \le C_3^2 M^2 \lc{gni} \left(\intom |\Delta a|^2 + \intom |\Delta b|^2 \right)
        + 2 C_3^4 M^4\lc{gni}
  \end{align}
  Next, we apply Lemma~\ref{lm:nabla_w_z_lp_cmp} (with $M = \gc{a_b_linfty} M$, $L = \gc{w_z_linfty}$, $T_0 = 1$ and $p = 4$) to obtain $\widehat{\gc{nabla_w_z_lp_cmp}} > 0$
  with the property that whenever a quadruple of functions $(\widehat{a_0}, \widehat{b_0}, \widehat{w_0}, \widehat{z_0})$ satisfies \eqref{eq:bdd_M} with $M$ replaced by $\gc{a_b_linfty} M$
  (and $(a_0, b_0, w_0, z_0)$ replaced by $(\widehat{a_0}, \widehat{b_0}, \widehat{w_0}, \widehat{z_0})$),
  then the corresponding solution $(\widehat a, \widehat b, \widehat w, \widehat z)$ of \eqref{prob:transformed} with maximal existence time $\widehat T_{\mathrm{max}}$ given by Lemma~\ref{lm:local_ex} fulfills
  \begin{align*}
        \intom |\nabla \widehat w(\cdot, t)|^4 + \intom |\nabla \widehat z(\cdot, t)|^4
    \le \widehat{\gc{nabla_w_z_lp_cmp}} \left( \intom (\widehat{w_0})^4 + \intom (\widehat{z_0})^4 + \intnstom |\nabla \widehat a|^4 + \intnstom |\nabla \widehat b|^4 \right)
  \end{align*}
  for all $t \in (0, \max\{1, \widehat T_{\mathrm{max}}\})$.
  Setting $D \defs \min\{D_u, D_v\}$, $\chi \defs \max\{\chi_u, \chi_v\}$, $\xi \defs \frac{\chi}{D}$,
  \begin{align}\label{eq:nabla_w_z_l4_total:def_t0}
    t_0 \defs \max\left\{0, T - \frac{1}{16\xi^4\gc{a_b_linfty}^4 \widehat{\gc{nabla_w_z_lp_cmp}} M^4 \lc{gni}^2}, T - 1\right\}
  \end{align}
  as well as $\newlc{nabla_w_z_l4_cmp} \defs \widehat{\gc{nabla_w_z_lp_cmp}} \left( \intom w^4(\cdot, t_0) + \intom z^4(\cdot, t_0) \right)$
  and recalling that classical solutions of \eqref{prob:transformed} are unique by Lemma~\ref{lm:local_ex},
  we can conclude
  \begin{align}\label{eq:nabla_w_z_l4_total:nabla_w'_z'_l4}
        \intom |\nabla w(\cdot, t)|^4 + \intom |\nabla z(\cdot, t)|^4
    \le \lc{nabla_w_z_l4_cmp} + \widehat{\gc{nabla_w_z_lp_cmp}} \left( \int_{t_0}^t \intom |\nabla a|^4 + \int_{t_0}^t \intom |\nabla b|^4 \right)
    \qquad \text{for all $t \in (t_0, T)$}.
  \end{align}
  Since $D_u \ure^{\xi_u w} \nabla \cdot (\ure^{-\xi_u w} \nabla a) = D_u \Delta a - \chi_u \nabla a \cdot \nabla w$,
  testing the equation for $a$ with $-\Delta a$ and applying Young's inequality thrice 
  gives
  \begin{align*}
          \frac12 \ddt \intom |\nabla a|^2
    &=    - D_u \intom |\Delta a|^2
          + \chi_u \intom (\nabla a \cdot \nabla w) \Delta a
          - \xi_u \intom (-aw + a f(b\ure^{-\xi_v z})) \Delta a \\
    &\le  - \frac{D_u}{2} \intom |\Delta a|^2
          + \chi \xi \intom |\nabla a|^2 |\nabla w|^2
          + \frac{\xi^2}{D} \intom a^2 (w + f(b\ure^{-\xi_v z}))^2 \\
    &\le  - \frac{D_u}{2} \intom |\Delta a|^2
          + \frac{D}{8\gc{a_b_linfty}^2 M^2\lc{gni}} \intom |\nabla a|^4
          + 2 \xi^4 D \gc{a_b_linfty}^2 M^2\lc{gni}  \intom |\nabla w|^4
          + \lc{zeroth_order_a}
    \qquad \text{in $(0, T)$},
  \end{align*}
  where $\newlc{zeroth_order_a} \defs \frac{\gc{a_b_linfty}^2 M^2 \xi^2}{D} (\gc{w_z_linfty} + \|f\|_{L^\infty((0, \infty))})^2 |\Omega|$.
  By integrating in time from $t_0$ to $t$ and noting that $T-t_0 \le 1$, we obtain
  \begin{align}\label{eq:nabla_w_z_l4_total:nabla_a_l2}
    &\pe  \intom |\nabla a(\cdot, t)|^2
          - \intom |\nabla a(\cdot, t_0)|^2
          + D_u \inttnstom  |\Delta a|^2 \notag \\
    &\le  \frac{D}{4\gc{a_b_linfty}^2M^2\lc{gni}} \inttnstom |\nabla a|^4
          + 4 \xi^4 D \gc{a_b_linfty}^2M^2\lc{gni} (t-t_0) \sup_{s \in (t_0, t)} \intom |\nabla w(\cdot, s)|^4
          + 2 \lc{zeroth_order_a} 
  \end{align}
  for all $t \in (t_0, T)$.
  Likewise, there is $\newlc{zeroth_order_b} > 0$ such that
  \begin{align}\label{eq:nabla_w_z_l4_total:nabla_b_l2}
    &\pe  \intom |\nabla b(\cdot, t)|^2
          - \intom |\nabla b(\cdot, t_0)|^2
          + D_v \inttnstom |\Delta b|^2 \notag \\
    &\le  \frac{D}{4\gc{a_b_linfty}^2M^2\lc{gni}} \inttnstom |\nabla b|^4
          + 4 \xi^4 D \gc{a_b_linfty}^2M^2\lc{gni} (t-t_0) \sup_{s \in (t_0, t)} \intom |\nabla z(\cdot, s)|^4
          + 2\lc{zeroth_order_b} 
  \end{align}
  for all $t \in (t_0, T)$.
  By \eqref{eq:nabla_w_z_l4_total:nabla_a_l2}, \eqref{eq:nabla_w_z_l4_total:nabla_b_l2}, \eqref{eq:nabla_w_z_l4_total:nabla_w'_z'_l4}, \eqref{eq:nabla_w_z_l4_total:def_t0} and \eqref{eq:nabla_w_z_l4_total:gni2},
  we therefore have
  \begin{align*}
    &\pe  \intom |\nabla a(\cdot, t)|^2 + \intom |\nabla b(\cdot, t)|^2
          - \intom |\nabla a(\cdot, t_0)|^2 - \intom |\nabla b(\cdot, t_0)|^2
          + D_u \intnstom |\Delta a|^2 + D_v \intnstom |\Delta b|^2 \\
    &\le  \frac{D}{4\gc{a_b_linfty}^2M^2\lc{gni}} \left( \inttnstom |\nabla a|^4 + \inttnstom |\nabla b|^4 \right) \\
    &\pe  + 4\xi^4 D \gc{a_b_linfty}^2M^2\lc{gni} (t-t_0) \sup_{s \in (t_0, t)} \left( \intom |\nabla w(\cdot, s)|^4 + \intom |\nabla z(\cdot, s)|^4 \right)
          + 2(\lc{zeroth_order_a} + \lc{zeroth_order_b}) \\
    &\le  \underbrace{\left( \frac{D}{4\gc{a_b_linfty}^2M^2\lc{gni}} + 4 \xi^4 D\gc{a_b_linfty}^2\widehat{\gc{nabla_w_z_lp_cmp}} M^2 \lc{gni} (T-t_0) \right)}
            _{\le D/(2\gc{a_b_linfty}^2M^2\lc{gni})}
            \left( \inttnstom |\nabla a|^4 + \inttnstom |\nabla b|^4 \right)
          + \newlc{coef_t_1} \\
    &\le  D_u \inttnstom |\Delta a|^2 + D_v \inttnstom |\Delta b|^2
          - \frac{D}{2\gc{a_b_linfty}^2M^2\lc{gni}} \left( \inttnstom |\nabla a|^4 + \inttnstom |\nabla b|^4 \right)
          + \newlc{coef_t_2}
  \end{align*}
  for $t \in (t_0, T)$,
  where $\lc{coef_t_1} \defs 4\xi^4 D \gc{a_b_linfty}^2M^2\lc{gni} \lc{nabla_w_z_l4_cmp} + 2\lc{zeroth_order_a} + 2\lc{zeroth_order_b}$
  and $\lc{coef_t_2} \defs \lc{coef_t_1} + 2D\gc{a_b_linfty}^2M^2$.
  By Beppo Levi's theorem, this implies
  \begin{align*}
        \inttntom |\nabla a|^4 + \inttntom |\nabla b|^4
    \le \frac{2\gc{a_b_linfty}^2M^2\lc{gni}}{D} \left( \intom |\nabla a(\cdot, t_0)|^2 + \intom |\nabla b(\cdot, t_0)|^2 + \lc{coef_t_2} \right).
  \end{align*}
  Another application of Lemma~\ref{lm:nabla_w_z_lp_cmp} then shows that the desired estimate \eqref{eq:nabla_w_z_l4_total:statement} holds for all $t \in (t_0, T)$ and some $\gc{nabla_w_z_l4_total} > 0$,
  while the inclusions $w, z \in C^1(\Ombar \times [0, t_0])$ trivially entail \eqref{eq:nabla_w_z_l4_total:statement} also for $t \in [0, t_0]$ (possibly after enlarging $\gc{nabla_w_z_l4_total}$).
\end{proof}

\subsection{Solutions are global in time: proof of Theorem~\ref{th:main} and Corollary~\ref{cor:ex_small_data}}
With Lemma~\ref{lm:nabla_w_z_l4_total} at hand, globality in time can be shown as in \cite[Lemma 2.14]{FuestEtAlGlobalExistenceClassical2022} (or \cite[Lemma~2.2]{TaoWinklerEnergytypeEstimatesGlobal2014}):
Parabolic regularity theory rapidly upgrades the bounds implied by \eqref{eq:w_z_linfty:statement}, \eqref{eq:a_b_linfty:statement} and \eqref{eq:nabla_w_z_l4_total:statement} to Hölder estimates.
\begin{lemma}\label{lm:global_ex}
  Assume \eqref{eq:main_ass} as well as \eqref{eq:bdd_M} for some $L, M > 0$.
  Then the solution $(a, b, w, z)$ given by Lemma~\ref{lm:local_ex} is global in time; that is, $\tmax = \infty$. 
\end{lemma}
\begin{proof}
  Suppose $\tmax < \infty$.
  We rewrite the first two equations in \eqref{prob:transformed} as
  \begin{align*}
    a_t = D_u \Delta a - \chi_u \nabla a \cdot \nabla w + \psi_a
    \quad \text{and} \quad
    b_t = D_v \Delta b - \chi_v \nabla b \cdot \nabla z + \psi_b
    \qquad \text{in $\Omega \times (0, \infty)$},
  \end{align*}
  where
  \begin{align*}
    \psi_a = -\xi_u aw + \xi_u a f(b \ure^{-\xi_v z})
    \quad \text{and} \quad
    \psi_b = -\xi_v bz + \xi_v b g(a \ure^{-\xi_u w}).
  \end{align*}
  As $a, b, w, z$ belong to $L^\infty(\Omega \times (0, \tmax))$ by Lemma~\ref{lm:a_b_linfty} and Lemma~\ref{lm:w_z_linfty}, so do $\psi_1, \psi_2$.
  Also recalling \eqref{eq:local_ex_init_reg} and Lemma~\ref{lm:nabla_w_z_l4_total},
  we may thus apply (a consequence of) maximal Sobolev regularity (cf.\ \cite[Lemma~2.13]{FuestEtAlGlobalExistenceClassical2022}) to obtain $\newlc{max_sob_reg_12_3} > 0$ such that
  \begin{align*}
    \|\nabla a\|_{L^{12}((0, \tmax); \leb\infty)} + \|\nabla b\|_{L^{12}((0, \tmax); \leb\infty)} \le \lc{max_sob_reg_12_3}.
  \end{align*}
  According to Lemma~\ref{lm:nabla_w_z_lp_cmp}, there then exists $\newlc{nabla_w_z_l12} > 0$ with
  \begin{align*}
    \|\nabla w\|_{L^\infty((0, \tmax); \leb{12})} + \|\nabla z\|_{L^\infty((0, \tmax); \leb{12})} \le \lc{nabla_w_z_l12}.
  \end{align*}
  This allows us to again invoke \cite[Lemma~2.13]{FuestEtAlGlobalExistenceClassical2022} to obtain $\newlc{max_sob_reg_12_12} > 0$ such that
  \begin{align*}
    \|a_t\|_{L^{12}(\Omega \times (0, \tmax))} + \|\Delta a\|_{L^{12}(\Omega \times (0, \tmax))} 
    + \|b_t\|_{L^{12}(\Omega \times (0, \tmax))} + \|\Delta b\|_{L^{12}(\Omega \times (0, \tmax))}
    \le \lc{max_sob_reg_12_12}.
  \end{align*}
  Thus, by \cite[Lemma~II.3.3]{LadyzenskajaEtAlLinearQuasilinearEquations1988} there is $\newlc{a_b_c10+alpha} > 0$ such that
  \begin{align*}
    \|a\|_{C^{\frac{19}{12}, \frac{19}{24}}(\Ombar \times [0, \tmax])} + \|b\|_{C^{\frac{19}{12}, \frac{19}{24}}(\Ombar \times [0, \tmax])} \le \lc{a_b_c10+alpha},
  \end{align*}
  which contradicts the extensibility criterion in Lemma~\ref{lm:local_ex}.
\end{proof}

Theorem~\ref{th:main} follows now easily from the lemmata above.
\begin{proof}[Proof of Theorem~\ref{th:main}]
  That the solution $(a, b, w, z)$ of \eqref{prob:transformed} constructed in Lemma~\ref{lm:local_ex} is global in time has been asserted in Lemma~\ref{lm:global_ex}.
  Upon setting $u \defs a \ure^{-\xi_u w}$ and $v \defs b \ure^{-\xi_v z}$, we also obtain a global classical solution of \eqref{prob}
  which due to Lemma~\ref{lm:a_b_linfty}, the evident estimates $u \le a$, $v \le b$ in $\Ombarinf$ and Lemma~\ref{lm:w_z_linfty} moreover fulfills \eqref{eq:main:bound} for some $C > 0$.
\end{proof}

Finally, we show that Theorem~\ref{th:main} allows for a quick proof of Corollary~\ref{cor:ex_small_data}; that is, of the existence of small data solutions for \eqref{prob} with $f(s) = g(s) = s$ for $s \ge 0$.
\begin{proof}[Proof of Corollary~\ref{cor:ex_small_data}]
  We let $\Omega$, $\alpha$ and $w_0, z_0$ be as in the statement of Corollary~\ref{cor:ex_small_data}.
  Moreover, we fix a nonnegative cutoff function $\zeta \in \con\infty$ with $\zeta(s) = s$ for $s \in [0, 1]$ and $\zeta \equiv 2$ in $[2, \infty)$.
  Then $\tilde f = \tilde g = \zeta$ fulfill \eqref{eq:intro:f_g}, so that for each $M > 0$,
  Theorem~\ref{th:main} provides us with a unique, global, bounded, nonnegative classical solution $(u, v, w, z)$ of \eqref{prob} (with $f$ and $g$ replaced by $\tilde f$ and $\tilde g$)
  and $C > 0$ such that \eqref{eq:main:bound} holds.
  In particular, if $M \le \frac{1}{C}$, then $u, v \le 1$ so that in that case $(u, v, w, z)$ also solves \eqref{prob} with $f(s) = g(s) = s$ for $s \ge 0$.
\end{proof}

\section{Smooth steady states}\label{sec:steady_states}
While for all $M_1, M_2 > 0$ the tuple $((M_1, M_2, f(M_2), g(M_1))$ forms a smooth steady state of \eqref{prob},
the conservation of mass for the first two solution components implies that these steady states may appear as limits only for the choices $M_1 = \ol u_0$ and $M_2 = \ol v_0$.
In Subsection~\ref{sec:conv_small_data}, we prove Theorem~\ref{th:conv}; that is, that this equilibrium indeed attracts solutions whenever the initial data are small.

The asymptotic behavior of solutions not covered by Theorem~\ref{th:conv} will be a main aspect of our numerical experiments performed in Section~\ref{simulation} below.
Analytically, we can at least rule out stabilization towards nonconstant smooth steady states: We show that there are no such equilibria in Subsection~\ref{sec:no_smooth_ss} .

\subsection{Convergence to homogeneous steady states for small data: proof of Theorem~\ref{th:conv}}\label{sec:conv_small_data}
Theorem~\ref{th:main} already contains a key step of the convergence proof, namely the fact that smallness of $u_0$ and $v_0$ implies smallness of $u$ and $v$ for all times.
With these estimates at hand, we can show that
\begin{align*}
  y \defs c \intom (u - \ol u_0)^2 + c \intom (v - \ol v_0)^2 + \intom (w - f(\ol v_0))^2 + \intom |\nabla w|^2 + \intom (z - g(\ol u_0))^2 + \intom |\nabla z|^2 
\end{align*}
(for suitably chosen $c > 0$) is a subsolution of a homogeneous ODE of the form $y' = -c' y$ for some $c' > 0$ and hence converges exponentially fast to $0$.
\begin{proof}[Proof of Theorem~\ref{th:conv}]
  We fix $\Omega$, $D_u$, $D_v$, $\chi_u$, $\chi_v$, $f$, $g$, $\alpha$, $w_0$ and $z_0$ as in the statement of Theorem~\ref{th:conv}.
  By the Poincar\'e inequality, there is $\newlc{poincare} > 0$ such that
  \begin{align}\label{eq:conv:poincare}
    \intom (\varphi - \ol \varphi)^2 \le \lc{poincare} \intom |\nabla \varphi|^2
    \qquad \text{for all $\varphi \in \sob12$},
  \end{align}
  and due to the assumptions on $f$ and $g$, there is $\newlc{f_g_c1_bdd} > 0$ such that
  \begin{align}\label{eq:conv:f'_g'_bdd}
    \|f'\|_{C^0([0, 1])} \le \lc{f_g_c1_bdd}
    \quad \text{and} \quad
    \|g'\|_{C^0([0, 1])} \le \lc{f_g_c1_bdd}.
  \end{align}
  As in the proof of Corollary~\ref{cor:ex_small_data} we fix $\zeta \in C^\infty(\Ombar)$ with $\zeta(s) = 1$ for $s \in [0, 1]$ and $\zeta \equiv 2$ in $[2, \infty)$.
  Since $\tilde f \defs \zeta f$ and $\tilde g \defs \zeta g$ fulfill \eqref{eq:intro:f_g},
  Theorem~\ref{th:main} asserts that there exists $\newlc{u_v_bdd} > 0$ such that for all $M > 0$ the following holds:
  If $u_0, v_0 \in \con{2+\alpha}$ are nonnegative and satisfy \eqref{eq:main:u0_v0} for some $M > 0$,
  there exists a global, nonnegative classical solution $(u, v, w, z)$ of \eqref{prob} (with $f, g$ replaced by $\tilde f, \tilde g$) with
  \begin{align}\label{eq:conv:u_v_bdd}
    \|u\|_{L^\infty(\Omega \times (0, \infty))} \le M \lc{u_v_bdd}
    \quad \text{and} \quad
    \|v\|_{L^\infty(\Omega \times (0, \infty))} \le M \lc{u_v_bdd}.
  \end{align}
  We choose $M > 0$ so small that
  \begin{align}\label{eq:conv:def_m}
    M^2 \lc{u_v_bdd}^2 \le \min\left\{\frac{1}{4 \lc{c_mult} \max\{\frac{\chi_u^2}{D_u}, \frac{\chi_v^2}{D_v}\}}, 1 \right\},
    \quad \text{where} \quad
    \newlc{c_mult} \defs \frac{3\lc{f_g_c1_bdd}^2 \max\{\lc{poincare}, 1\}}{\min\{D_u, D_v\}},
  \end{align}
  and fix nonnegative $u_0, v_0 \in \con{2+\alpha}$ fulfilling \eqref{eq:main:u0_v0} as well as the solution $(u, v, w, z)$ of \eqref{prob} (with $f, g$ replaced by $\tilde f, \tilde g$) given by Theorem~\ref{th:main}.
  We note that \eqref{eq:conv:u_v_bdd} and the second estimate contained in \eqref{eq:conv:def_m} imply $u, v \le 1$ in $\Ombarinf$.
  Thus, $(\tilde f(u), \tilde g(v)) = (f(u), g(v))$ in $\Ombarinf$ and \eqref{eq:conv:f'_g'_bdd} results in
  \begin{align}\label{eq:conv:f'_g'_bdd2}
    \max\{|f'(u)|,\, |g'(v)|\} \le \lc{f_g_c1_bdd}
    \qquad \text{in $\Ombarinf$}.
  \end{align}
  By testing the first equation in \eqref{prob} with $u - \ol u_0$ and making use of Young's inequality, \eqref{eq:conv:u_v_bdd}
  and \eqref{eq:conv:poincare} (we note that integrating the first equation implies $\intom u(\cdot, t) = \intom u_0$ for all $t \ge 0$, so that \eqref{eq:conv:poincare} is indeed applicable),
  we obtain 
  \begin{align*}
          \frac12 \ddt \intom (u - \ol u_0)^2
    &=    - D_u \intom |\nabla u|^2
          - \chi_u \intom u \nabla u \cdot \nabla w \\
    &\le  - \frac{3D_u}{4} \intom |\nabla u|^2
          + \frac{M^2 \lc{u_v_bdd}^2 \chi_u^2}{D_u} \intom |\nabla w|^2 \\
    &\le  - \frac{D_u}{2} \intom |\nabla u|^2
          - \frac{D_u}{4 \lc{poincare}} \intom (u - \ol u_0)^2
          + \frac{M^2 \lc{u_v_bdd}^2 \chi_u^2}{D_u}  \intom |\nabla w|^2
    \qquad \text{in $(0, \infty)$}.
  \end{align*}
  Moreover, testing the third equation in \eqref{prob}, $w_t = -w + f(v) = -(w - f(\ol v_0)) + f(v) - f(\ol v_0)$, with $w - f(\ol v_0)$
  and Young's inequality, the mean value theorem and \eqref{eq:conv:f'_g'_bdd2} yield
  \begin{align*}
          \frac12 \ddt \intom (w - f(\ol v_0))^2
    &=    - \intom (w - f(\ol v_0))^2 + \intom (f(v) - f(\ol v_0)) (w - f(\ol v_0)) \\
    &\le  - \frac12 \intom (w - f(\ol v_0))^2 + \frac{\lc{f_g_c1_bdd}^2}{2} \intom (v - \ol v_0)^2
    \qquad \text{in $(0, \infty)$},
  \end{align*}
  while testing the same equation with $-\Delta w$ gives
  \begin{align*}
        \frac12 \ddt \intom |\nabla w|^2
    &=  - \intom |\nabla w|^2
        + \intom f'(v) \nabla v \cdot \nabla w
    \le - \frac12 \intom |\nabla w|^2
        + \frac{\lc{f_g_c1_bdd}^2}{2} \intom |\nabla v|^2
    \qquad \text{in $(0, \infty)$}.
  \end{align*}
  Analogously,
  \begin{align*}
          \frac12 \ddt \intom (v - \ol v_0)^2
    &\le  - \frac{D_v}{2} \intom |\nabla v|^2
          - \frac{D_v}{4 \lc{poincare}} \intom (v - \ol v_0)^2
          + \frac{M^2 \lc{u_v_bdd}^2 \chi_v^2}{D_v} \intom |\nabla z|^2
  \end{align*}
  and
  \begin{align*}
          \frac12 \ddt \left( \intom (z - g(\ol u_0))^2 + \intom |\nabla z|^2 \right)
    &\le  - \frac12 \intom (z - g(\ol u_0))^2 + \frac{\lc{f_g_c1_bdd}^2}{2} \intom (u - \ol u_0)^2 
          - \frac12 \intom |\nabla z|^2 + \frac{\lc{f_g_c1_bdd}^2}{2} \intom |\nabla u|^2
  \end{align*}
  hold in $(0, \infty)$.
  Recalling \eqref{eq:conv:def_m}, we conclude that the function $y \colon [0, \infty) \to [0, \infty)$ defined by
  \begin{align*}
    y(t) \defs  
    \frac12 \left(
      \lc{c_mult} \intom (u - \ol u_0)^2
      + \lc{c_mult} \intom (v - \ol v_0)^2
      + \intom (w - f(\ol v_0))^2
      + \intom (z - g(\ol u_0))^2
      + \intom |\nabla w|^2
      + \intom |\nabla z|^2
    \right)
  \end{align*}
  for $t \ge 0$ fulfills
  \begin{align*}
          y'
    &\le  - \left( \frac{\min\{D_u, D_v\}\lc{c_mult}}{4\lc{poincare}} - \frac{\lc{f_g_c1_bdd}^2}{2} \right)  \left( \intom (u - \ol u_0)^2 + \intom (v - \ol v_0)^2 \right) \\
    &\pe  - \left( \frac{\min\{D_u, D_v\}\lc{c_mult}}{2} - \frac{\lc{f_g_c1_bdd}^2}{2} \right) \left( \intom |\nabla u|^2 + \intom |\nabla v|^2 \right)  \\
    &\pe  - \frac12 \left( \intom (w - f(\ol v_0))^2 + \intom (z - g(\ol u_0))^2 \right) \\
    &\pe  - \left( \frac{1}{2} - M^2 \lc{u_v_bdd}^2  \lc{c_mult} \max\left\{\frac{\chi_u^2}{D_u}, \frac{\chi_v^2}{D_v}\right\} \right) \left( \intom |\nabla w|^2 + \intom |\nabla z|^2 \right) \\
    &\le  - \lc{exponent} y
    \qquad \text{in $(0, \infty)$},
  \end{align*}
  where $\newlc{exponent} \defs \min\{\frac{\lc{f_g_c1_bdd}^2}{2\lc{c_mult}}, \frac12\}$.
  Thus, $y(t) \le \ure^{-\lc{exponent} t} y(0) \to  0$ for $t \to \infty$.
  This entails \eqref{eq:conv:uv} and \eqref{eq:conv:wz} for $p = 2$,
  upon which the statements for $p \in (2, \infty)$ follow from \eqref{eq:conv:u_v_bdd}
  and the interpolation inequality $\|\varphi\|_{\leb p} \le \|\varphi\|_{\leb \infty}^\frac{p-2}{p} \|\varphi\|_{\leb2}^\frac{2}{p}$, valid for all $\varphi \in \leb\infty$.
\end{proof}

\subsection{Lack of smooth heterogeneous steady states}\label{sec:no_smooth_ss}
We now show that all smooth steady states of \eqref{prob}, that is, solutions to \eqref{prob:ss}, are spatially homogeneous,
provided $f$ and $g$ belong to $\bigcup_{\eps > 0} C^\omega((-\eps, \infty); [0, \infty))$.
In a rather straightforward manner, this result can be extended to wider classes of functions $f$ and $g$,
but as the prototypical choices mentioned in the introduction are covered by Proposition~\ref{prop:no_hetero_ss},
we confine ourselves to analytical functions $f$ and $g$, for which the proof is particularly short.
\begin{proof}[Proof of Proposition~\ref{prop:no_hetero_ss}]
  Inserting the third and fourth equation in \eqref{prob:ss} into the first two and the last equations therein yields
  \begin{align*}
    \begin{cases}
      0 = \Delta u + \xi_u \nabla \cdot (u \nabla f(v)) & \text{in $\Omega$}, \\
      0 = \Delta v + \xi_v \nabla \cdot (v \nabla g(u)) & \text{in $\Omega$}, \\
      \partial_\nu u + \xi_u u \partial_\nu f(v) = 0    & \text{on $\partial \Omega$}, \\
      \partial_\nu v + \xi_v v \partial_\nu g(u) = 0    & \text{on $\partial \Omega$},
    \end{cases}
  \end{align*}
  where we have again set $\xi_u \defs \frac{\chi_u}{D_u}$ and $\xi_v \defs \frac{\chi_v}{D_v}$.
  Due to the supposed regularity of $u$ and $v$, \cite[Lemma~4.1]{BraukhoffLankeitStationarySolutionsChemotaxisconsumption2019} asserts that there are $\alpha, \beta \in \R$ such that
  \begin{align}\label{eq:no_hetero_ss:u_v_formula}
    u = \alpha \ure^{-\xi_u f(v)}
    \quad \text{and} \quad
    v = \beta \ure^{-\xi_v g(u)}
    \qquad \text{in $\Ombar$},
  \end{align}
  so that in particular
  \begin{align}\label{eq:no_hetero_ss:u_formula}
    u = \alpha \ure^{-\xi_u f(\beta \ure^{-\xi_v g(u)})}
    \qquad \text{in $\Ombar$}.
  \end{align}
  Both $(-\eps, \infty) \ni s \mapsto \alpha \ure^{-\xi_u f(\beta \ure^{-\xi_v g(s)})}$ and $(-\eps, \infty) \ni s \mapsto s$
  are analytical functions on $(-\eps, \infty)$ for some $\eps > 0$.
  As the former only vanishes at $0$ (and then everywhere) if $\alpha = 0$, these functions differ.
  Therefore, the set $S = \{\, s \in (-\eps, \infty) \mid s = \alpha \ure^{-\xi_u f(\beta \ure^{-\xi_v g(s)})} \,\}$ is discrete by the identity theorem.
  According to \eqref{eq:no_hetero_ss:u_formula}, the image of the continuous function $u \colon \Ombar \to \R$ is contained in $S$, which is only possible if $u$ is constant.
  Recalling \eqref{eq:no_hetero_ss:u_v_formula}, we conclude that $v$ is constant,
  whenceupon the third and fourth equations in \eqref{prob:ss} direct imply that also $w$ and $z$ are constant.
\end{proof}

\section{Numerical method}
\label{Numerics}
\subsection{Galerkin approximation}
\label{Galerkin}
In this subsection, we will give the details on the implicit finite element discretization for solving the problem (\ref{prob}) numerically. A finite element discretization of the system (\ref{prob}) is based on its weak formulation, which reads: Find $u, v, w, z \in L^\infty(0, T; L^2(\Omega))\cap L^2(0, T; H^1(\Omega))$ with $u_t, v_t, w_t, z_t \in L^2(0, T; (H^1(\Omega))^\star)$ for given $u^0, v^0, w^0, z^0 \in L^2(\Omega)$ such that a.e.\ in $(0, \infty)$ we have
\begin{equation}
\begin{aligned} \label{weak}
   & \langle u_t,\psi \rangle = -D_u \int_{\Omega} \nabla u \cdot \nabla \psi dx - \chi_u \int_{\Omega} u\nabla w \cdot \nabla \psi dx, \\
    &  \langle v_t,\psi \rangle  = -D_v \int_{\Omega} \nabla v \cdot \nabla \psi dx - \chi_v \int_{\Omega} v\nabla z \cdot \nabla \psi dx, \\
    &  \langle w_t,\psi \rangle  = -\int_{\Omega} w\, \psi\, dx + \int_{\Omega} f(v)\, \psi \, dx , \\
    &  \langle z_t,\psi \rangle  = -\int_{\Omega} z\, \psi\, dx + \int_{\Omega} g(u)\, \psi\, dx ,
\end{aligned}
\end{equation}
for all $\psi \in \con\infty$. Here, $ \langle \cdot , \cdot \rangle $ represents the duality pairing between $(H^1(\Omega))^\star$ and $H^{1}(\Omega)$.

To define a finite element discretization of problem (\ref{prob}), we first consider ${\mathcal T}_h$, a uniformly regular triangulation of $\Omega$ with a mesh size $h$. Then, we construct the finite element space $X_h$ consisting of continuous piecewise multi-linear functions
 \[ X_h = \{\phi \in H^1 (\Omega) ; \phi |_K \in Q_1(K) , \forall K \in {\mathcal T}_h\},\, \quad Z_h = X_h \cap H_0^1(\Omega),\, \]
with basis functions $\psi_{j}$, $j = 1,\dots,M$, such that $X_h = \operatorname{span} \{\psi_{j}\}$, where $M$ is the number of degrees of freedom and $Q_1$ is a space consisting of piecewise multi-linear functions . Any function $u_h \in X_h$ can be written in a unique way with respect to these basis functions as
\begin{equation*}
    u_h = \sum_{j=1}^M\,u_j\,\psi_j\,,
\end{equation*}
and hence it can be identified with the coefficient vector $\mathbf{u} = (u_1, \dots, u_M)$; $v_h, w_h,$ and $z_h$ can be defined similarly. Next, let $0=t_0<t_1<\dots<t_N=T$ be an equidistant decomposition of the time interval $[0, T]$ with
$\Delta t = t^{n+1} - t^{n}$, $n=0,\dots,N-1$. We use $u_h^{n}$, $v_h^{n}$, $w_h^{n}$, $z_h^{n} \in X_h$ to denote the approximation of the solutions at each time level $t^{n}$. Furthermore, an important feature of the considered system consists in the nonlinear terms and coupling between the equations, so that a fully implicit discretization leads to a coupled nonlinear algebraic system. We compute the solution of this nonlinear problem using fixed-point iterations. As a result, after applying the usual approach for deriving a Galerkin finite element scheme for space discretization, considering the $\theta$-method for time discretization, and using a fixed-point iteration to treat the nonlinear terms in the system the linearized  algebraic form corresponding to the system (\ref{weak}) reads as:
\allowdisplaybreaks
\begin{align}
   &({\mathbb M}+\theta\,\Delta t\,{\mathbb A}_{k-1}^{n+1,u})\,\mathbf{u}_{k}^{n+1}=
   ({\mathbb M}-(1-\theta)\,\Delta t\,{\mathbb A}^{n,u})\,\mathbf{u}^n\,, \label{eq1u} \\
   &({\mathbb M}+\theta\,\Delta t\,{\mathbb A}_{k-1}^{n+1,v})\,\mathbf{v}_{k}^{n+1}=
   ({\mathbb M}-(1-\theta)\,\Delta t\,{\mathbb A}^{n,v})\,\mathbf{v}^n\,, \label{eq1v} \\
   & (1+\theta\, \Delta t)\,\mathbb M\,\mathbf{w}_{k}^{n+1} =
      (1 - (1-\theta)\, \Delta t)\, \mathbb M\,\mathbf{w}^{n}
      + \Delta t\,(\theta \, G_{k}^{n+1,v} + (1-\theta)\, G^{n,v})\,, \label{eq1w} \\
   & (1+\theta\, \Delta t)\,\mathbb M\,\mathbf{z}_{k}^{n+1} =
      (1- (1-\theta)\, \Delta t)\,\mathbb M\,\mathbf{z}^{n}
      + \Delta t\,(\theta \, G_{k}^{n+1,u} + (1-\theta)\, G^{n,u})\,, \label{eq1z}
\end{align}
for $\theta \in [0, 1]$, where
$\mathbf{u}_{k}^{n+1} = (u_{j,k}^{n+1})_{j=1}^{M}$,
$\mathbf{v}_{k}^{n+1} = (v_{j,k}^{n+1})_{j=1}^{M}$,
$\mathbf{w}_{k}^{n+1} = (w_{j,k}^{n+1})_{j=1}^{M}$,
and
$\mathbf{z}_{k}^{n+1} = (z_{j,k}^{n+1})_{j=1}^{M}$ denote the vectors of unknowns at time level $t^{n+1}$ and iteration step $k \,, k= 1,\ldots,$ and 
$\mathbf{u}^{n} = (u_{j}^{n})_{j=1}^{M}$,
$\mathbf{v}^{n} = (v_{j}^{n})_{j=1}^{M}$,
$\mathbf{w}^{n} = (w_{j}^{n})_{j=1}^{M}$,
and 
$\mathbf{z}^{n} = (z_{j}^{n})_{j=1}^{M}$ are the known solutions from the previous time level $t^{n}$. Setting $\mathbf{u}_{0}^{n+1} = \mathbf{u}^{n}$, $\mathbf{v}_{0}^{n+1} = \mathbf{v}^{n}$, $\mathbf{w}_{0}^{n+1} = \mathbf{w}^{n}$, and $\mathbf{z}_{0}^{n+1} = \mathbf{z}^{n}$ with $\mathbf{u}^{0} = \mathbf{u}(x, 0)$, $\mathbf{v}^{0} = \mathbf{v}(x, 0)$, $\mathbf{w}^{0} = \mathbf{w}(x, 0)$, and $\mathbf{z}^{0} = \mathbf{z}(x, 0)$, the entries of the mass matrix, stiffness matrices and vectors above are given by
\begin{align*}
&\mathbb{M}_{ij} = \int_{\Omega} \psi_j\,\psi_i\, \,,\\
& \mathbb{A}_{ij,k-1}^{n+1,u} = D_u \int_{\Omega} \nabla \psi_j \cdot \nabla \psi_i 
                              + \chi_u \int_{\Omega} \psi_j \nabla w_{h,k-1}^{n+1} \cdot         \nabla \psi_i \,, \\
& \mathbb{A}_{ij}^{n,u} = D_u \int_{\Omega} \nabla \psi_j \cdot \nabla \psi_i
                              + \chi_u \int_{\Omega} \psi_j \nabla w_{h}^{n} \cdot         \nabla \psi_i \,, \\
& \mathbb{A}_{ij,k-1}^{n+1,v} = D_v \int_{\Omega} \nabla \psi_j \cdot \nabla \psi_i
                              + \chi_v \int_{\Omega} \psi_j \nabla z_{h,k-1}^{n+1} \cdot         \nabla \psi_i, \\
& \mathbb{A}_{ij}^{n,v} = D_v \int_{\Omega} \nabla \psi_j \cdot \nabla \psi_i
                              + \chi_v \int_{\Omega} \psi_j \nabla z_{h}^{n} \cdot         \nabla \psi_i\,, \\
& G_{i,k}^{n+1,v} = \int_{\Omega} f(v_{h,k}^{n+1})\, \psi_i, \\
& G_{i}^{n,v} = \int_{\Omega} f(v_{h}^{n})\, \psi_i, \\
& G_{i,k}^{n+1,u} = \int_{\Omega} g(u_{h,k}^{n+1})\, \psi_i,\\
& G_{i}^{n,u} = \int_{\Omega} g(u_{h}^{n})\, \psi_i
\end{align*}
for $i,j = 1, \dots, M$ where $\psi_i \in Z_h$ and $u_{h,k}^{n+1}$, $v_{h,k}^{n+1}$, $w_{h,k}^{n+1}$, $z_{h,k}^{n+1}$ denote the fully discrete solution functions at fixed point step $k=1,2,\ldots .$

\subsection{Positivity-preserving flux-corrected transport scheme}
\label{FCT}
Since all components of the system (\ref{prob}) represent densities of populations or graffiti, they should be nonnegative, and Lemma~\ref{lm:local_ex} asserts that the analytical solutions considered in the first part of the present paper are indeed nonnegative. Therefore, the numerical solutions of the model problem must also be nonnegative in order to satisfy the physics behind the system. Thus, the numerical methods should be constructed in such a way that the qualitative properties of the exact solutions are preserved. The standard Galerkin discretization described in the previous subsection usually produces oscillatory nonpositive solutions, which leads to numerical instabilities, especially when the convective part of the system is dominant. Therefore, a stabilization has to be applied. One appropriate possibility is to modify the algebraic system resulting from the Galerkin discretization, which will be discussed in the following.

The numerical behavior of the gang concentrations $u$ and $v$ heavily depends on the properties of the matrices on the left- and right-hand sides of (\ref{eq1u}) and (\ref{eq1v}). In the following, we will introduce a positivity-preserving flux-corrected transport (FCT) scheme following the work of Kuzmin \cite{Kuz09, Kuz12a, KT02} which can preserve the positivity of the concentrations $u$ at all times, the positivity of $v$ can be obtained similarly by repeating the same process.
\begin{definition}
A matrix $\mathbb{A}$ is called a Z-matrix if it has only nonpositive off-diagonal entries,
monotone if $\mathbb{A}^{-1} \geq 0$
and an M-matrix if it is a monotone Z-matrix.
\end{definition}
\begin{lemma}\label{lem:lin_syst}
Consider a fully discrete system of the form 
\begin{equation}\label{eq:general_system}
   {\mathbb B}\,\mathbf{u}^{n+1}
   ={\mathbb K}\,\mathbf{u}^n
\end{equation}
and suppose that the coefficients of  ${\mathbb B}=(b_{ij})_{i,j=1}^M$ and 
${\mathbb K}=(k_{ij})_{i,j=1}^M$ satisfy
\begin{equation*}
  b_{ii}\ge0\,, \quad k_{ii}\ge0\,, \quad 
  b_{ij}\leqslant 0\,, \quad k_{ij} \geqslant 0,
  \qquad\forall\,\,i,j=1,\dots,M\,,\,\,i\neq j\,.
\end{equation*}
If $\mathbb{B}$ is strictly or irreducibly diagonally dominant, then it is an M-matrix and
\begin{enumerate}
\item[1.] the scheme (\ref{eq:general_system}) is globally positivity-preserving, i.e., $\mathbf{u}^{n+1}\ge0$ if $\mathbf{u}^n\ge0$, 
\item[2.] the global discrete maximum principle (DMP) holds if $\sum_{j} b_{ij} = \sum_{j} k_{ij}\; \forall i=1,\ldots,M$, i.e.,
 \begin{align*}
(\min \mathbf{u}^{n})^{-} \leq \mathbf{u}^{n+1} \leq (\max \mathbf{u}^{n})^{+},
\end{align*}
where $(\max \mathbf{u}^n)^{+} = \max \{0, \mathbf{u}^n\}$ and $(\min \mathbf{u}^{n})^{-} = \min \{0, \mathbf{u}^n\}$, $n=0,\ldots, N-1.$
\end{enumerate}
\end{lemma}
\begin{proof}
See \cite[Algebraic flux correction I, Theorem~4]{Kuz2012}.
\end{proof}
For the system (\ref{eq1u}), which can be written as (\ref{eq:general_system}), the above properties do not hold since the mass matrix $\mathbb{M}$ may contain some nonnegative off-diagonal entries and also some positive off-diagonal entries might appear in the stiffness matrices. Therefore, as a remedy, following the work of Kuzmin \cite{Kuz09, Kuz12a, KT02}, we not only replace the mass matrix $\mathbb{M}$ by its diagonal counterpart, the lump matrix $\mathbb{M}_L$,
\[\mathbb{M}_L = \operatorname{diag}(m_1,\cdots,m_M),\, \quad m_i=\sum_{j=1}^M m_{ij},\, \quad i=1,\cdots,M,\] but also add symmetric artificial diffusion matrices $\mathbb{D}_{k-1}^{n+1,u} = (d_{ij,k-1}^{n+1,u})_{i,j=1}^{M}$ and $\mathbb{D}^{n,u} = (d_{ij}^{n,u})_{i,j=1}^{M}$ to the stiffness matrices $\mathbb{A}_{k-1}^{n+1,u}$ and $\mathbb{A}^{n,u}$ to eliminate their nonnegative off-diagonal entries as:
\begin{equation*}
  d^{n+1,u}_{ij,k-1} = -\max\{a_{ij,k-1}^{n+1,u},0,a_{ji,k-1}^{n+1,u}\}\quad
  \mbox{for} \ i\neq j\,,\quad
  d^{n+1,u}_{ii,k-1}=-\sum_{j=1,j \neq i}^M d^{n+1,u}_{ij,k-1}\,,\quad n=0,\cdots,N-1 \,, \quad k=1,2,\ldots 
\end{equation*}
and 
\begin{equation*}
  d^{n,u}_{ij}=-\max\{a^{n,u}_{ij},0,a^{n,u}_{ji}\}\quad
  \mbox{for} \ i\neq j\,,\qquad\quad
  d^{n,u}_{ii}=-\sum_{j=1,j \neq i}^M d^{n,u}_{ij}\,, \quad n=0,\cdots,N-1. 
\end{equation*}

Denoting $\mathbb{\tilde{A}}_{k-1}^{n+1,u} = \mathbb{A}_{k-1}^{n+1,u} + \mathbb{D}_{k-1}^{n+1,u}$ and $\mathbb{\tilde{A}}^{n,u} = \mathbb{A}^{n,u} + \mathbb{D}^{n,u}$, the low-order form of the (\ref{eq1u}) can be written as:
\begin{equation}\label{eq1ulow}
   (\mathbb{M}_{L}+\theta\,\Delta t\,\mathbb{\tilde{A}}_{k-1}^{n+1,u})\,\mathbf{u}_{k}^{n+1}=
   (\mathbb{M}_{L}-(1-\theta)\,\Delta t\,{\mathbb{\tilde{A}}}^{n,u})\,\mathbf{u}^n\,,\qquad
   n=0,\dots,N-1\,.
\end{equation}

\begin{lemma}\label{lem:time_step}
Let the time step $\Delta t$ satisfy
\begin{equation}\label{eq:time_step}
    m_i - (1 - \theta)\, \Delta t\, \tilde{a}_{ii}^{n,u} \geq 0\,,\qquad
    m_i + \theta\, \Delta t\, \sum_{j=1}^{M} a_{ij,k-1}^{n+1,u} > 0, \,\, \quad n=0,\cdots,N-1 \,, \quad k=1,2,\ldots ,
\end{equation}
then the low-order scheme (\ref{eq1ulow}) is positivity-preserving.
\end{lemma}
\begin{proof}
Denoting ${\mathbb K} = (\mathbb{M}_{L}-(1-\theta)\,\Delta t\,{\mathbb{\tilde{A}}}^{n,u})$, since $\mathbb{M}_{L}$ is diagonal and ${\mathbb{\tilde{A}}}^{n,u}$ is a Z-matrix the off-diagonal entries of ${\mathbb K}$ are non-negative and it also has non-negative diagonal entries if the first condition in (\ref{eq:time_step}) is satisfied, thus ${\mathbb K} \geq 0$. Now, set ${\mathbb B} = (\mathbb{M}_{L}+\theta\,\Delta t\,\mathbb{\tilde{A}}_{k-1}^{n+1,u})$, since $ \mathbb{D}_{k-1}^{n+1,u}$ has zero row sum, if the second condition in (\ref{eq:time_step}) holds we can conclude that ${\mathbb B}$ is strictly diagonally dominant and hence nonsingular. Furthermore ${\mathbb B} $ is a matrix of non-negative type hence it is an M-matrix, therefore according to Lemma~\ref{lem:lin_syst}, (\ref{eq1ulow}) is positivity-preserving.
\end{proof}

\begin{remark}
The scheme (\ref{eq1ulow}) can be simplified as
\begin{equation*}
   {\mathbb B}\,\mathbf{u}^{n+1}
   ={\mathbb K}\,\mathbf{u}^n,
\end{equation*}
where $\mathbb{B}, \mathbb{K} \in \mathbb{R}^{M\times M}$ and $\mathbf{u}^{n+1}, \mathbf{u}^{n} \in \mathbb{R}^{M} $. Assuming that 
$$\mathbb{B}^{-1}\geq 0, \quad \mathbb{K} \geq 0, \quad \mathbb{B}\, \mathbb{I}_M\, \geq \mathbb{K}\, \mathbb{I}_M,$$
where $\mathbb{I}_M$ denotes a vector with all entries equal to $1$, one obtains with $M= (\max \mathbf{u}^n)^{+} $ and $m = (\min \mathbf{u}^{n})^{-} $, $n=0,\ldots, N-1$ that
\begin{align*}
\mathbf{u}^{n+1} = \mathbb{B}^{-1}\, \mathbb{K}\, \mathbf{u}^{n}\, \leq M\, \mathbb{B}^{-1}\, \mathbb{K}\, \mathbb{I}_M\, \leq M\, \mathbb{B}^{-1}\, \mathbb{B}\, \mathbb{I}_M\, = M\, \mathbb{I}_M,\\
\mathbf{u}^{n+1} = \mathbb{B}^{-1}\, \mathbb{K}\, \mathbf{u}^{n}\, \geq m\, \mathbb{B}^{-1}\, \mathbb{K}\, \mathbb{I}_M\, \geq m\, \mathbb{B}^{-1}\, \mathbb{B}\, \mathbb{I}_M\, = m\, \mathbb{I}_M\,,
\end{align*}
thus that
\begin{align*}
(\min \mathbf{u}^{n})^{-} \leq \mathbf{u}^{n+1} \leq (\max \mathbf{u}^{n})^{+}, \quad i=1, ..., M,
\end{align*}
i.e., that the DMP is satisfied. We would like to notice that due to the construction and following the Lemma~\ref{lem:time_step} the conditions $\mathbb{B}^{-1}\geq 0$ ($\mathbb{B}$ is an M-matrix and thus invertible) and $ \mathbb{K} \geq 0$ already hold. However, without further assumptions on the matrices $\mathbb{\tilde{A}}_{k-1}^{n+1,u}$ and $\mathbb{\tilde{A}}^{n,u}$, proving the last assumption above (which is a replacement of the row-sum property in Lemma~\ref{lem:lin_syst}) is very challenging, see \cite{BJK23} for more information. 
\end{remark}

Although, the solution of (\ref{eq1ulow}) suppresses the spurious oscillations and does not produce negative solutions, it often becomes inaccurate since the amount of the added artificial diffusion is usually too large. Therefore, the idea of FEM-FCT is to modify the right-hand side of (\ref{eq1ulow}) in such a way that the solutions become less diffusive in the smooth regions while their positivity is still conserved. By construction, the difference between (\ref{eq1u}) and (\ref{eq1ulow}) reads
\begin{equation}\label{eq:recover1}
   (\mathbb{M}_L-{\mathbb M})(\mathbf{u}_{k}^{n+1}-\mathbf{u}^n)
   +\theta\,\Delta t\,{\mathbb D}_{k-1}^{n+1,u}\,\mathbf{u}_{k}^{n+1}
   +(1-\theta)\,\Delta t\,{\mathbb D}^{n,u}\,\mathbf{u}^n\,, \quad n=0,\cdots,N-1 \,, \quad k=1,2,\ldots,
\end{equation}
which is nonlinear since it depends on the approximate solution $\mathbf{u}_{k}^{n+1}$.
To be aligned with treating nonlinearities in the equation \eqref{eq1u} using fixed point iteration, we replace $u_{k}^{n+1}$ by $u_{k-1}^{n+1}$, which leads to
\begin{equation*}
   \overline{\mathbf{f}}_{k-1}^{n+1,u} \defs
   (\mathbb{M}_L-{\mathbb M})(\mathbf{u}_{k-1}^{n+1}-\mathbf{u}^n)
   +\theta\,\Delta t\,{\mathbb D}_{k-1}^{n+1,u}\,\mathbf{u}_{k-1}^{n+1}
   +(1-\theta)\,\Delta t\,{\mathbb D}^{n,u}\,\mathbf{u}^n\,, \quad n=0,\cdots,N-1 \,, 
\end{equation*}
which admits a decomposition into a sum of discrete internodal fluxes 
\begin{equation}\label{eq:f-decompose}
   \overline{\mathbf{f}}_{k-1}^{n+1,u}= \sum_{j=1}^{M} f_{ij,k-1}^{n+1,u},\, \quad f_{ij,k-1}^{n+1,u} = -f_{ji,k-1}^{n+1,u},\, \quad i,j=1,\ldots,M,
\end{equation}
since the matrices ${\mathbb D}_{k-1}^{n+1,u}, {\mathbb D}^{n,u}$ and $(\mathbb{M}_L-{\mathbb M})$ are symmetric and have zero row sums. Moreover, the so-called algebraic fluxes $f_{ij,k-1}^{n+1,u}$ are given by
\begin{equation*}
   f^{n+1,u}_{ij,k-1}=\big(-m_{ij}+\theta\,\Delta t\,d^{n+1,u}_{ij,k-1}\big)
   (u_{j,k-1}^{n+1}-u_{i,k-1}^{n+1})
   +\big(m_{ij}+(1-\theta)\,\Delta t\,d^{n,u}_{ij}\big)(u^{n}_j-u^{n}_i).
\end{equation*}
Now, the amount of algebraic fluxes inserted into each node must be limited by a limiter $\alpha_{ij,k-1}^{n+1,u} \in [0, 1]$ in such a way that in addition to keeping the numerical solutions positive, it also controls the amount of added artificial diffusion especially in the regions where the solutions are smooth and well-resolved where $\alpha_{ij,k-1}^{n+1,u} = 1$ is appropriate, therefore one replaces (\ref{eq:f-decompose}) by
\begin{equation*}
   {\mathbf{f}}_{k-1}^{n+1,u}= \sum_{j=1}^{M} \alpha_{ij,k-1}^{n+1,u} f_{ij,k-1}^{n+1,u},\, \quad \alpha_{ij,k-1}^{n+1,u} = \alpha_{ji,k-1}^{n+1,u},\, \quad i,j=1,\ldots,M,
\end{equation*}
this leads to 
\begin{equation}\label{eq1uhigh}
   (\mathbb{M}_{L}+\theta\,\Delta t\,\mathbb{\tilde{A}}_{k-1}^{n+1,u})\,\mathbf{u}_{k}^{n+1}=
   (\mathbb{M}_{L}-(1-\theta)\,\Delta t\,{\mathbb{\tilde{A}}}^{n,u})\,\mathbf{u}^n + {\mathbf{f}}_{k-1}^{n+1,u} \,,\qquad
   n=0,\dots,N-1\,, \quad k=1,2,\ldots 
\end{equation}
and it can be rewritten in the form
\begin{eqnarray*}
   \mathbb{M}_L\,\overline{\mathbf{u}}
   &=&(\mathbb{M}_L-(1-\theta)\,\Delta t\,{\mathbb{\tilde{A}}^{n,u}})\,\mathbf{u}^n\,,
   \\
   \mathbb{M}_L\,\tilde{\mathbf{u}}&=&\mathbb{M}_L\,\overline{\mathbf{u}} + \mathbf{f}_{k-1}^{n+1,u}\,,
   \\
   (\mathbb{M}_L+\theta\,\Delta t\,{\mathbb{\tilde{A}}_{k-1}^{n+1,u}})\,\mathbf{u}_{k}^{n+1}
   &=&\mathbb{M}_L\,\tilde{\mathbf{u}}\,.
\end{eqnarray*}
which is a high-resolution finite-element scheme. Moreover, it is positivity-preserving under the conditions (\ref{eq:time_step}) and for appropriate choice of flux limiters, see \cite{Hey2023}.

We use the limiting strategy based on Zalesak's algorithm \cite{Zal79} to determine an appropriate value of $\alpha_{ij,k-1}^{n+1,u}$. The limiting process begins with canceling all fluxes that are diffusive in nature and tend to flatten the solutions profile \cite{Kuz12a}. The required modification is:
\begin{equation*}
   f^{n+1,u}_{ij,k-1}\defs0\qquad\mbox{if}\,\,\,\,
   f^{n+1,u}_{ij,k-1}\,(\overline{u}_j-\overline{u}_i) > 0\,.
\end{equation*}
To define the limiter, we perform the following steps:
\begin{itemize}
\item[1.]
Compute the sum of positive/negative antidiffusive fluxes into node $i$,
\begin{equation}\label{eq:zalesak1}
   P_{i}^{+} = \sum_{j\in{\mathcal N}_i}\,\max \{0, {f}_{ij,k-1}^{n+1,u}\}\,,\qquad 
   P_{i}^{-} = \sum_{j\in{\mathcal N}_i}\,\min \{0, {f}_{ij,k-1}^{n+1,u}\}\,,
\end{equation}
where $\mathcal N_i$ is the set of the nearest neighbors of the node $i$. 
\item[2.]
Compute the distance to a local extremum of the auxiliary solution 
$\overline{\mathbf{u}}$ at the neighboring nodes that share an edge with the node $i$,
\begin{equation}\label{eq:zalesak2}
   Q_i^+=m_i\,(\overline u_i^{\mathrm{max}}-\overline u_i) \,,\qquad
   Q_i^-=m_i\,(\overline u_i^{\mathrm{min}}-\overline u_i) \,,
\end{equation}
with 
$   \overline u_i^{\mathrm{min}} = \min_{j\in{\mathcal N}_i\cup\{i\}} \overline u_j\,\,,
   \overline u_i^{\mathrm{max}} = \max_{j\in{\mathcal N}_i\cup\{i\}} \overline u_j\,\,, i=1,\dots,M\,.$
\item[3.] 
Compute the nodal correction factors for the net increment at the node $i$,
\begin{equation}\label{eq:zalesak3}
   R_i^+=\min\left\lbrace 1,\dfrac{Q_i^+}{ P_i^+}\right\rbrace,\qquad
   R_i^-=\min\left\lbrace 1,\dfrac{Q_i^-}{ P_i^-}\right\rbrace,
\end{equation}
if $P_i^+$ or $P_i^-$ vanishes then set $R_i^+ =1$ or $R_i^- =1$, respectively.
\item[4.]
Check the sign of the antidiffusive flux and apply the correction factor by
\begin{equation}\label{eq:zalesak4}
   \alpha_{ij} = \left\{
   \begin{array}{ll}
      \min\{R_{i} ^{+}, R_{j} ^{-} \}\quad&\mbox{if}\quad\,f^{n+1,u}_{ij,k-1}>0\,,\\[1mm]
      1&\mbox{if}\quad\,f^{n+1,u}_{ij,k-1} = 0\,, \\[1mm]
      \min\{R_{i} ^{-}, R_{j} ^{+} \}\quad&\mbox{if}\quad\,f^{n+1,u}_{ij,k-1}<0\,.
   \end{array}\right.
\end{equation}
\end{itemize}
Another way how to derive (\ref{eq1uhigh}) is to apply the FCT stabilization to the nonlinear problem obtained by applying the Galerkin discretization and the $\theta$-method before introducing the fixed-point iterations. The algebraic fluxes defined in this way depend on the unknown solution $u$ and hence the FCT stabilization introduces an additional nonlinearity. However, since the problem at hand is already nonlinear, both nonlinearities can be handled simultaneously using a fixed-point iteration. Moreover, the procedure explained above can also be used to obtain a high-resolution positivity-preserving approximation of $v$. 

We summarize the procedure above in the following algorithm. 
\begin{algorithm}[H]\caption{FEM-FCT}
\begin{algorithmic}[1]
   \FOR{time step $n \leftarrow 0, \cdots, N$}
   \FOR{iteration step $k=1,2,\ldots$}
   \STATE{\textbf{solve} for $\mathbf{u}_k^{n+1}$}
      \STATE{ \quad Solve $\mathbb{M}_L\,\overline{\mathbf{u}}
             =(\mathbb{M}_L-(1-\theta)\,\Delta t\,{\mathbb{\tilde{A}}^{n,u}})\,\mathbf{u}^n$ }
      \STATE{\quad Compute $\mathbf{f}_{k-1}^{n+1,u}$ using Zalesak's algorithm (\ref{eq:zalesak1})--(\ref{eq:zalesak4})}
      \STATE{\quad Solve $ \mathbb{M}_L\,\tilde{\mathbf{u}} = \mathbb{M}_L\,\overline{\mathbf{u}} +                      \mathbf{f}_{k-1}^{n+1,u}$}
      \STATE{\quad Solve $ (\mathbb{M}_L+\theta\,\Delta t\,{\mathbb{\tilde{A}}_{k-1}^{n+1,u}})\,   \mathbf{u}_{k}^{n+1}
   =\mathbb{M}_L\,\tilde{\mathbf{u}}$} 
      \STATE{\textbf{end solve}}
      
         \STATE{\textbf{solve} for $\mathbf{v}_k^{n+1}$}
      \STATE{\quad Solve  $\mathbb{M}_L\,\overline{\mathbf{v}}
             =(\mathbb{M}_L-(1-\theta)\,\Delta t\,{\mathbb{\tilde{A}}^{n,v}})\,\mathbf{v}^n$ }
      \STATE{\quad Compute $\mathbf{f}_{k-1}^{n+1,v}$ using Zalesak's algorithm (\ref{eq:zalesak1})--(\ref{eq:zalesak4})}
      \STATE{\quad Solve $ \mathbb{M}_L\,\tilde{\mathbf{v}}=\mathbb{M}_L\,\overline{\mathbf{v}} +                      \mathbf{f}_{k-1}^{n+1,v}$}
      \STATE{\quad Solve $ (\mathbb{M}_L+\theta\,\Delta t\,{\mathbb{\tilde{A}}_{k-1}^{n+1,v}})\,   \mathbf{v}_{k}^{n+1}
   =\mathbb{M}_L\,\tilde{\mathbf{v}}$} 
      \STATE{\textbf{end solve}}
      \STATE{\textbf{solve} for $\mathbf{w}_k^{n+1}$}
      \STATE{\quad Compute $\mathbf{w}_k^{n+1}$ from (\ref{eq1w})}
      \STATE{\textbf{end solve}}
      \STATE{\textbf{solve} for $\mathbf{z}_k^{n+1}$}
      \STATE{\quad Compute $\mathbf{z}_k^{n+1}$ from (\ref{eq1z})}
      \STATE{\textbf{end solve}}
      \ENDFOR
        \ENDFOR
\end{algorithmic}
\end{algorithm}

We iterate until the residual/difference between two successive solutions is less than a prescribed tolerance or until the maximum number of the iterations is reached, then set $\mathbf{u}^{n+1} = \mathbf{u}_{k}^{n+1}$, $\mathbf{v}^{n+1} = \mathbf{v}_{k}^{n+1}$, $\mathbf{w}^{n+1} = \mathbf{w}_{k}^{n+1}$, and $\mathbf{z}^{n+1} = \mathbf{z}_{k}^{n+1}$ and advance to the next time level. The system of the algebraic equations was solved using the direct solver UMFPACK \cite{Dav04} and for the implementation of our newly designed algorithm we used the deal.II library.

\section{Numerical simulations}
\label{simulation}
In this section, we present numerical experiments for solving the model problem (\ref{prob}) with $f(v)=\frac{v}{1+v}$ and $g(u)=\frac{u}{1+u}$, and compare our results with results from the literature \cite{AlsenafiBarbaroConvectionDiffusionModel2018, Alsenafi2021, BarbaroEtAlAnalysisCrossdiffusionModel2021} pertaining to related systems. The numerical simulations are computed on a square domain $\Omega = [-6, 6]^2$ that is discretized uniformly using quadrilateral elements with 5 levels of refinements which creates 1089 degrees of freedom, i.e., 33 grid points in each direction. In addition, conforming bilinear finite elements are used for all unknown variables. Furthermore, we use the time interval $[0, T]$, $T = 1000$, with a time step $\Delta t = 1.0$ in the Crank--Nicolson discretization method ($\theta$-scheme for $\theta = 0.5$). The setting above is fixed through all computations, unless otherwise mentioned.

\subsection{Can the gangs separate?}
\label{sep-study}
In this subsection, we investigate the possibility of segregation for the model problem (\ref{prob}). Thus, we consider different values of $D_u$, $D_v$, $\chi_u$, and $\chi_v$ in our simulations. The initial conditions used in the following examples are given by
\begin{equation}\label{eq:init_data_num1}
u^0(x,y) = 0.1 + \ure^{-(x-2)^2 - (y-2)^2}\,, \quad v^0(x,y) = 0.1 + \ure^{-(x+2)^2 - (y+2)^2}\,, \quad w^0(x,y) = 0\,, \quad z^0(x,y) = 0\,,
\end{equation}
see for instance Figure~\ref{fig1ab}(a) for an illustration of $u^0, v^0$. In our figures, we use the color red when the gang density $u$ presented in the discrete domain is larger than the gang density $v$ by at least $10^{-6}$ and dark blue color when the opposite inequality holds true. Locations where the difference between the gang densities is less than $10^{-6}$ are displayed in dark purple. Similarly and again with a cutoff of $10^{-6}$, we use orange if the amount of graffiti density $z$ is larger in a region, light blue for the opposite situation and light purple if the graffiti densities $z$ and $w$ are roughly the same.

Unlike as in the analytical part, we do not fix $D_u$, $D_v$, $\chi_u$, $\chi_v$ but on the contrary are interested in their effect on the dynamics.
However, if $(u, v, w, z)$ is a solution of \eqref{prob} with $f = g = \id$ and initial data $(u^0, v^0, w^0, z^0)$,
then $(Au, Bv, Bw, Az)$ solves \eqref{prob} with initial data $(Au^0, Bu^0, Bw^0, Az^0)$ and $(\chi_u, \chi_v)$ replaced by $(\frac{\chi_u}{B}, \frac{\chi_v}{A})$ for any $A, B > 0$.
That is, instead of considering larger initial data one may likewise consider larger $\chi_u, \chi_v$,
and one expects that for functions $f$ and $g$ with linear growth near $0$ these modifications have similar effects as well --
at least as long as the solutions remain comparatively small; if they become large this scaling argument may no longer be applicable.
In particular, Theorem~\ref{th:conv} suggests convergence towards homogeneous equilibria for small $\chi_u, \chi_v$ (and is inconclusive for large $\chi_u, \chi_v$).

\subsubsection{Convergence towards constant steady state}\label{sec:numeric_conv_sss}
To begin with, we first consider a case where $D_u=D_v=\chi_u=\chi_v=0.25$. In the top row of Figure~\ref{fig1ab}, we present the plot of the gang population densities $u$ and $v$ and in the bottom row their corresponding graffiti densities $z$ and $w$ computed using the standard Galerkin method (\ref{eq1u})--(\ref{eq1z}) over time. Figure~\ref{fig1ab}(a) corresponds to the gangs' initial conditions at $t=0$; evolving the time, we observe that the gangs start to spread inside the domain and mark their territories by spraying graffiti. By the time $t=5$, the whole upper triangle part of the domain is completely dominated by the gang $u$ and the lower triangle part is dominated by the gang $v$, this situation seems to continue till nearly $t=400$, after that we can see that the dark purple starts to merge from the middle section and continues to grow wider to the point that covers the whole domain at $t=700$ and remains the same by the end of the time, which means that the same amount of the gang densities $u$ and $v$ are presented in each location of the domain. The same happens for their corresponding graffiti $z$ and $w$ in the bottom row, where there is no amount of graffiti at the beginning (at =$t0)$, then they gradually increase  when the gangs start to mark their own territories and by the time $t=700$ each location is marked by the same amount of graffiti from each gang. It is evident from Figure~\ref{fig1ab} that segregation does not happen in this case. Before moving on to the next example, we examine this case more closely in Figures~\ref{fig1c}-\ref{fig1d}, by taking a snapshot along the line $y=x$ over time. Figure~\ref{fig1c}(a) shows the initial conditions at $t=0$, when the time evolves it seems that the approximate solutions converge toward constant steady-states already at $t=400$ in Figure~\ref{fig1c}(c), but the close up in Figure~\ref{fig1d}(a),(b) shows that the gang density $u$ and its corresponding graffiti $z$ are slightly larger in the first half of the domain $[-6,0)$ and slightly smaller in the second half $(0, 6]$ than the gang density $v$ and its corresponding graffiti $w$; however, they tend to stabilize at around $0.121817$ and $0.108588$, respectively, at time $t=700$ and remain the same through the rest of the time, which corresponds well to the results shown in Figure~\ref{fig1ab}.

This case has been studied for (\ref{prob:2eq}) in \citep{BarbaroEtAlAnalysisCrossdiffusionModel2021}, where the authors also show that the approximate solutions converge towards steady states; however they also report that segregation might not happen for that problem, at least not in the considered solution framework.

\begin{figure}[H]
	\centering
	\begin{subfigure}{0.155\textwidth}
		\includegraphics[width=\textwidth]{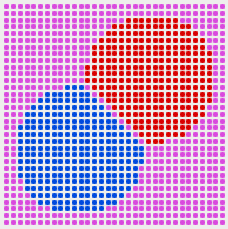}
		\caption{t = 0}
	\end{subfigure}
	\begin{subfigure}{0.155\textwidth}
		\includegraphics[width=\textwidth]{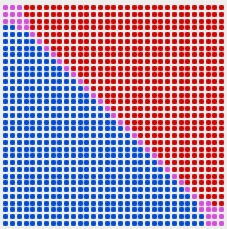}
		\caption{t = 400}
	\end{subfigure}
	\begin{subfigure}{0.155\textwidth}
		\includegraphics[width=\textwidth]{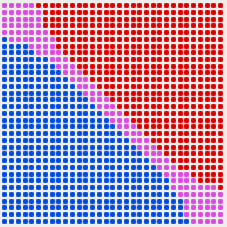}
		\caption{t = 500}
	\end{subfigure}
	\begin{subfigure}{0.155\textwidth}
		\includegraphics[width=\textwidth]{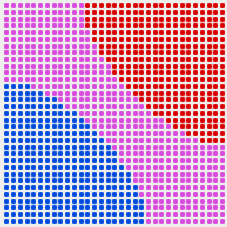}
		\caption{t = 600}
	\end{subfigure}
	\begin{subfigure}{0.155\textwidth}
		\includegraphics[width=\textwidth]{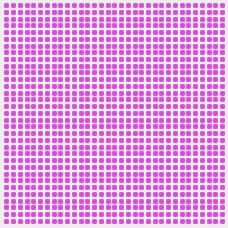}
		\caption{t = 700}
	\end{subfigure}
	\begin{subfigure}{0.155\textwidth}
		\includegraphics[width=\textwidth]{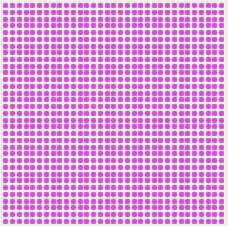}
		\caption{t = 1000}
	\end{subfigure}
	
		\begin{subfigure}{0.155\textwidth}
		\includegraphics[width=\textwidth]{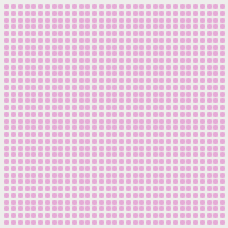}
		\caption{t = 0}
	\end{subfigure}
	\begin{subfigure}{0.155\textwidth}
		\includegraphics[width=\textwidth]{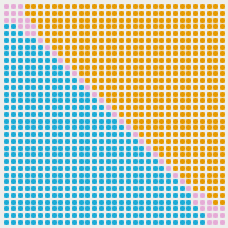}
		\caption{t = 400}
	\end{subfigure}
	\begin{subfigure}{0.155\textwidth}
		\includegraphics[width=\textwidth]{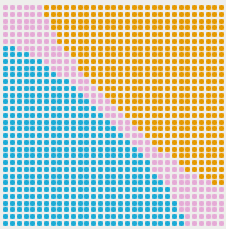}
		\caption{t = 500}
	\end{subfigure}
	\begin{subfigure}{0.155\textwidth}
		\includegraphics[width=\textwidth]{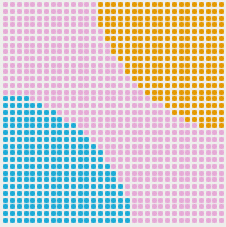}
		\caption{t = 600}
	\end{subfigure}
	\begin{subfigure}{0.155\textwidth}
		\includegraphics[width=\textwidth]{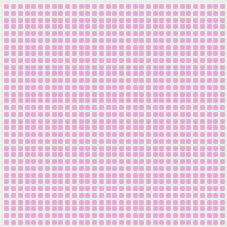}
		\caption{t = 700}
	\end{subfigure}
	\begin{subfigure}{0.155\textwidth}
		\includegraphics[width=\textwidth]{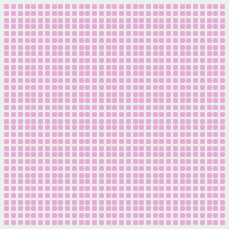}
		\caption{t = 1000}
	\end{subfigure}
	\caption{\small
	Numerical solutions for the model problem (\ref{prob}) obtained using the standard Galerkin method, at different time instant when $D_u=D_v=\chi_u=\chi_v =0.25$. For the choice of colors, see the beginning of Subsection~\ref{sep-study}.
	 }
	\label{fig1ab}
\end{figure} 

\begin{figure}[H]
	\centering
		\begin{subfigure}{0.155\textwidth}
		\includegraphics[width=\textwidth]{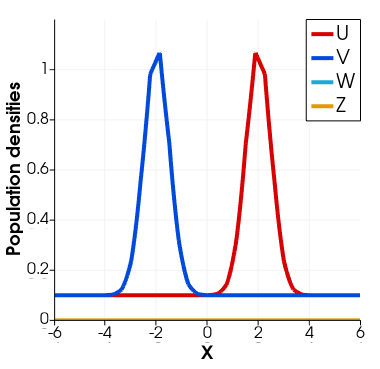}
		\caption{t = 0}
	\end{subfigure}
	\begin{subfigure}{0.155\textwidth}
		\includegraphics[width=\textwidth]{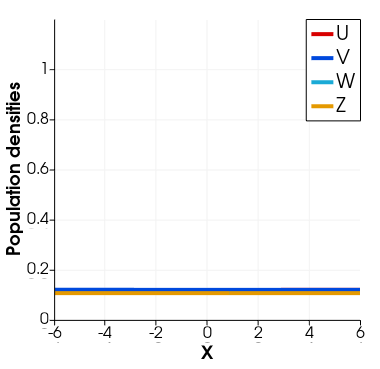}
		\caption{t = 400}
	\end{subfigure}
	\begin{subfigure}{0.155\textwidth}
		\includegraphics[width=\textwidth]{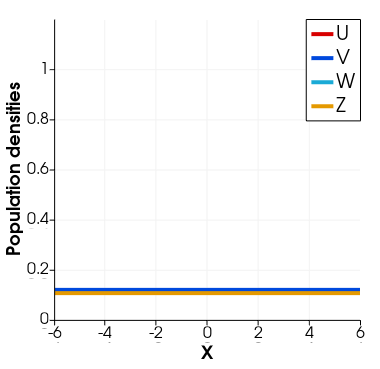}
		\caption{t = 500}
	\end{subfigure}
	\begin{subfigure}{0.155\textwidth}
		\includegraphics[width=\textwidth]{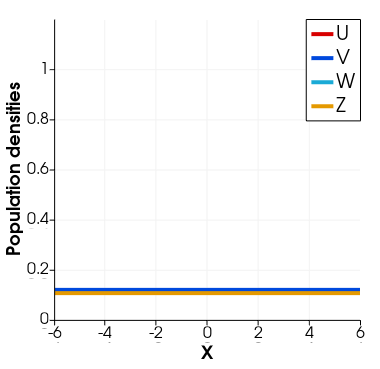}
		\caption{t = 600}
	\end{subfigure}
	\begin{subfigure}{0.155\textwidth}
		\includegraphics[width=\textwidth]{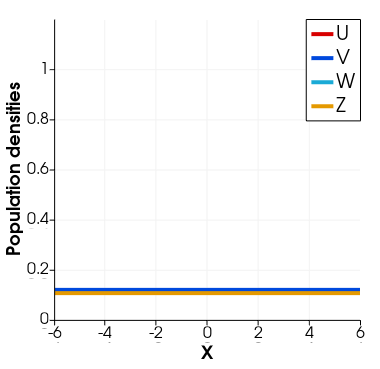}
		\caption{t = 700}
	\end{subfigure}
	\begin{subfigure}{0.155\textwidth}
		\includegraphics[width=\textwidth]{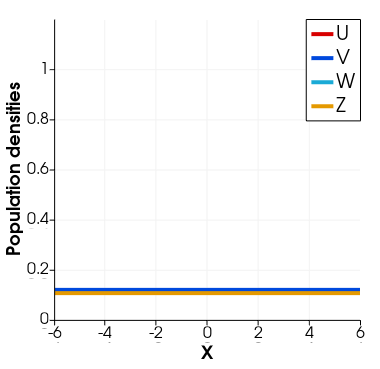}
		\caption{t = 1000}
	\end{subfigure}
	\caption{\small The size of gang populations $u$,$v$ and the amount of their corresponding graffiti $z$ and $w$ over time along the line $y=x$ at different time instant $t=0, 400, 500, 600, 700, 1000$ when $D_u=D_v=\chi_u=\chi_v =0.25.$}
	\label{fig1c}
\end{figure} 

\begin{figure}[H]
	\centering
	\begin{subfigure}{0.155\textwidth}
		\includegraphics[width=\textwidth]{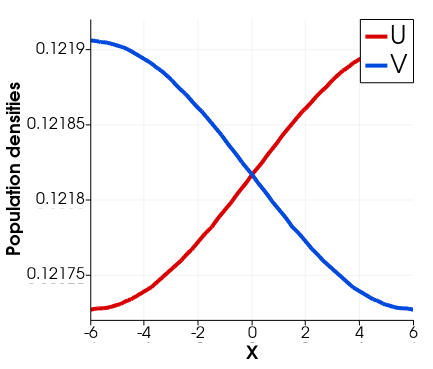}
		\caption{t = 400}
	\end{subfigure}
	\begin{subfigure}{0.155\textwidth}
		\includegraphics[width=\textwidth]{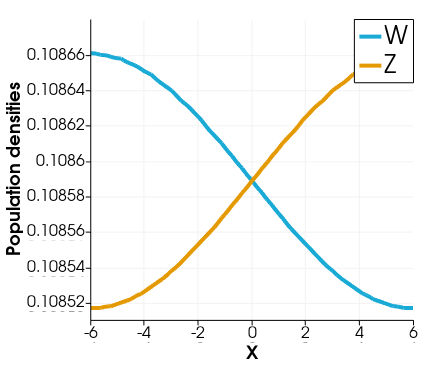}
		\caption{t = 400}
	\end{subfigure}
	\begin{subfigure}{0.155\textwidth}
		\includegraphics[width=\textwidth]{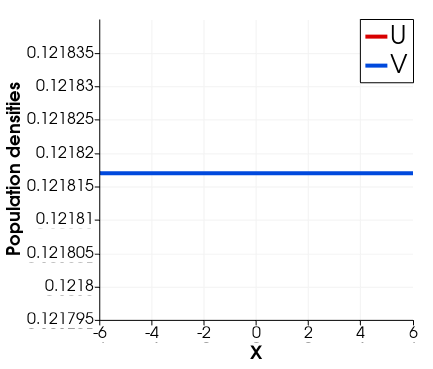}
		\caption{t = 700}
	\end{subfigure}
	\begin{subfigure}{0.155\textwidth}
		\includegraphics[width=\textwidth]{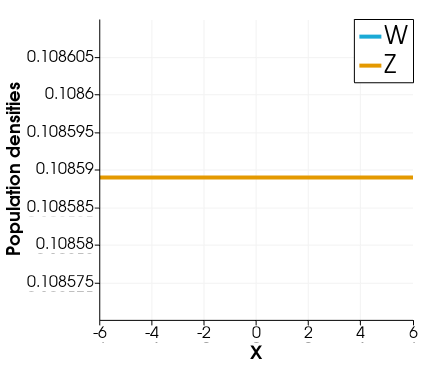}
		\caption{t = 700}
	\end{subfigure}
	\begin{subfigure}{0.155\textwidth}
		\includegraphics[width=\textwidth]{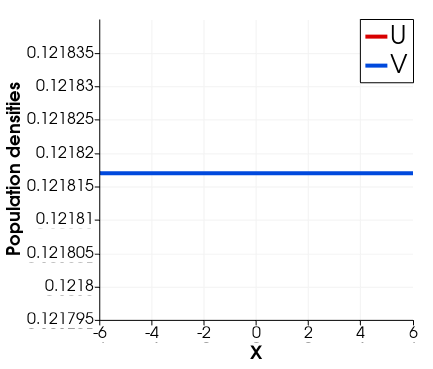}
		\caption{t = 1000}
	\end{subfigure}
	\begin{subfigure}{0.155\textwidth}
		\includegraphics[width=\textwidth]{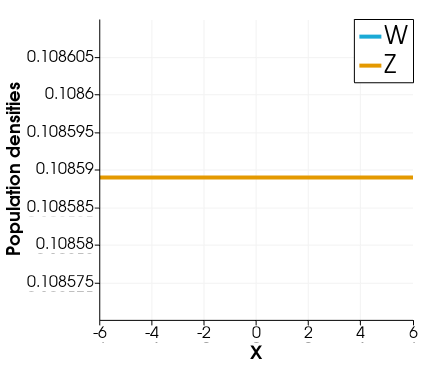}
		\caption{t = 1000}
	\end{subfigure}
	\caption{\small
		 Close up of populations densities along the line $y=x$ at different time instant $t= 400, 700, 1000$ when $D_u=D_v=\chi_u=\chi_v =0.25.$
	 }
	\label{fig1d}
\end{figure} 

Next, we set $D_u=D_v=3$ and $\chi_u=\chi_v=0.25$ and again use the standard Galerkin method \eqref{eq1u}--\eqref{eq1z}. Figure~\ref{fig2ab} shows that, when the time evolves gangs start to spread inside the domain and no accumulation seems to happen, the dark purple in the top row and light purple in the bottom row gradually cover the entire domain by the time $t=200$, i.e., the gangs are evenly distributed in each location and produce the same amount of graffiti. It is clear from Figure~\ref{fig2ab} that no separation happens when the problem is diffusion-dominated. The snapshots of the results are presented along the line $y=x$ in Figures~\ref{fig2c} and Figure~\ref{fig2d}, where we observe that the gang densities and their corresponding graffiti converge toward constant steady states and seem to stabilize around $0.121818$ and $0.10858$, respectively, by the time $t=200$.

A similar behavior has been observed in \citep{AlsenafiBarbaroConvectionDiffusionModel2018} for an agent-based model:
As long as a system parameter $\beta$ corresponding to $\chi_u, \chi_v$ in \eqref{prob} is sufficiently small, the gangs completely mix inside the domain and the expected densities converge toward constant steady states.
\begin{figure}[H]
	\centering
	\begin{subfigure}{0.155\textwidth}
		\includegraphics[width=\textwidth]{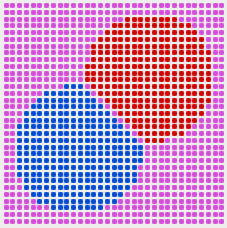}
		\caption{t = 0}
	\end{subfigure}
	\begin{subfigure}{0.155\textwidth}
		\includegraphics[width=\textwidth]{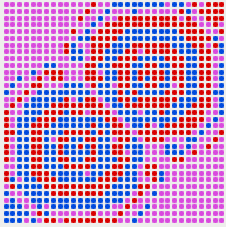}
		\caption{t = 50}
	\end{subfigure}
	\begin{subfigure}{0.155\textwidth}
		\includegraphics[width=\textwidth]{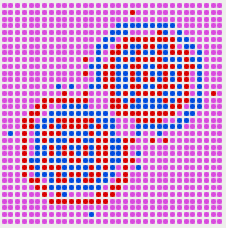}
		\caption{t = 75}
	\end{subfigure}
	\begin{subfigure}{0.155\textwidth}
		\includegraphics[width=\textwidth]{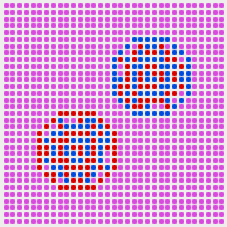}
		\caption{t = 100}
	\end{subfigure}
	\begin{subfigure}{0.155\textwidth}
		\includegraphics[width=\textwidth]{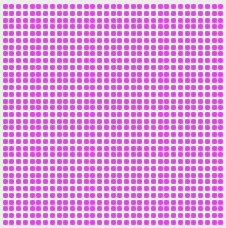}
		\caption{t = 200}
	\end{subfigure}
	\begin{subfigure}{0.155\textwidth}
		\includegraphics[width=\textwidth]{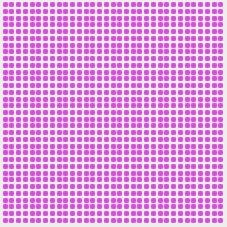}
		\caption{t = 1000}
	\end{subfigure}
	
		\begin{subfigure}{0.155\textwidth}
		\includegraphics[width=\textwidth]{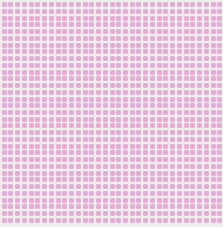}
		\caption{t = 0}
	\end{subfigure}
	\begin{subfigure}{0.155\textwidth}
		\includegraphics[width=\textwidth]{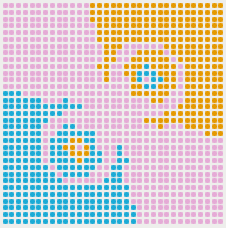}
		\caption{t = 50}
	\end{subfigure}
	\begin{subfigure}{0.155\textwidth}
		\includegraphics[width=\textwidth]{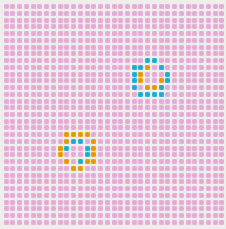}
		\caption{t = 75}
	\end{subfigure}
	\begin{subfigure}{0.155\textwidth}
		\includegraphics[width=\textwidth]{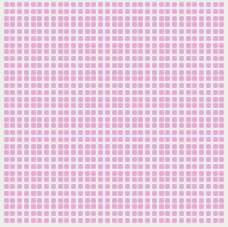}
		\caption{t = 100}
	\end{subfigure}
	\begin{subfigure}{0.155\textwidth}
		\includegraphics[width=\textwidth]{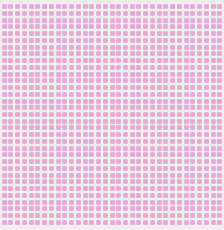}
		\caption{t = 200}
	\end{subfigure}
	\begin{subfigure}{0.155\textwidth}
		\includegraphics[width=\textwidth]{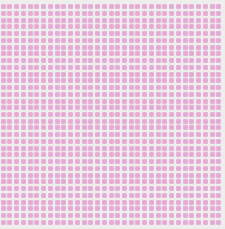}
		\caption{t = 1000}
	\end{subfigure}
	\caption{\small
	Numerical solutions for the model problem (\ref{prob}) obtained using the standard Galerkin method, at different time instant $t=0, 50, 75, 100, 200, 1000$ when $D_u=D_v=3.0$ and $\chi_u=\chi_v =0.25$. For the choice of colors, see the beginning of Subsection~\ref{sep-study}.
	 }
	\label{fig2ab}
\end{figure} 

\begin{figure}[H]
	\centering
		\begin{subfigure}{0.155\textwidth}
		\includegraphics[width=\textwidth]{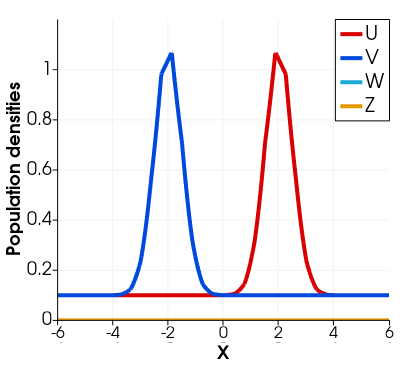}
		\caption{t = 0}
	\end{subfigure}
	\begin{subfigure}{0.155\textwidth}
		\includegraphics[width=\textwidth]{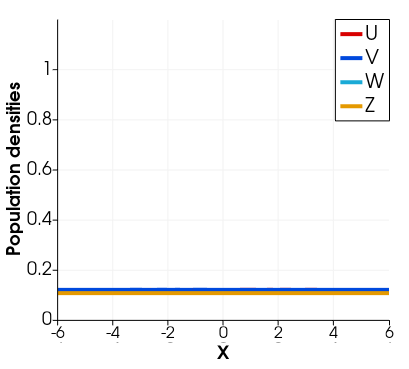}
		\caption{t = 50}
	\end{subfigure}
	\begin{subfigure}{0.155\textwidth}
		\includegraphics[width=\textwidth]{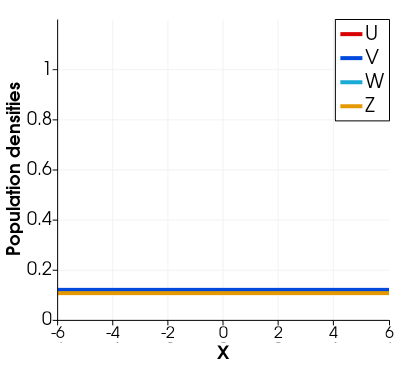}
		\caption{t = 75}
	\end{subfigure}
	\begin{subfigure}{0.155\textwidth}
		\includegraphics[width=\textwidth]{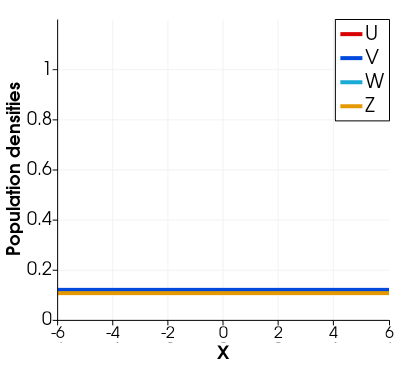}
		\caption{t = 100}
	\end{subfigure}
	\begin{subfigure}{0.155\textwidth}
		\includegraphics[width=\textwidth]{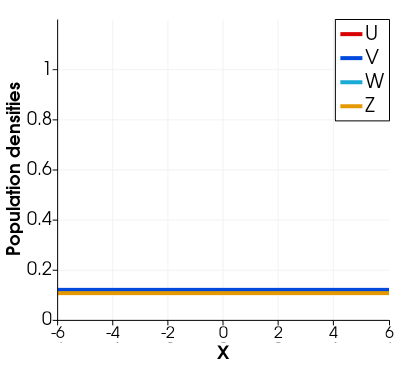}
		\caption{t = 200}
	\end{subfigure}
	\begin{subfigure}{0.155\textwidth}
		\includegraphics[width=\textwidth]{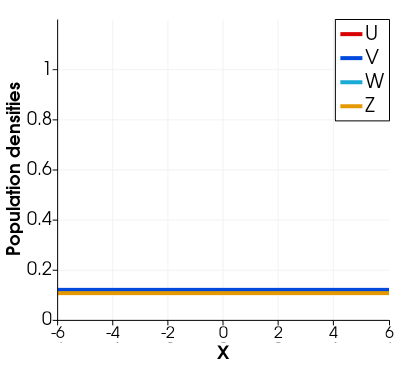}
		\caption{t = 1000}
	\end{subfigure}
	\caption{\small The amount of gang densities $u$,$v$ and their corresponding graffiti $z$ and $w$ over time along the line $y=x$ at different time instant $t=0, 50, 75, 100, 200, 1000$ when $D_u=D_v=3.0$ and $\chi_u=\chi_v =0.25.$}
	\label{fig2c}
\end{figure} 

\begin{figure}[H]
	\centering
	\begin{subfigure}{0.155\textwidth}
		\includegraphics[width=\textwidth]{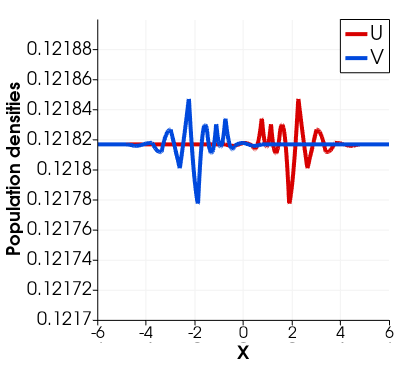}
		\caption{t = 100}
	\end{subfigure}
	\begin{subfigure}{0.155\textwidth}
		\includegraphics[width=\textwidth]{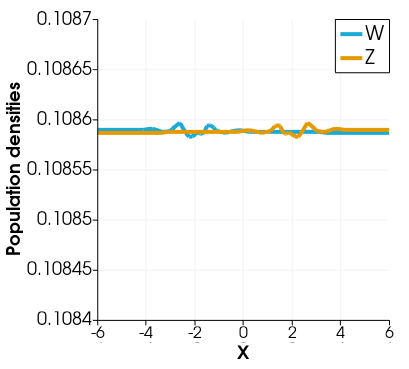}
		\caption{t = 100}
	\end{subfigure}
	\begin{subfigure}{0.155\textwidth}
		\includegraphics[width=\textwidth]{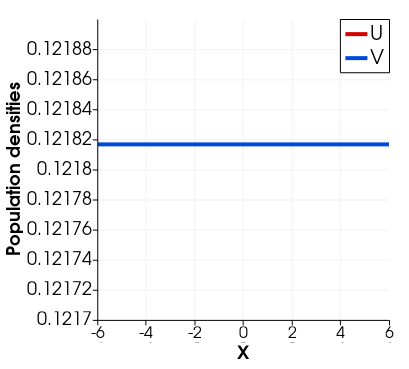}
		\caption{t = 200}
	\end{subfigure}
	\begin{subfigure}{0.155\textwidth}
		\includegraphics[width=\textwidth]{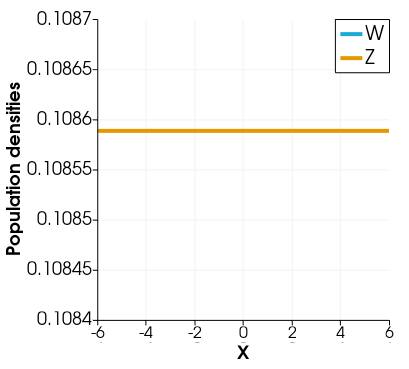}
		\caption{t = 200}
	\end{subfigure}
	\begin{subfigure}{0.155\textwidth}
		\includegraphics[width=\textwidth]{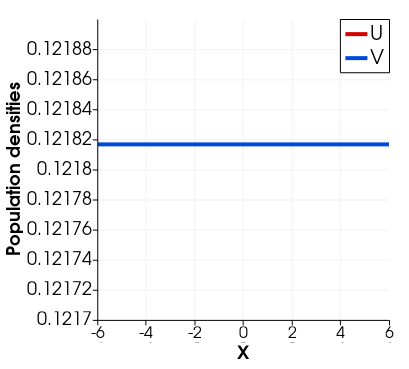}
		\caption{t = 1000}
	\end{subfigure}
	\begin{subfigure}{0.155\textwidth}
		\includegraphics[width=\textwidth]{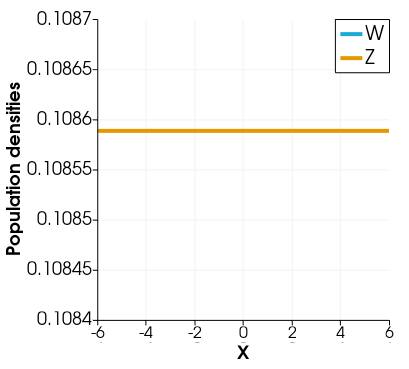}
		\caption{t = 1000}
	\end{subfigure}
	\caption{\small
		 Close up of population densities along the line $y=x$ at different time instant $t= 100, 200, 1000$ when $D_u=D_v=3.0$ and $\chi_u=\chi_v =0.25.$
	 }
	\label{fig2d}
\end{figure} 

\subsubsection{Nonhomogeneous limit functions}\label{sec:numeric_conv_nonsss}
For our next example, we consider $D_u=D_v=0.25$ and $\chi_u=\chi_v=3$. As shown in Figure~\ref{fig3}, the standard Galerkin method (\ref{eq1u})--(\ref{eq1z}) produces significant spurious oscillation in the entire domain which leads to negative nonphysical approximate solutions, and the simulation blows up at around $t=35$. As explained in the previous section, as a remedy we apply the FEM-FCT scheme to reduce the oscillations and preserve the positivity of the solutions at all time. The top row of Figure~\ref{fig4ab} shows the evolutionary movement of the gangs inside the domain. We observe that after a certain amount of time, red and blue clusters start to form, featuring some symmetries which track back to the symmetry of the initial data and the fact that the parameters are the same for both gangs. The same kind of pattern for graffiti densities emerge in the bottom row as each gang marks their own territory. The population densities seem to almost stabilize around $t=500$ and remain the same till the end of the time interval $T=1000$. Keeping the diffusion coefficient as before and changing the amount of convection coefficient to $\chi_u=\chi_v=10$, a different segregation pattern can be seen in Figure~\ref{fig5ab}. Therefore, we conclude that larger amount of convection in the model leads to directional movement of the gangs and creates partially separated regions. The snapshots of the results obtained with the FEM-FCT for both cases are displayed in Figures~\ref{fig4c} and \ref{fig5c}, respectively.

\begin{figure}[H]
	\centering
		\begin{subfigure}{0.12\textwidth}
		\includegraphics[width=\textwidth]{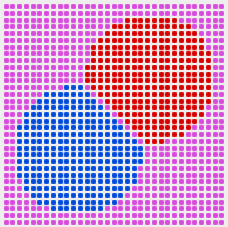}
		\caption{t = 0}
	\end{subfigure}
	\begin{subfigure}{0.12\textwidth}
		\includegraphics[width=\textwidth]{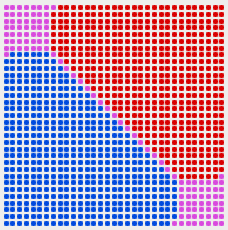}
		\caption{t = 5}
	\end{subfigure}
	\begin{subfigure}{0.12\textwidth}
		\includegraphics[width=\textwidth]{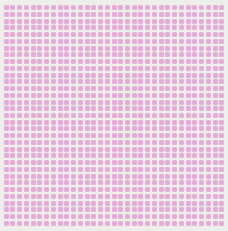}
		\caption{t = 0}
	\end{subfigure}
	\begin{subfigure}{0.12\textwidth}
		\includegraphics[width=\textwidth]{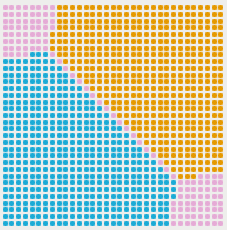}
		\caption{t = 5}
	\end{subfigure}
	\begin{subfigure}{0.155\textwidth}
		\includegraphics[width=\textwidth]{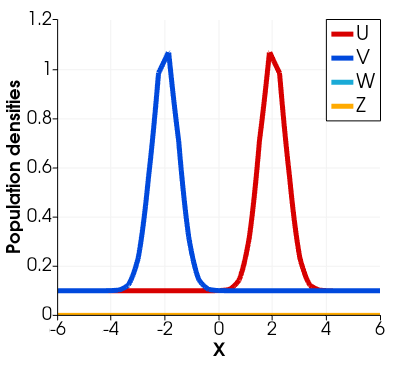}
		\caption{t = 0}
	\end{subfigure}
	\begin{subfigure}{0.155\textwidth}
		\includegraphics[width=\textwidth]{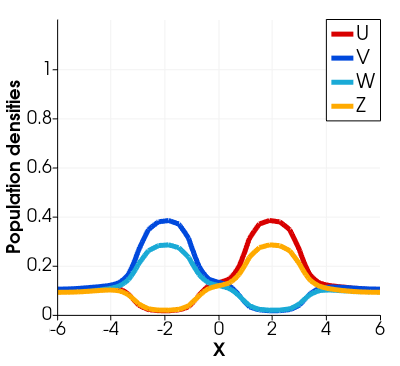}
		\caption{t = 5}
	\end{subfigure}
	\begin{subfigure}{0.155\textwidth}
		\includegraphics[width=\textwidth]{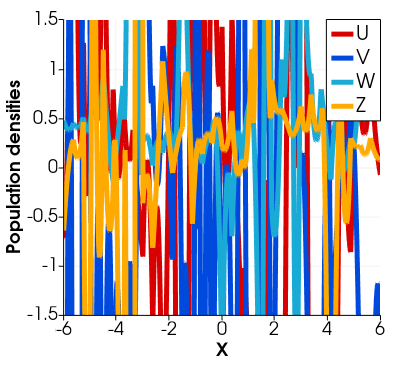}
		\caption{t = 35}
	\end{subfigure}
	\caption{\small Numerical solutions and their snapshots for the model problem (\ref{prob}) obtained using the standard Galerkin method, at different time instant $t=0, 5, 35$ when $D_u=D_v=0.25$ and $\chi_u=\chi_v =3.0$.}
	\label{fig3}
\end{figure} 

\begin{figure}[H]
	\centering
	\begin{subfigure}{0.16\textwidth}
		\includegraphics[width=\textwidth]{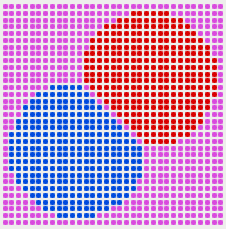}
		\caption{t = 0}
	\end{subfigure}
	\begin{subfigure}{0.155\textwidth}
		\includegraphics[width=\textwidth]{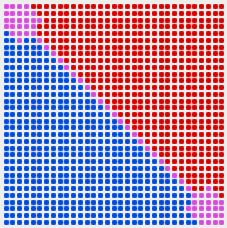}
		\caption{t = 5}
	\end{subfigure}
	\begin{subfigure}{0.155\textwidth}
		\includegraphics[width=\textwidth]{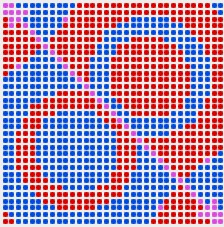}
		\caption{t = 50}
	\end{subfigure}
	\begin{subfigure}{0.155\textwidth}
		\includegraphics[width=\textwidth]{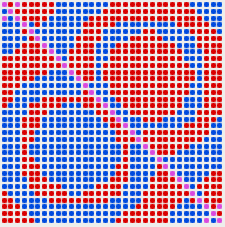}
		\caption{t = 100}
	\end{subfigure}
	\begin{subfigure}{0.155\textwidth}
		\includegraphics[width=\textwidth]{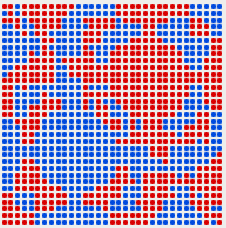}
		\caption{t = 500}
	\end{subfigure}
	\begin{subfigure}{0.155\textwidth}
		\includegraphics[width=\textwidth]{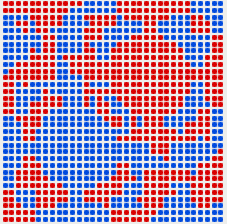}
		\caption{t = 1000}
	\end{subfigure}
	
		\begin{subfigure}{0.155\textwidth}
		\includegraphics[width=\textwidth]{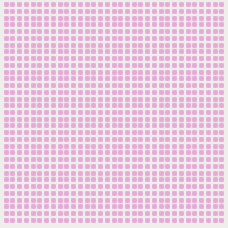}
		\caption{t = 0}
	\end{subfigure}
	\begin{subfigure}{0.155\textwidth}
		\includegraphics[width=\textwidth]{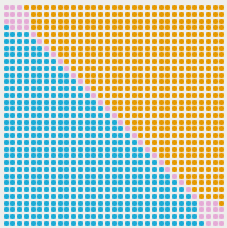}
		\caption{t = 5}
	\end{subfigure}
	\begin{subfigure}{0.155\textwidth}
		\includegraphics[width=\textwidth]{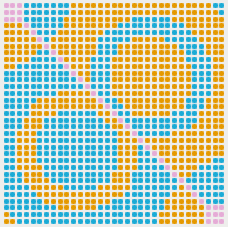}
		\caption{t = 50}
	\end{subfigure}
	\begin{subfigure}{0.155\textwidth}
		\includegraphics[width=\textwidth]{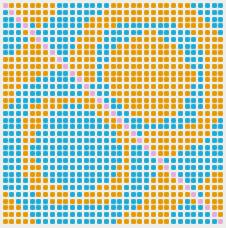}
		\caption{t = 100}
	\end{subfigure}
	\begin{subfigure}{0.155\textwidth}
		\includegraphics[width=\textwidth]{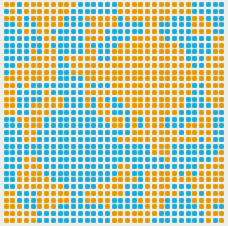}
		\caption{t = 500}
	\end{subfigure}
	\begin{subfigure}{0.155\textwidth}
		\includegraphics[width=\textwidth]{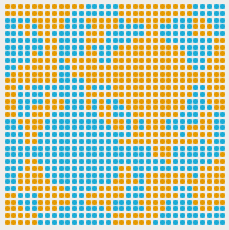}
		\caption{t = 1000}
	\end{subfigure}
	\caption{\small
	Numerical solutions for the model problem (\ref{prob}) obtained using the FEM-FCT method, at different time instant $t=0, 5, 50, 100, 500, 1000$ when $D_u=D_v= 0.25$ and $\chi_u=\chi_v =3.0$. For the choice of colors, see the beginning of Subsection~\ref{sep-study}.
	 }
	\label{fig4ab}
\end{figure} 

\begin{figure}[H]
	\centering
		\begin{subfigure}{0.155\textwidth}
		\includegraphics[width=\textwidth]{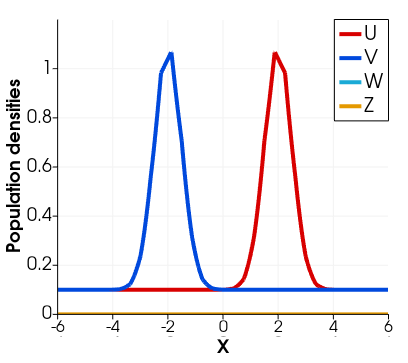}
		\caption{t = 0}
	\end{subfigure}
	\begin{subfigure}{0.155\textwidth}
		\includegraphics[width=\textwidth]{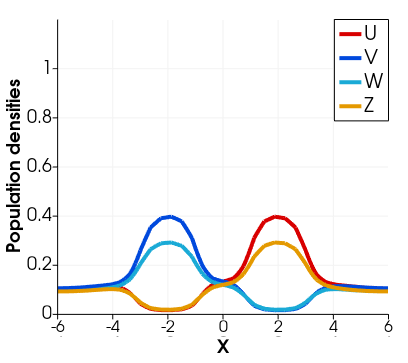}
		\caption{t = 5}
	\end{subfigure}
	\begin{subfigure}{0.155\textwidth}
		\includegraphics[width=\textwidth]{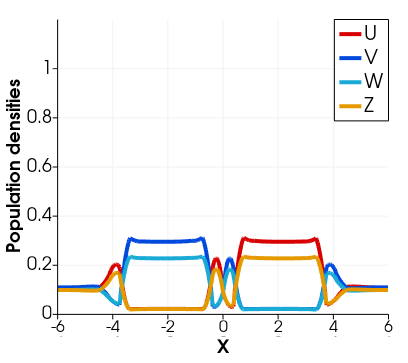}
		\caption{t = 50}
	\end{subfigure}
	\begin{subfigure}{0.155\textwidth}
		\includegraphics[width=\textwidth]{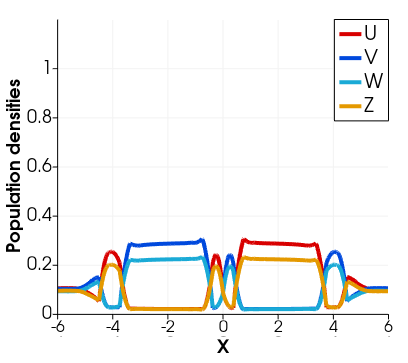}
		\caption{t = 100}
	\end{subfigure}
	\begin{subfigure}{0.155\textwidth}
		\includegraphics[width=\textwidth]{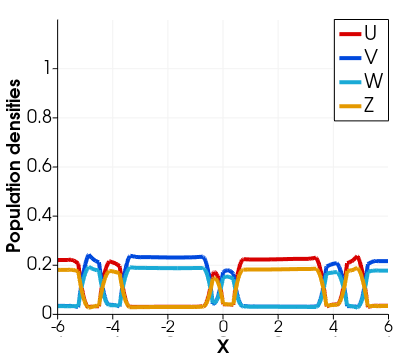}
		\caption{t = 500}
	\end{subfigure}
	\begin{subfigure}{0.155\textwidth}
		\includegraphics[width=\textwidth]{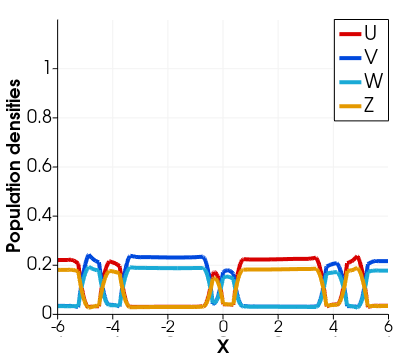}
		\caption{t = 1000}
	\end{subfigure}
	\caption{\small The size of gang populations $u$,$v$ and the amount of their corresponding graffiti $z$ and $w$ over time along the line $y=x$ at different time instant $t=0, 5, 50, 100, 500, 1000$ when $D_u=D_v=0.25$ $\chi_u=\chi_v =3.0.$}
	\label{fig4c}
\end{figure} 

\begin{figure}[H]
	\centering
	\begin{subfigure}{0.155\textwidth}
		\includegraphics[width=\textwidth]{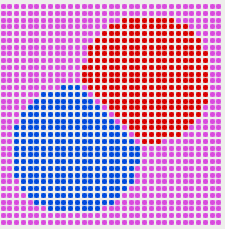}
		\caption{t = 0}
	\end{subfigure}
	\begin{subfigure}{0.155\textwidth}
		\includegraphics[width=\textwidth]{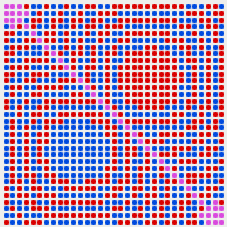}
		\caption{t = 5}
	\end{subfigure}
	\begin{subfigure}{0.155\textwidth}
		\includegraphics[width=\textwidth]{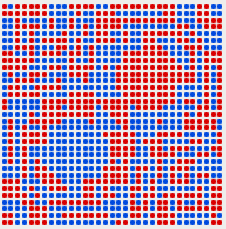}
		\caption{t = 50}
	\end{subfigure}
	\begin{subfigure}{0.155\textwidth}
		\includegraphics[width=\textwidth]{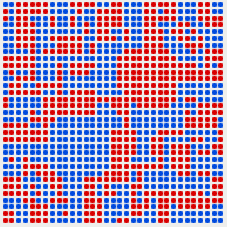}
		\caption{t = 100}
	\end{subfigure}
	\begin{subfigure}{0.155\textwidth}
		\includegraphics[width=\textwidth]{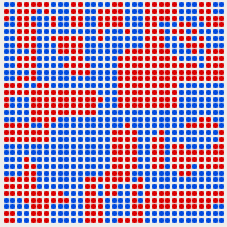}
		\caption{t = 500}
	\end{subfigure}
	\begin{subfigure}{0.155\textwidth}
		\includegraphics[width=\textwidth]{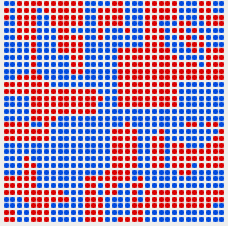}
		\caption{t = 1000}
	\end{subfigure}
	
		\begin{subfigure}{0.155\textwidth}
		\includegraphics[width=\textwidth]{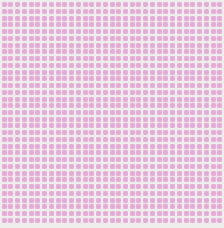}
		\caption{t = 0}
	\end{subfigure}
	\begin{subfigure}{0.155\textwidth}
		\includegraphics[width=\textwidth]{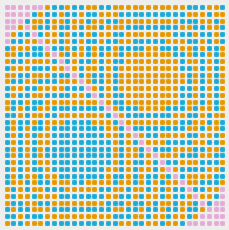}
		\caption{t = 5}
	\end{subfigure}
	\begin{subfigure}{0.155\textwidth}
		\includegraphics[width=\textwidth]{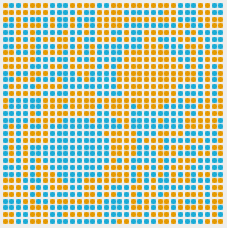}
		\caption{t = 50}
	\end{subfigure}
	\begin{subfigure}{0.155\textwidth}
		\includegraphics[width=\textwidth]{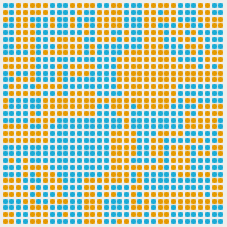}
		\caption{t = 100}
	\end{subfigure}
	\begin{subfigure}{0.155\textwidth}
		\includegraphics[width=\textwidth]{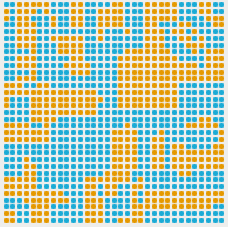}
		\caption{t = 500}
	\end{subfigure}
	\begin{subfigure}{0.155\textwidth}
		\includegraphics[width=\textwidth]{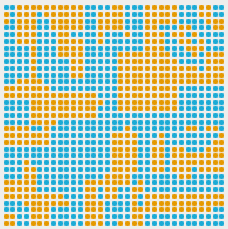}
		\caption{t = 1000}
	\end{subfigure}
	\caption{\small
	Numerical solutions for the model problem (\ref{prob}) obtained using the FEM-FCT method, at different time instant $t=0, 5, 50, 100, 500, 1000$ when $D_u=D_v= 0.25$ and $\chi_u=\chi_v =10$. For the choice of colors, see the beginning of Subsection~\ref{sep-study}.
	 }
	\label{fig5ab}
\end{figure}

\begin{figure}[H]
	\centering
		\begin{subfigure}{0.155\textwidth}
		\includegraphics[width=\textwidth]{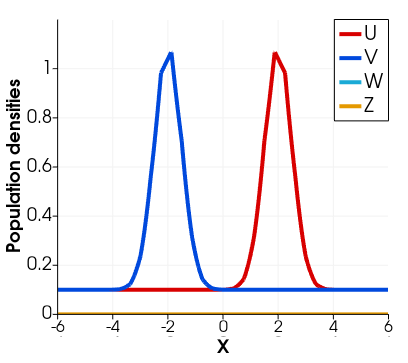}
		\caption{t = 0}
	\end{subfigure}
	\begin{subfigure}{0.155\textwidth}
		\includegraphics[width=\textwidth]{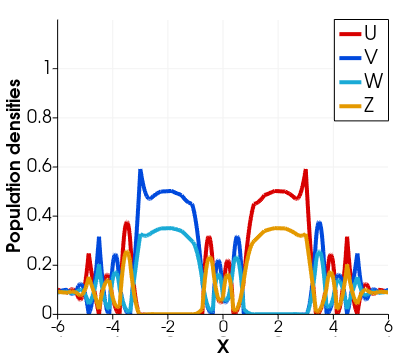}
		\caption{t = 5}
	\end{subfigure}
	\begin{subfigure}{0.155\textwidth}
		\includegraphics[width=\textwidth]{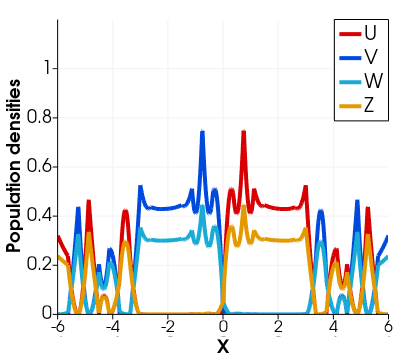}
		\caption{t = 50}
	\end{subfigure}
	\begin{subfigure}{0.155\textwidth}
		\includegraphics[width=\textwidth]{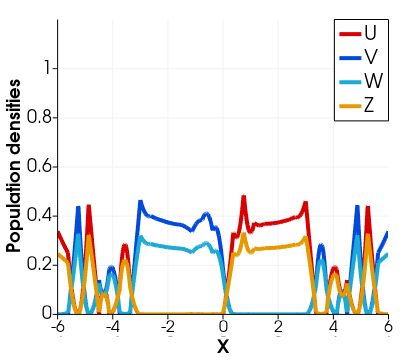}
		\caption{t = 100}
	\end{subfigure}
	\begin{subfigure}{0.155\textwidth}
		\includegraphics[width=\textwidth]{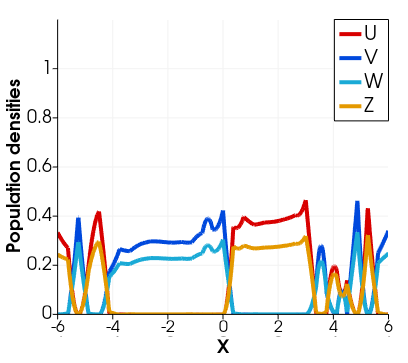}
		\caption{t = 500}
	\end{subfigure}
	\begin{subfigure}{0.155\textwidth}
		\includegraphics[width=\textwidth]{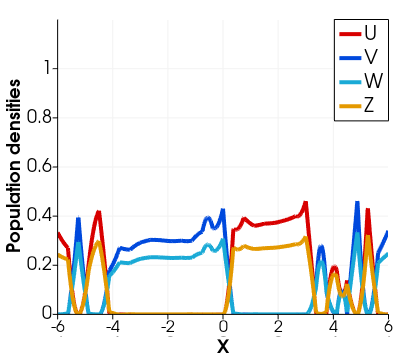}
		\caption{t = 1000}
	\end{subfigure}
	\caption{\small The size of gang populations $u$,$v$ and the amount of their corresponding graffiti $z$ and $w$ over time along the line $y=x$ at different time instant $t=0, 5, 50, 100, 500, 1000$ when $D_u=D_v=0.25$ and $\chi_u=\chi_v =10.$}
	\label{fig5c}
\end{figure}

Comparing our results with the stochastic simulation in the agent-based model in \citep{AlsenafiBarbaroConvectionDiffusionModel2018}, we notice some similarities between our experiments' outcome and their reported results, namely that large values of $\beta$ (the convection coefficient in their model) lead to a well-segregated phase. However, in each patch dominated by red or blue in Figures~\ref{fig4ab} and \ref{fig5ab}, there is still some amount of the opposite group with much smaller -- but still positive -- density present (cf.\ Figures~\ref{fig4c} and \ref{fig5c}), illustrating partial (as opposed to complete) segregation. Moreover, unlike as \citep{AlsenafiBarbaroConvectionDiffusionModel2018}, we do not observe any coarsening of the regions: Once territories are established, they are rather stable.

Next, we consider fixed diffusion rates as in the previous examples, i.e., $ D_u=D_v=0.25$, but different convection coefficients $\chi_u=2$ and $\chi_v=4$, meaning that the second gang's movement is more affected by the first gang's graffiti than vice-versa. As shown in Figure~\ref{fig6ab}, the gang with density $v$ starts to cluster together very tightly in small patches and partially separates from the opposite group at around $t=200$ while the other gangs dominates on larger regions. That is, we observe that the segregation pattern for both groups is quite different here, in contrast to the previous examples. The snapshots of concentration densities over time along the line $y=x$ are displayed in Figure~\ref{fig6c}. As it can be seen, the concentration density $v$ and its corresponding graffiti $w$ occupy less parts of the domain but with higher amount of densities (almost around 0.4 and 0.3 in their dominated regions, respectively), however, the gang population $u$ spreads more freely in the wider range but with less density (nearly under 0.2). Of course, both gangs have the same total amount of mass throughout evolution since there are no source or sink terms in the model.

In \citep[Subsection~5.1]{Alsenafi2021}, a similar behavior, namely that the gang with the largest graffiti avoidance rate concentrates on smaller regions,
has been observed for an agent-based model with three rivaling gangs.

\begin{figure}[H]
	\centering
	\begin{subfigure}{0.155\textwidth}
		\includegraphics[width=\textwidth]{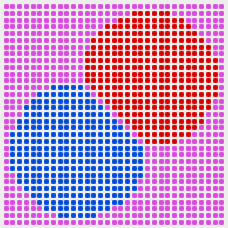}
		\caption{t = 0}
	\end{subfigure}
	\begin{subfigure}{0.155\textwidth}
		\includegraphics[width=\textwidth]{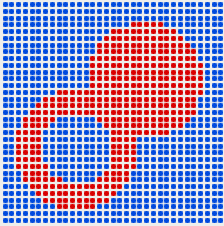}
		\caption{t = 50}
	\end{subfigure}
	\begin{subfigure}{0.155\textwidth}
		\includegraphics[width=\textwidth]{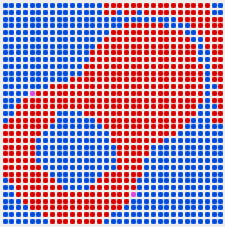}
		\caption{t = 100}
	\end{subfigure}
	\begin{subfigure}{0.155\textwidth}
		\includegraphics[width=\textwidth]{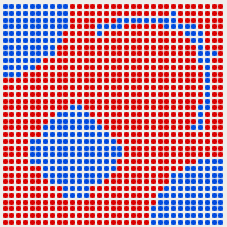}
		\caption{t = 200}
	\end{subfigure}
	\begin{subfigure}{0.155\textwidth}
		\includegraphics[width=\textwidth]{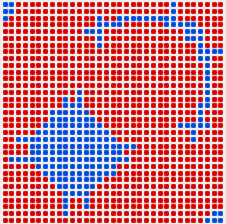}
		\caption{t = 400}
	\end{subfigure}
	\begin{subfigure}{0.155\textwidth}
		\includegraphics[width=\textwidth]{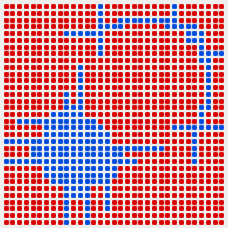}
		\caption{t = 1000}
	\end{subfigure}
	
		\begin{subfigure}{0.155\textwidth}
		\includegraphics[width=\textwidth]{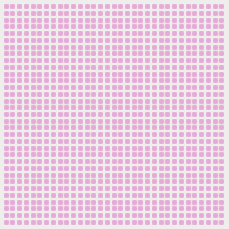}
		\caption{t = 0}
	\end{subfigure}
	\begin{subfigure}{0.155\textwidth}
		\includegraphics[width=\textwidth]{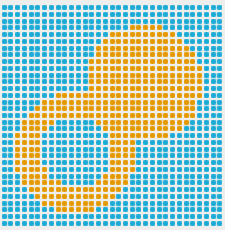}
		\caption{t = 50}
	\end{subfigure}
	\begin{subfigure}{0.155\textwidth}
		\includegraphics[width=\textwidth]{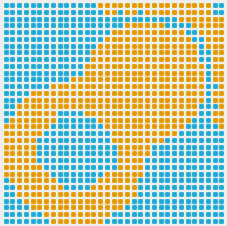}
		\caption{t = 100}
	\end{subfigure}
	\begin{subfigure}{0.155\textwidth}
		\includegraphics[width=\textwidth]{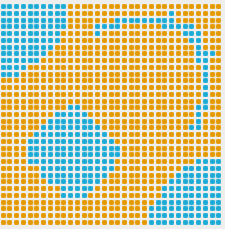}
		\caption{t = 200}
	\end{subfigure}
	\begin{subfigure}{0.155\textwidth}
		\includegraphics[width=\textwidth]{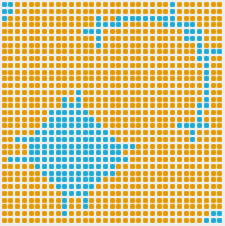}
		\caption{t = 400}
	\end{subfigure}
	\begin{subfigure}{0.155\textwidth}
		\includegraphics[width=\textwidth]{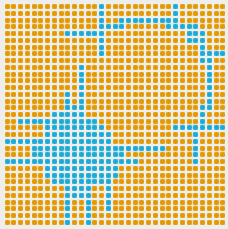}
		\caption{t = 1000}
	\end{subfigure}
	\caption{\small
	Numerical solutions for the model problem (\ref{prob}) obtained using the FEM-FCT method, at different time instant $t=0, 50, 100, 200, 400, 1000$ when $D_u=D_v= 0.25$, $\chi_u=2.0 $ and $\chi_v =4.0$. For the choice of colors, see the beginning of Subsection~\ref{sep-study}.
	 }
	\label{fig6ab}
\end{figure} 

\begin{figure}[H]
	\centering
		\begin{subfigure}{0.155\textwidth}
		\includegraphics[width=\textwidth]{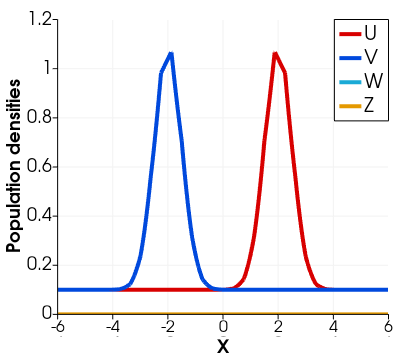}
		\caption{t = 0}
	\end{subfigure}
	\begin{subfigure}{0.155\textwidth}
		\includegraphics[width=\textwidth]{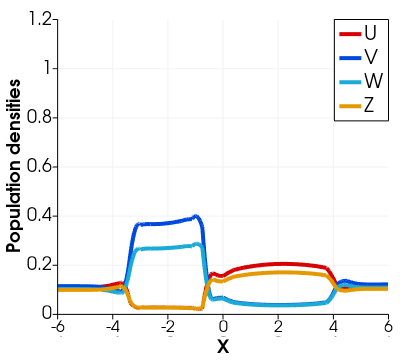}
		\caption{t = 50}
	\end{subfigure}
	\begin{subfigure}{0.155\textwidth}
		\includegraphics[width=\textwidth]{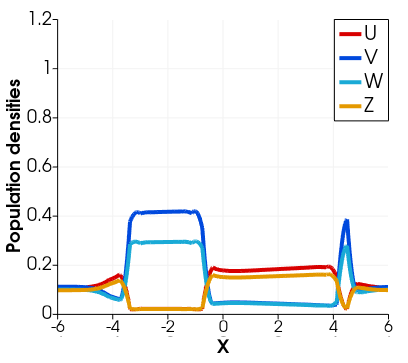}
		\caption{t = 100}
	\end{subfigure}
	\begin{subfigure}{0.155\textwidth}
		\includegraphics[width=\textwidth]{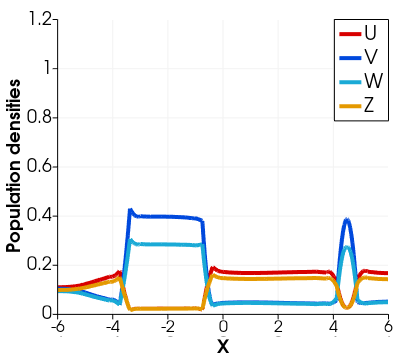}
		\caption{t = 200}
	\end{subfigure}
	\begin{subfigure}{0.155\textwidth}
		\includegraphics[width=\textwidth]{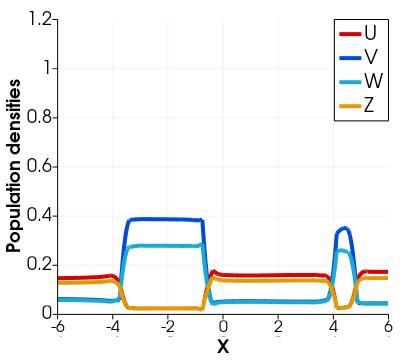}
		\caption{t = 400}
	\end{subfigure}
	\begin{subfigure}{0.155\textwidth}
		\includegraphics[width=\textwidth]{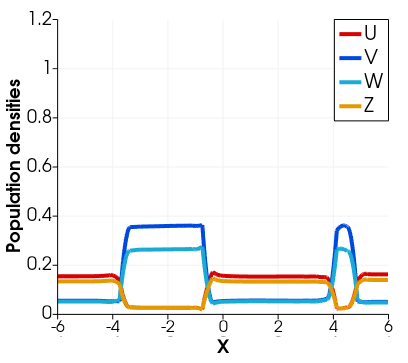}
		\caption{t = 1000}
	\end{subfigure}
	\caption{\small The size of gang populations $u$,$v$ and the amount of their corresponding graffiti $z$ and $w$ over time along the line $y=x$ at different time instant $t=0, 50, 100, 200, 400, 1000$ when $D_u=D_v=0.25$, $\chi_u= 2.0$ and $\chi_v =4.0.$}
	\label{fig6c}
\end{figure}

\subsubsection{An example of complete segregation}\label{sec:complete_seg}
Finally, we turn to answer the question whether complete separation occurs at all or not; that is, whether the supports of the gang densities can become disjoint in the large-time limit. We show that this is at least possible in settings where  the domain is divided into two parts in each of which one gang density is extremely small already at the initial time, and when additionally the diffusion rates are very low, the former prevents mixed amount of population and their interaction in the same regions at the beginning and the later controls their spread inside the domain over time. For this reason, we consider the case where $D_u=D_v=0.01$ and $\chi_u=\chi_v=3$ and the initial conditions are given by
\begin{equation}\label{eq:init_data_num2}
u^0(x,y) = \ure^{-(x-3)^2 -(y-3)^2}\,, \quad v^0(x,y) = \ure^{-(x+3)^2 - (y+3)^2}\,, \quad w^0(x,y) = 0\,, \quad z^0(x,y) = 0\,,
\end{equation}
which are displayed in Figure~\ref{fig7ab}(a). Compared to the initial data in \eqref{eq:init_data_num1}, the functions in \eqref{eq:init_data_num2} miss the additive constant $0.1$ and their maxima are further apart from each other. Therefore, the dark purple at the initial stage in Figure~\ref{fig7ab}(a) represents not only that both concentrations are (roughly) equal but additionally that they are both roughly zero.

From Figure~\ref{fig7ab}, we observe that the gangs stay totally separated over the time and there is no interaction between the two groups from beginning to the end. The snapshot of the result along the line $y=x$ are displayed in the Figure~\ref{fig7c}. We also show the close up of population densities at different time instants along the line $y=x$ in Figure~\ref{fig7d}, as can be seen no size of gang densities $u$ appears in the regions occupied by the gang density $v$ and vice-versa.

\begin{figure}[H]
	\centering
	\begin{subfigure}{0.155\textwidth}
		\includegraphics[width=\textwidth]{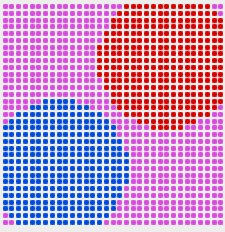}
		\caption{t = 0}
	\end{subfigure}
	\begin{subfigure}{0.155\textwidth}
		\includegraphics[width=\textwidth]{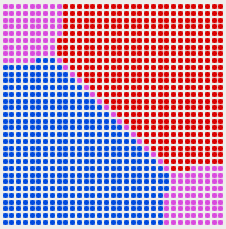}
		\caption{t = 50}
	\end{subfigure}
	\begin{subfigure}{0.155\textwidth}
		\includegraphics[width=\textwidth]{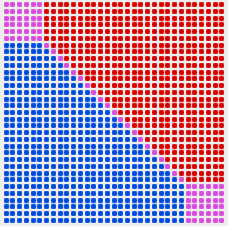}
		\caption{t = 75}
	\end{subfigure}
	\begin{subfigure}{0.155\textwidth}
		\includegraphics[width=\textwidth]{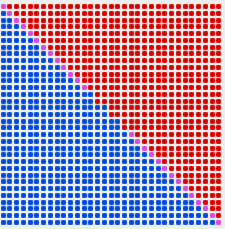}
		\caption{t = 200}
	\end{subfigure}
	\begin{subfigure}{0.155\textwidth}
		\includegraphics[width=\textwidth]{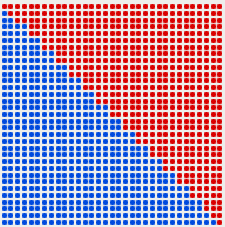}
		\caption{t = 500}
	\end{subfigure}
	\begin{subfigure}{0.155\textwidth}
		\includegraphics[width=\textwidth]{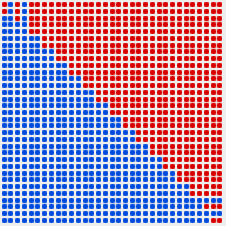}
		\caption{t = 1000}
	\end{subfigure}
	
		\begin{subfigure}{0.155\textwidth}
		\includegraphics[width=\textwidth]{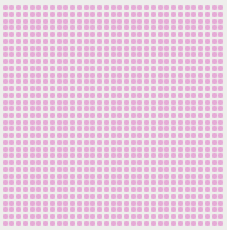}
		\caption{t = 0}
	\end{subfigure}
	\begin{subfigure}{0.155\textwidth}
		\includegraphics[width=\textwidth]{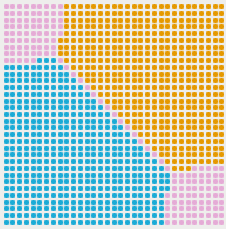}
		\caption{t = 50}
	\end{subfigure}
	\begin{subfigure}{0.155\textwidth}
		\includegraphics[width=\textwidth]{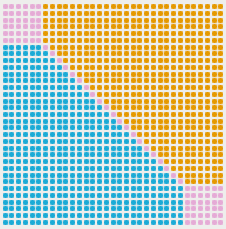}
		\caption{t = 75}
	\end{subfigure}
	\begin{subfigure}{0.155\textwidth}
		\includegraphics[width=\textwidth]{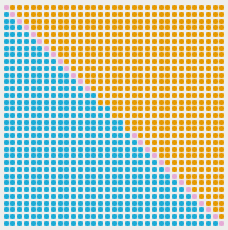}
		\caption{t = 200}
	\end{subfigure}
	\begin{subfigure}{0.155\textwidth}
		\includegraphics[width=\textwidth]{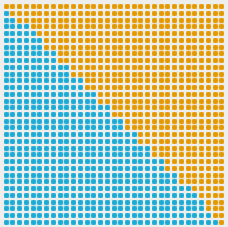}
		\caption{t = 500}
	\end{subfigure}
	\begin{subfigure}{0.155\textwidth}
		\includegraphics[width=\textwidth]{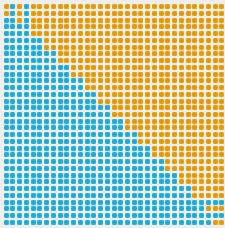}
		\caption{t = 1000}
	\end{subfigure}
	\caption{\small
	Numerical solutions for the model problem (\ref{prob}) obtained using the FEM-FCT method, at different time instant $t=0, 50, 75, 200, 500, 1000$ when $D_u=D_v=0.01 $ and $\chi_u=\chi_v =3.0$. For the choice of colors, see the beginning of Subsection~\ref{sep-study}.
	 }
	\label{fig7ab}
\end{figure} 

\begin{figure}[H]
	\centering
		\begin{subfigure}{0.155\textwidth}
		\includegraphics[width=\textwidth]{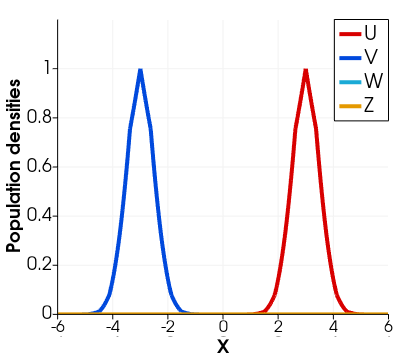}
		\caption{t = 0}
	\end{subfigure}
	\begin{subfigure}{0.155\textwidth}
		\includegraphics[width=\textwidth]{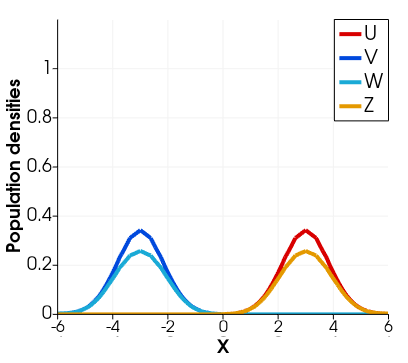}
		\caption{t = 50}
	\end{subfigure}
	\begin{subfigure}{0.155\textwidth}
		\includegraphics[width=\textwidth]{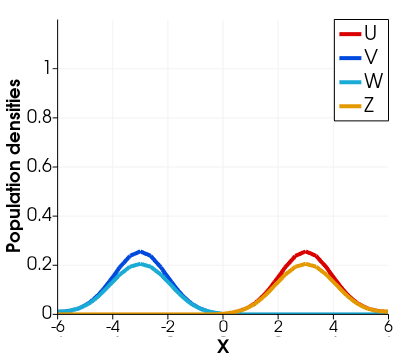}
		\caption{t = 75}
	\end{subfigure}
	\begin{subfigure}{0.155\textwidth}
		\includegraphics[width=\textwidth]{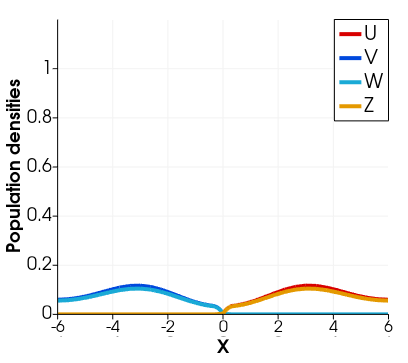}
		\caption{t = 200}
	\end{subfigure}
	\begin{subfigure}{0.155\textwidth}
		\includegraphics[width=\textwidth]{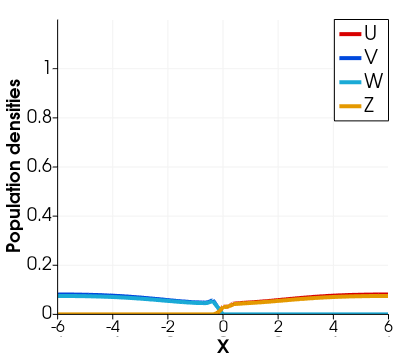}
		\caption{t = 500}
	\end{subfigure}
	\begin{subfigure}{0.155\textwidth}
		\includegraphics[width=\textwidth]{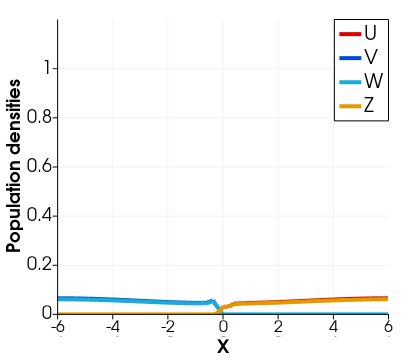}
		\caption{t = 1000}
	\end{subfigure}
	\caption{\small The size of gang populations $u$,$v$ and the amount of their corresponding graffiti $z$ and $w$ over time along the line $y=x$ at different time instant $t=0, 50, 75, 200, 500, 1000$ when $D_u=D_v=0.01$ and $\chi_u=\chi_v =3.0.$}
	\label{fig7c}
\end{figure} 

\begin{figure}[H]
	\centering
	\begin{subfigure}{0.155\textwidth}
		\includegraphics[width=\textwidth]{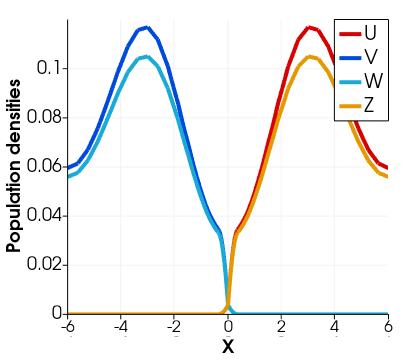}
		\caption{t = 200}
	\end{subfigure}
	\begin{subfigure}{0.155\textwidth}
		\includegraphics[width=\textwidth]{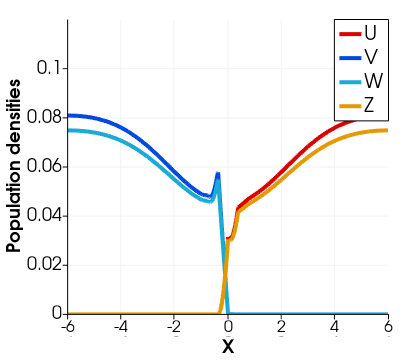}
		\caption{t = 500}
	\end{subfigure}
	\begin{subfigure}{0.155\textwidth}
		\includegraphics[width=\textwidth]{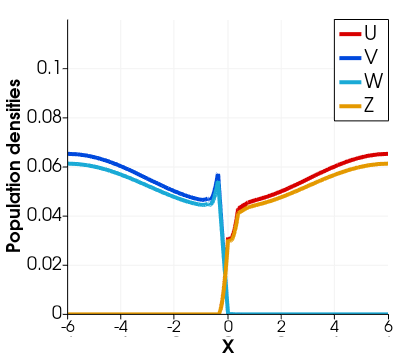}
		\caption{t = 1000}
	\end{subfigure}
	\caption{\small
		 Close up of populations densities along the line $y=x$ at different time instant $t= 200, 500, 1000$ when $D_u=D_v=0.01 $ and $\chi_u=\chi_v =3.0.$
	 }
	\label{fig7d}
\end{figure} 

\subsection{Time step and mesh convergence study}
Now, we examine the effect of time and mesh refinement, which is one of the standard numerical experiments used to validate approximate solutions. For this reason, as in the first example in Subsection~\ref{sec:numeric_conv_nonsss} (see Figures~\ref{fig4ab}, \ref{fig4c}) we consider $D_u=D_v=0.25$, $\chi_u=\chi_v=3$ and,  the initial conditions (\ref{eq:init_data_num1})
over the time interval $[0, 500]$. To begin with, we keep all of the setting fixed as before and only coarsen and refine the mesh. Comparing the plots depicted in Figure~\ref{fig8a} for different levels of refinement at final time $T=500$, we observe that the segregation patterns are almost the same with slight smoothness in the segregated patches boundaries when using the finer mesh. Therefore, it is clear that the formation of blue and red cluster inside the domain does not change significantly by changing the mesh size. The snapshots of the mesh convergence results along the line $y=x$ are also displayed in Figure~\ref{fig8b}, which shows the approximate solutions obtained at different refinement levels almost converge toward the same segregation states.

\begin{figure}[H]
	\centering
	\begin{subfigure}{0.19\textwidth}
		\includegraphics[width=\textwidth]{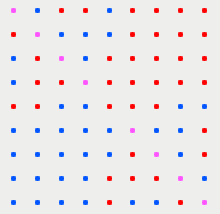}
		\caption{3 refinements}
	\end{subfigure}
	\begin{subfigure}{0.19\textwidth}
		\includegraphics[width=\textwidth]{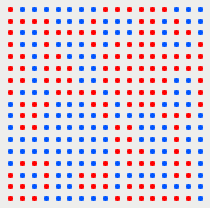}
		\caption{4 refinements}
	\end{subfigure}
	\begin{subfigure}{0.19\textwidth}
		\includegraphics[width=\textwidth]{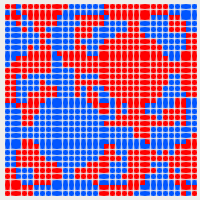}
		\caption{5 refinements}
	\end{subfigure}
	\begin{subfigure}{0.19\textwidth}
		\includegraphics[width=\textwidth]{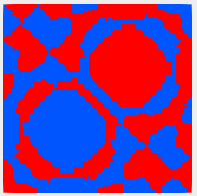}
		\caption{6 refinements}
	\end{subfigure}
	\begin{subfigure}{0.19\textwidth}
		\includegraphics[width=\textwidth]{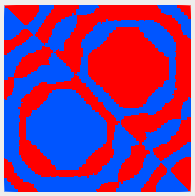}
		\caption{7 refinements}
	\end{subfigure}
	\caption{\small
	Numerical solutions $u$ and $v$ obtained using FEM-FCT scheme in different refinement levels at final time $T= 500$ where $D_u=D_v=0.25$ and $\chi_u=\chi_v=3$. For the choice of colors, see the beginning of Subsection~\ref{sep-study}.
	 }
	\label{fig8a}
\end{figure} 

Next, preserving the 5 level of refinements as in all the previous examples and keep the other parameters fixed as before, we choose different time step sizes in our computation. Figure~\ref{fig9a} shows that the segregation patterns are similar and the differences between the approximate solutions obtained using FEM-FCT with different time steps $\Delta t=1.0, 0.5, 0.25, 0.125,$ and $0.0625$ are very small. The snapshots of the results along the line $y=x$ are depicted in Figure~\ref{fig9b}, which shows the time convergence behavior.

\begin{figure}[H]
	\centering
	\begin{subfigure}{0.275\textwidth}
		\includegraphics[width=\textwidth]{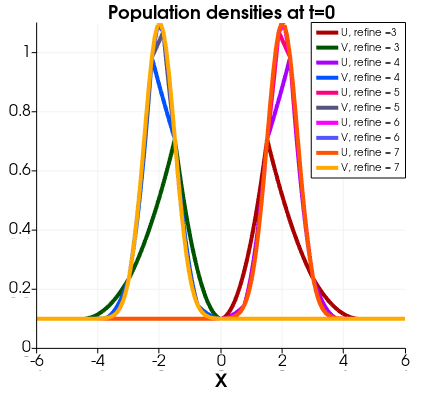}
		\caption{$t$ = 0.0}
	\end{subfigure}
	\begin{subfigure}{0.275\textwidth}
		\includegraphics[width=\textwidth]{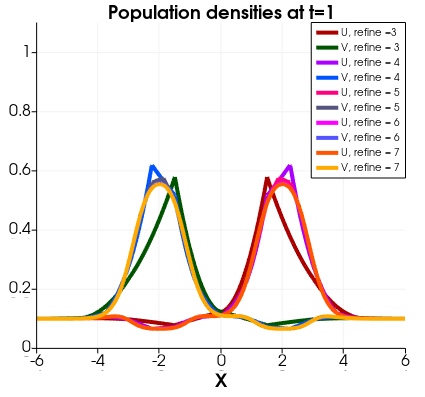}
		\caption{$t$ = 1}
	\end{subfigure}
	\begin{subfigure}{0.275\textwidth}
		\includegraphics[width=\textwidth]{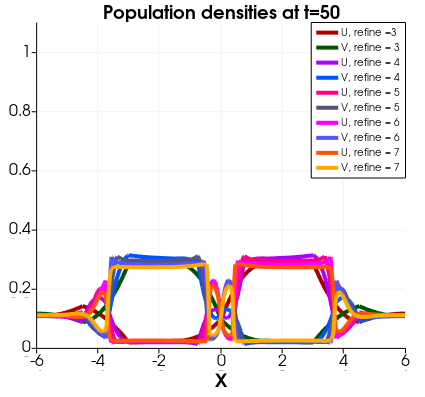}
		\caption{$t$ = 50}
	\end{subfigure}
	\begin{subfigure}{0.275\textwidth}
		\includegraphics[width=\textwidth]{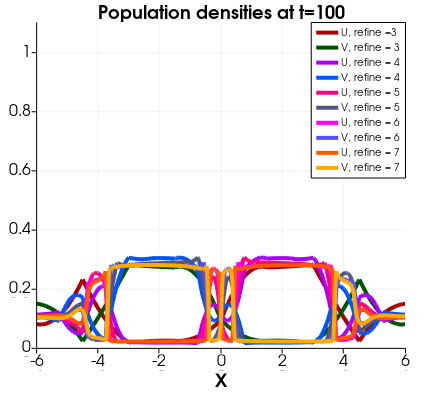}
		\caption{$t$ = 100}
	\end{subfigure}
	\begin{subfigure}{0.275\textwidth}
		\includegraphics[width=\textwidth]{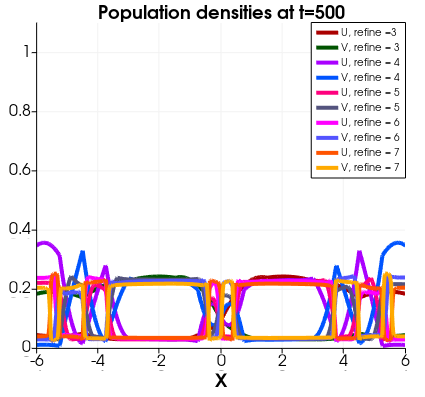}
		\caption{$t$ = 500}
	\end{subfigure}
	\caption{\small
	The amount of gang concentration $u$, $v$ and graffiti densities $w$ and $z$ along the line $y=x$ in different level of refinement at different time instants $ t= 0, 1, 50, 100, 500$ where $D_u=D_v=0.25$ and $\chi_u=\chi_v=3$.
	 }
	\label{fig8b}
\end{figure} 

\begin{figure}[H]
	\centering
	\begin{subfigure}{0.19\textwidth}
		\includegraphics[width=\textwidth]{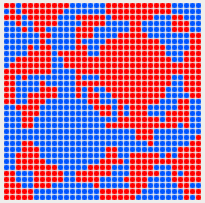}
		\caption{$\Delta t = 1.0$}
	\end{subfigure}
	\begin{subfigure}{0.19\textwidth}
		\includegraphics[width=\textwidth]{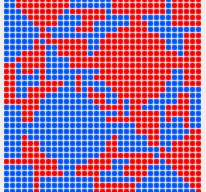}
		\caption{$\Delta t = 0.5$}
	\end{subfigure}
	\begin{subfigure}{0.19\textwidth}
		\includegraphics[width=\textwidth]{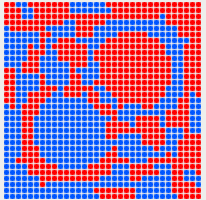}
		\caption{$\Delta t = 0.25$}
	\end{subfigure}
	\begin{subfigure}{0.19\textwidth}
		\includegraphics[width=\textwidth]{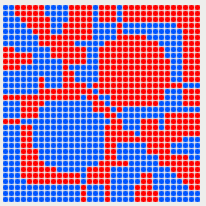}
		\caption{$\Delta t = 0.125$}
	\end{subfigure}
	\begin{subfigure}{0.19\textwidth}
		\includegraphics[width=\textwidth]{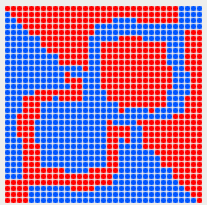}
		\caption{$\Delta t = 0.0625$}
	\end{subfigure}
	\caption{\small
Numerical solutions $u$ and $v$ obtained using FEM-FCT scheme with different number of time steps $\Delta t = 1.0, 0.5, 0.25, 0.125, 0.0625$ at final time $T= 500$ where $D_u=D_v=0.25$ and $\chi_u=\chi_v=3$. For the choice of colors, see the beginning of Subsection~\ref{sep-study}.
	 }
	\label{fig9a}
\end{figure} 

\begin{figure}[H]
	\centering
	\begin{subfigure}{0.275\textwidth}
		\includegraphics[width=\textwidth]{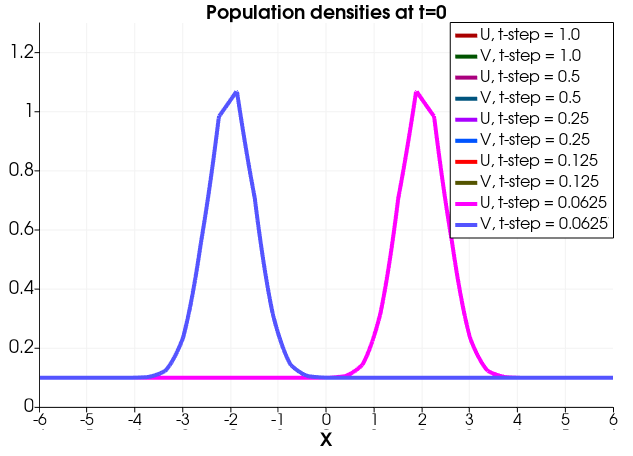}
		\caption{ $t=0$}
	\end{subfigure}
	\begin{subfigure}{0.275\textwidth}
		\includegraphics[width=\textwidth]{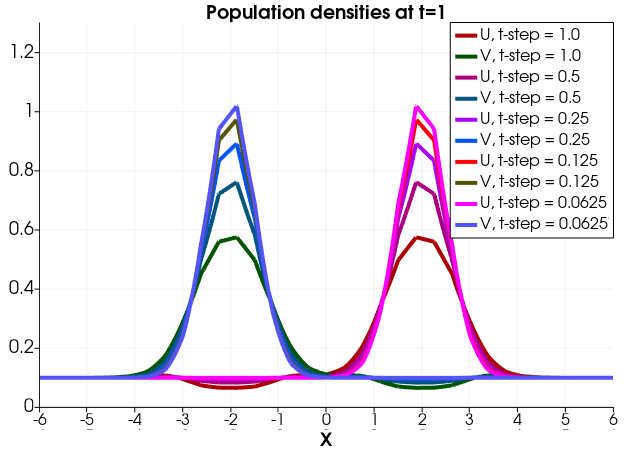}
		\caption{$t=1.0$}
	\end{subfigure}
	\begin{subfigure}{0.275\textwidth}
		\includegraphics[width=\textwidth]{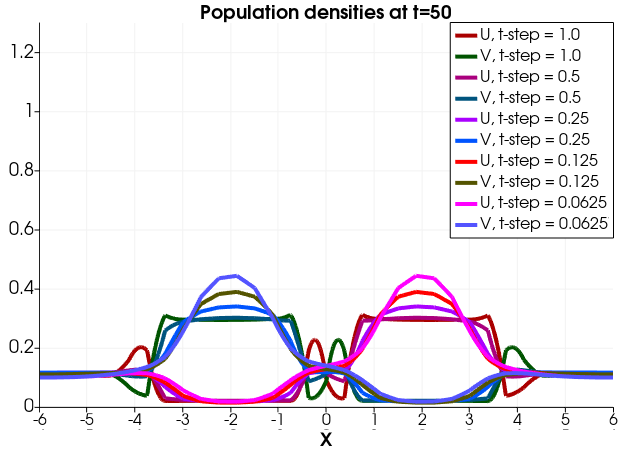}
		\caption{$t=50$}
	\end{subfigure}
	\begin{subfigure}{0.275\textwidth}
		\includegraphics[width=\textwidth]{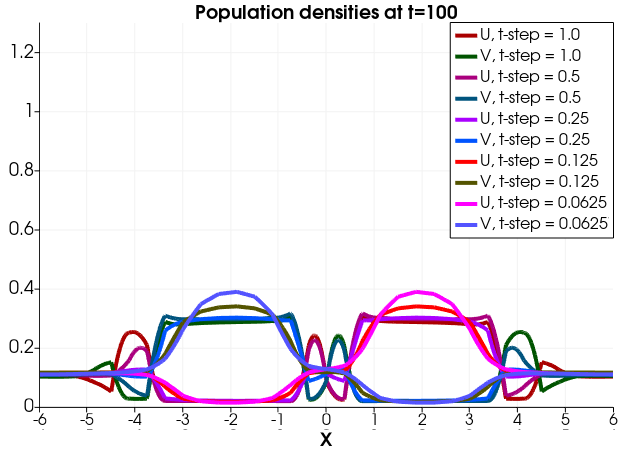}
		\caption{$t= 100$}
	\end{subfigure}
	\begin{subfigure}{0.275\textwidth}
		\includegraphics[width=\textwidth]{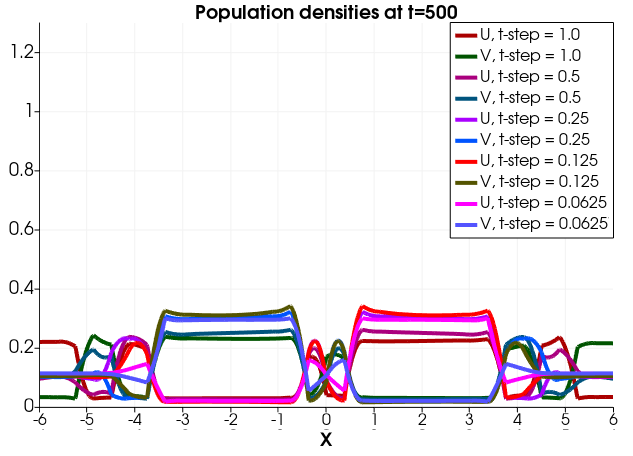}
		\caption{$t=500$}
	\end{subfigure}
	\caption{\small
	The amount of gangs concentration $u$, $v$ and graffiti densities $z$ and $w$ along the line $y=x$ with different time  steps at different time instants $ t= 0, 1, 50, 100, 500$ where $D_u=D_v=0.25$ and $\chi_u=\chi_v=3$.
	 }
	\label{fig9b}
\end{figure} 

We conclude that these convergence studies provide strong indication
that our numerical experiments are reliable
and that hence indeed both partial and complete separation may occur in \eqref{prob}.

\section*{Acknowledgments}
The work of Shahin Heydari was supported through the grant SVV-2023-260711 of Charles University.


\end{document}